\documentclass[a4paper,11pt]{amsart}
\title{Bounded deviations in higher genus I: closed geodesics}
\date{\today}

\author{Pierre-Antoine Guihéneuf}
\address{Pierre-Antoine Guihéneuf: Sorbonne Université, Université Paris Cité, CNRS, IMJ-PRG, F-75005 Paris, France --- 
IRL Jean-Christophe Yoccoz CNRS / IMPA, Estr. Dona Castorina, 110
Jardim Botânico, Rio de Janeiro, Brasil}
\email{pierre-antoine.guiheneuf@imj-prg.fr}
\thanks{P.-A.\ G.\ thanks the Jean-Christophe Yoccoz international laboratory CNRS/IMPA for the semester in Brazil during which the ideas of this work were born.}

\author{F\'abio Armando Tal}
\address{F\'abio Armando Tal: Instituto de Matem\'atica e Estat\'istica da Universidade de S\~ao Paulo, R. do Mat\~ao, 1010
- Vila Universitaria, S\~ao Paulo, Brasil}
\email{fabiotal@ime.usp.br}

\usepackage[utf8]{inputenc}
\usepackage[english]{babel}
\usepackage[T1]{fontenc}  
\usepackage{amsfonts}
\usepackage{amsmath}
\usepackage{amssymb}
\usepackage{amsthm}

\DeclareFontFamily{U}{mathx}{}
\DeclareFontShape{U}{mathx}{m}{n}{<-> mathx10}{} 
\DeclareSymbolFont{mathx}{U}{mathx}{m}{	n}
\DeclareMathAccent{\widehat}{0}{mathx}{"70}
\DeclareMathAccent{\widecheck}{0}{mathx}{"71}

\usepackage{enumitem}
\setlist{noitemsep}
\usepackage{tikz}
\usepackage[pdftex,colorlinks=true,linkcolor=blue,citecolor=blue,urlcolor=blue]{hyperref}

\newtheorem{lemma}{Lemma}[section]
\newtheorem{theorem}[lemma]{Theorem}
\newtheorem{conjecture}[lemma]{Conjecture}
\newtheorem{theo}{Theorem}
\newtheorem{corollary}[theo]{Corollary}

\newtheorem{prop}[lemma]{Proposition}
\newtheorem{coro}[lemma]{Corollary}

\theoremstyle{definition}
\newtheorem{definition}[lemma]{Definition}
\theoremstyle{remark}
\newtheorem{rem}[lemma]{Remark}

\newcommand{\B}{\mathcal{B}}

\newcommand{\Sp}{\mathbf{S}}

\newcommand{\F}{\mathcal{F}}

\newcommand{\Hy}{\mathbf{H}}
\newcommand{\N}{\mathbf{N}}

\newcommand{\R}{\mathbf{R}}

\newcommand{\T}{\mathbf{T}}

\newcommand{\G}{\mathcal{G}}

\newcommand{\Z}{\mathbf{Z}}
\newcommand{\M}{\mathcal{M}^\textrm{erg}_{\vartheta>0}}

\newcommand{\Homeo}{\operatorname{Homeo}}
\newcommand{\supp}{\operatorname{supp}}

\newcommand{\Fix}{\operatorname{Fix}}
\newcommand{\rot}{\operatorname{rot}}
\newcommand{\rote}{\operatorname{rot}_{\mathrm{erg}}}

\newcommand{\diam}{\operatorname{diam}}
\newcommand{\inte}{\operatorname{int}}
\newcommand{\dom}{\operatorname{dom}}
\newcommand{\Id}{\operatorname{Id}}

\newcommand{\cl}{\mathcal{N}}
\newcommand{\wt}{\widetilde}
\newcommand{\wh}{\widehat}
\newcommand{\wc}{\widecheck}

\newcommand{\Me}{\mathcal{M}^{\mathrm{erg}}}
\newcommand{\Merg}{\mathcal{M}^{\mathrm{erg}}_{\vartheta>0}}

\begin{document}

\maketitle

\begin{abstract}
This is the first article of a series of two where we study the problem of bounded deviations for homeomorphisms of closed surfaces of genus $\ge 2$. This first part studies bounded deviations with respect to closed geodesics. As a byproduct of our proofs, we also get a criterion of existence of periodic orbits in terms of big deviation with respect to some closed geodesic.
The combination with the second part \cite{paper2PAF} generalises to the higher genus case most of the bounded deviations results already known for the torus.
\end{abstract}

\tableofcontents

\section{Introduction}

The concept of \emph{bounded deviations} plays an increasingly central role in rotation theory and in the study of surface homeomorphisms. Roughly speaking, it measures how well certain dynamical invariants capture the displacement of orbits up to sublinear errors. To illustrate, consider a lift $\widetilde f:\R\to\R$ of a circle endomorphism of degree one, with rotation interval $[a,b]$. Then a simple but interesting result states that for every $\widetilde x\in\R$ and every positive integer $n$, one has
\[
-1+an \;\leq\; \widetilde f^n(\widetilde x)-\widetilde x \;\leq\; 1+bn.
\]
This inequality shows that the rotation interval provides sharp linear control of the long-term displacement, with deviations bounded by a uniform constant.

Already in the case of the torus $\T^2$, the picture becomes richer and more subtle. The notion of bounded deviations has been intensively studied there and is closely tied to the geometry of the rotation set. The main results for torus homeomorphisms will be recalled in the next subsection. 
However, it is already possible to see that beyond their intrinsic interest, bounded deviations statements have proven to be a powerful method in applications, as well as a fundamental tool in the development of rotation theory on the torus and the annulus. For example, for both the torus and the closed annulus case, bounded deviations were the key point in the proofs of the strong form of Boyland’s conjecture (\cite{lct1,zbMATH06425076,conejerostal2}). Also, it is used as a criterion for semi-conjugacy results of torus homeomorphisms to a circle rotation \cite{zbMATH06782951, zbMATH07391928} or a torus rotation \cite{zbMATH05574011}. 

Still for $\T^2$, in the absence of bounded deviations, the dynamics is called ``fully essential'' \cite{zbMATH06294042} and has a lot of nice features (see also \cite{zbMATH06908424} for the higher genus case). 

Finally, very recently, there has been an increasing interest in understanding the action induced by surface homeomorphisms on the fine curve graph and its classification into hyperbolic, parabolic or elliptic action (defined in \cite{zbMATH07420183}), a topic that has several connections with bounded deviations. For instance, a torus homeomorphism homotopic to the identity acts elliptically on the fine curve graph (defined in \cite{zbMATH07420183}) if and only if it has bounded deviations in some rational direction \cite{guiheneuf2023parabolicisometriesfinecurve}.


\bigskip

In this series of two papers we attempt to provide a comprehensive study of bounded deviations for homeomorphisms of closed surfaces of genus $g\ge 2$. While this topic is quite well understood for torus homeomorphisms, up to now the results in higher genus were rather partial.

In Part I we are interested in bounded deviations with respect to closed geodesics (which is a higher genus equivalent of the case of torus homeomorphisms whose rotation set is a segment with rational slope) while Part II deals with bounded deviations with respect to simple non-closed geodesics (which is a higher genus equivalent of the case of torus homeomorphisms whose rotation set is a segment with irrational slope).

Both parts are completely independent; they share the preliminary section which is essentially made of non-new results.

\subsection*{Bounded deviations for torus homeomorphisms}

Let us review the results for the torus. For a surface $S$, denote $\Homeo_0(S)$ the set of homeomorphisms of $S$ that are homotopic (or, equivalently, isotopic, see \cite{zbMATH03221970}) to the identity. For $f\in\Homeo_0(\T^2)$, choose $\wt f\in\Homeo_0(\R^2)$ a lift of $f$ to the universal cover $\R^2$ of $\T^2$. We say that the homeomorphism $f$ \emph{has bounded deviations in the direction $v\in\R^2\setminus\{0\}$} if there exists $\rho\in\R^2$ and $C>0$ such that for any $\wt x\in\R^2$ and any $n\in\Z$, we have
\begin{equation}\label{EqBndDevTorus}
\big|\langle\wt f^n(\wt x) - \wt x - n\rho,\, v \rangle\big|\le C.
\end{equation}
Roughly speaking, if a homeomorphism has bounded deviations in some rational direction, then its dynamics looks like the one of an annulus homeomorphism. 

A lot of criteria for the existence of bounded deviations are expressed in terms of the rotation set of the homeomorphism. Given $f\in\Homeo_0(\T^2)$, the \emph{rotation set} of its lift $\wt f\in\Homeo_0(\R^2)$ is the set 
\[\rot(\wt f) = \Big\{\rho\in\R^2 \ \big|\ \exists (\wt x_k)\in(\R^2)^\N,\, n_k\underset{k\to+\infty}\nearrow +\infty : \rho = \frac{\wt f^{n_k}(\wt x_k)-\wt x_k}{n_k}\Big\}.\]
This is a conjugacy invariant in $\Homeo_0(\R^2)$ that contains all the asymptotic rotation speeds around the torus. This is a compact and convex subset of $\R^2$ \cite{zbMATH04084609}. 

Bounded deviations in some direction hold for torus homeomorphisms having a periodic point and whose rotation set is a nondegenerate line segment: this can be obtained as a combination of the case of the rotation set being a segment with rational slope \cite{zbMATH06304088, zbMATH06914177} and the case where it has irrational slope \cite{zbMATH07488214} (using \cite{lct1} to rule out a case). Moreover, the hypothesis of having a periodic point is unnecessary if we suppose that the homeomorphism is minimal \cite{zbMATH07548721}.

There is also a (different) notion of bounded deviation in the case of nonempty interior rotation sets (one can obtain in a few lines that a homeomorphism whose rotation set has nonempty interior cannot have bounded deviation in any direction). Such bounded deviations ``from $\rho(\wt f)$'' were obtained in \cite{zbMATH06425076} ($C^{1+\alpha}$ case) and \cite{lct1} (general case).
Yet another bounded deviation result holds for homeomorphisms having as a rotation set a segment with irrational slope and a rational endpoint ``in the opposite direction of the one of the segment'' \cite[Theorem 1.3]{zbMATH07867510}. 
There are also some bounded deviation results for torus homeomorphisms homotopic to Dehn twists \cite{zbMATH06296542}.

A collection of (counter) examples of torus homeomorphisms shows that this is more or less all one can hope for results of bounded deviations for the torus.
\bigskip




\section*{Known results in higher genus}

For higher genus, the only known results up to now hold under hypotheses of ``big rotation set'': The first of them states bounded deviations ``from the rotation set'' for $C^{1+\alpha}$ diffeomorphisms of surfaces $S$ of genus $g\ge 2$ under the so-called condition of \emph{fully essential system of curves} (which implies the fact that the homological rotation set has 0 in its interior) \cite{zbMATH07282570}. Later on, \cite{lellouch} also obtained bounded deviations ``from the rotation set'' under the weaker hypothesis that the homological rotation set has 0 in its interior (without regularity assumption). 

Still, a global picture of bounded deviations for surfaces of higher genus was lacking, and in particular the counterparts of the case of torus homeomorphisms having a segment as a rotation set. This is the goal of this series of two articles, the first one treating the counterpart of torus homeomorphisms having as rotation set a segment with rational slope, and the second one treating the irrational slope case. 

Note that, as a homeomorphism of a surface of genus $g\ge 2$ has at least one fixed point (this is a consequence of the Lefschetz fixed point theorem), one cannot hope to have higher genus counterparts of bounded deviation results that hold for torus homeomorphisms without periodic points.

\subsection*{Crossing lifts and tracking geodesics}

The condition \eqref{EqBndDevTorus} for bounded deviations does not adapt directly in higher genus. In this paper, we will consider a bounded deviation notion involving the crossing number with some closed geodesic on the surface.

Let $S$ be a closed surface (compact, connected, orientable, without boundary) of genus $g\ge 2$. We equip $S$ with a Riemannian metric $d$ of constant curvature $-1$. Let $\wt S$ be the universal cover of $S$; by the uniformisation theorem $\wt S$ is isometric to the hyperbolic plane $\Hy^2$ (with a metric we also denote by $d$). This universal cover 
has a boundary at infinity that we will denote by $\partial\wt S$; this boundary is homeomorphic to the circle $\Sp^1$. We also denote $\G$ the group of deck transformations of $\wt S$ (\emph{i.e.}~the set of lifts of $\Id_S$ to $\wt S$). Every homeomorphism $f\in\Homeo_0(S)$ has a preferred lift $\wt f\in\Homeo_0(\wt S)$ (the only one homotopic to $\Id_{\wt S}$); this lift commutes with elements of $\G$ and extends continuously to $\wt S\cup\partial \wt S$ with $\Id_{\partial\wt S}$. The compactification $\wt S\cup \partial\wt S$ will be equipped with a finite diameter distance (\emph{e.g.}~coming from the Euclidean distance on the unit disc in the Poincaré disc model).

\begin{definition}\label{DefCrossLifts}
Let $f\in\Homeo_0(S)$. We say that an orbit segment $y,\dots,f^{n_0}(y)$ of $S$ \emph{crosses $N$ different lifts} of some closed geodesic $\gamma$ of $S$ if there exist lifts $\wt y$ and $\wt\gamma$ of $y$ and $\gamma$ to $\wt S$, and $R_1,\dots, R_N\in \G$ some deck transformations such that the $R_i \wt\gamma$ are pairwise different lifts of $\gamma$ and such that for any $1\le i\le N$, the points $\wt y$ and $\wt f^{n_0}(\wt y)$ belong to two different connected components of the complement of $R_i\wt\gamma$.
\end{definition}

This is equivalent to asking that the points $\wt y$ and $\wt f^{n_0}(\wt y)$ are separated by all the geodesics $R_i\wt\gamma$, in other words that the minimal geometric intersection number between $\gamma$ and a curve homotopic (relative to endpoints) to $I^{[0,n_0]}(y)$ is at least $N$.

Note that, applied to the torus, this definition gives back the classical definition \eqref{EqBndDevTorus} of bounded deviation in a rational direction. 
\bigskip

The condition on the homeomorphism that will imply bounded deviation involves the notion of tracking geodesic introduced in \cite{alepablo}. Denote $\Me(f)$ the set of $f$-invariant ergodic Borel probability measures. 
The following is a direct consequence of Kingman's subadditive ergodic theorem \cite[Lemma~1.6]{alepablo}.

\begin{lemma}\label{LemErgoRotSpeed}
Let $\mu\in\Me(f)$. Then there exists a constant \(\vartheta_\mu\in\R_+\) --- called the \emph{rotation speed} of $\mu$ --- such that
\[\lim\limits_{n \to +\infty}\frac{1}{n}d\big(\wt z, \wt f^n(\wt z)\big) = \lim\limits_{n \to +\infty}\frac{1}{n}d\big(\wt z, \wt f^{-n}(\wt z)\big) = \vartheta_\mu,\]
for \(\mu\)-almost every point \(z \in S\).	
\end{lemma}

We denote by \(\M(f)\) the set of $\mu\in\Me(f)$ such that $\vartheta_\mu>0$. As usual, we will parametrise geodesics by arclength. Points that are typical for some ergodic measure of $\M(f)$ follow a so-called \emph{tracking geodesic} \cite[Theorem~B]{alepablo}.

\begin{theorem}\label{DefTrackGeod}
Let $\mu\in\M(f)$. Then $\mu$-a.e.~$z\in S$ admits a \emph{tracking geodesic} $\gamma$: for each lift \(\wt z\) of \(z\), there exists a lift $\wt \gamma$ of $\gamma$ such that:
\begin{equation}\label{eq:trackingequation}
\lim\limits_{n \to +\infty}\frac{1}{n}d\big(\wt f^n(\wt z), \wt \gamma(n \vartheta_\mu)\big) = \lim\limits_{n \to +\infty}\frac{1}{n}d\big(\wt f^{-n}(\wt z), \wt \gamma(-n \vartheta_\mu)\big) = 0.
\end{equation}	
\end{theorem}

The geodesic associated to a $\mu$-typical $z\in S$ will be denoted by $\gamma_z$ and the one associated to the lift $\wt z$ will be denoted $\wt\gamma_{\wt z}$, and parametrised such that $d(\wt z,\wt\gamma_{\wt z}) = d(\wt z,\wt\gamma_{\wt z}(0))$. 

Note that if a tracking geodesic of a $\mu$-typical point $z\in S$ is closed, then all tracking geodesics associated to $\mu$-typical points are equal to this tracking geodesic \cite[Theorem~D]{alepablo}.

\subsection*{Main result}

The following is the main theorem of Part I. It is a higher genus counterpart of D\'avalos' bounded deviations result for the torus \cite{zbMATH06914177}, where he proves that bounded deviations hold if a torus homeomorphism has a segment with rational slope and containing a rational point as a rotation set. 

\begin{theo}\label{ThBndedDevRat}
Let $f\in\Homeo_0(S)$, where $S$ is a closed surface of genus $g\ge 2$. Let $\gamma$ be a closed geodesic that is a tracking geodesic for some $\mu\in\M(f)$. Then there exists $N>0$ such that if an orbit $\wt y,\dots,\wt f^{n_0}(\wt y)$ of $\wt f$ crosses $N$ different lifts of $\gamma$, then there exists an $f$-periodic orbit with one lift to $\wt S$ having its tracking geodesic intersecting at least two of these lifts of $\gamma$.
\end{theo}

Note that, in particular, the conclusion implies that there is a periodic orbit whose tracking geodesic intersects $\gamma$. Note also that this theorem, and the other ones we will state, have their hypotheses stated in terms of \emph{ergodic} properties and not properties of ``Misiurewicz-Ziemian'' types of rotation sets.

We conjecture that a stronger statement should hold under the additional assumption that the set of fixed points is inessential:

\begin{conjecture}
Let $f\in\Homeo_0(S)$, where $S$ is a closed surface of genus $g\ge 2$. Suppose that the set of contractible fixed points of $f$ is inessential. Let $\gamma$ be a closed geodesic that is a tracking geodesic for some $\mu\in\M(f)$. Then there exists $C>0$ such that if an orbit $\wt y,\dots,\wt f^{n_0}(\wt y)$ of $\wt f$ crosses $V_C(\wt\gamma) = \{\wt x\in \wt S \mid d(\wt x,\wt\gamma)<R\}$ for some lift $\wt\gamma$ of $\gamma$, then there exists an $f$-periodic orbit with one lift to $\wt S$ having its tracking geodesic intersecting $\wt\gamma$.
\end{conjecture}

If such a result is true, its combination with \cite[Theorem~D]{paper2PAF} could allow to classify all surface homeomorphisms with inessential fixed point set and without rotational horseshoe, for example by associating to each measure class (defined in the next paragraph) an invariant open set together with a pseudo-lamination (possibly under non-wandering assumptions), expressing that the rotational dynamics mimics the one of a flow.

\subsection*{Tracking sets and consequences}

Theorem~\ref{DefTrackGeod} allows us to define a set of geodesics associated to an ergodic measure of $\M(f)$ \cite[Theorem~ C]{alepablo}. 

\begin{theorem}\label{theoExistTrackSet}
To any $\mu\in\M(f)$ is associated a set $\dot\Lambda_\mu\subset T^1 S$ that is invariant under the geodesic flow on $T^1 S$, and such that for $\mu$-a.e.~$z\in S$, we have 
\[\overline{\dot\gamma_z(\R)} = \dot\Lambda_\mu.\]
\end{theorem}

This allows us to define the \emph{tracking set} of $f$ as 
\[\rote^{\text{htopic}}(f) = \bigcup_{\mu\in\M(f)}\dot\Lambda_\mu.\]

This also allows us to define an equivalence relation $\sim$ on $\M(f)$ by: $\mu_1\sim \mu_2$ if one of the following is true:
\begin{itemize}
\item $\dot\Lambda_{\mu_1} = \dot\Lambda_{\mu_2}$;
\item There exist $\nu_1,\dots,\nu_m\in\M$ such that $\nu_1=\mu_1$, $\nu_m=\mu_2$ and for all $1\le i<m$, there exist two geodesics of $\dot\Lambda_{\nu_i}$ and $\dot\Lambda_{\nu_{i+1}}$ that intersect transversally. 
\end{itemize} 

We denote $(\cl_i)_i$ the classes of this equivalence relation. By \cite{alepablo} (Section 6.2, and in particular Lemmas 6.7 and 6.8), we have the following:
\begin{itemize}
\item The set of classes is of cardinal at most $5g-5$;
\item For any class $\cl_i$, the set $\Lambda_i:=\bigcup_{\mu\in\cl_i}\dot\Lambda_\mu$ is connected\footnote{Because it is either a single geodesic, or the closure of a single geodesic, or a path-connected set in the case of a chaotic class (defined just below).} (be careful, these sets need not be closed); 
\item To any class $\cl_i$ is associated a surface $S_i\subset S$ whose boundary is made of a finite collection of closed geodesics and minimal for inclusion among such surfaces such that for any $\mu\in \cl_i$ we have $\dot\Lambda_\mu \subset T^1 S_i$. If $\inte(S_i)\neq\emptyset$, then $S_i$ is open. Moreover, the surfaces $S_i$ are pairwise disjoint. 
\end{itemize}

\begin{definition}\label{DefClassMeas}
There are three types of classes: classes $\cl_i$ such that $\bigcup_{\mu\in\cl_i}\dot\Lambda_\mu$:
\begin{itemize}[nosep]
\item is a single closed geodesic are called \emph{closed} classes;
\item is a minimal lamination that is not a closed geodesic are called \emph{minimal non-closed} classes;
\item has transverse intersection are called \emph{chaotic} classes.
\end{itemize}
\end{definition}

Note that if $\cl_i$ is a closed class, then $S_i$ is a single closed geodesic, while for other classes the associated surface $S_i$ has nonempty interior. 
\bigskip

Theorem~\ref{ThBndedDevRat} implies Corollary~\ref{CoroBndedDevRat2}, which together with \cite[Corollary~C]{paper2PAF} implies the following result.

\begin{corollary}\label{CoroBndedDevRat}
Let $S$ be a compact boundaryless hyperbolic surface and $f\in \Homeo_0(S)$. Let $\gamma$ be a closed geodesic that is the boundary component of the surface associated to a class $\cl_i$.
Let $\wt f$ be the canonical lift of $f$ to the universal cover $\wt S$ of $S$. Then there exists $N>0$ such that an orbit of $\wt f$ cannot cross more than $N$ different lifts of $\gamma$.
\end{corollary}

This result implies the following statement about the fine curve graph (see \cite{zbMATH07420183, zbMATH07588408} for definitions, and \cite[Lemma~18]{guiheneuf2023hyperbolic} for a criterion saying that bounded deviations with respect to a closed geodesic implies elliptic action on $\mathcal C^\dagger(S)$).

\begin{corollary}
Let $S$ be a compact boundaryless hyperbolic surface and $f\in \Homeo_0(S)$. Suppose that $f$ has a class $\cl_i$ such that $S_i \neq S$. Then $f$ acts elliptically on $\mathcal C^\dagger(S)$. 
\end{corollary}

Hence, by \cite{guiheneuf2023hyperbolic}, the only cases where $f$ does not act elliptically is when $f$ has a class $\cl_i$ such that $S_i = S$ (\emph{i.e.}~either if $f$ has a single filling chaotic class, and in this case the ergodic homological rotation set of $f$ has nonempty interior by \cite{alepablo}, or if $f$ has a single filling minimal non-closed class) or when $f$ is irrotational, that is, when  the rotation speed of any $f$-invariant ergodic measure is null.

\subsection*{Other potential applications}

In a forthcoming series of two articles \cite{PABig1, PABig2}, the first author develops rotational hyperbolic theory for
surface homeomorphisms: the idea is to define a relation similar to heteroclinic connection for chaotic classes. A key tool in this study is one of our results of creation of periodic points (Corollary~\ref{CoroKiSert}). This theory is then applied to homeomorphisms whose rotation set spans the whole homology to get results about the shape of the rotation set, realization of subsets of the rotation set as rotation vectors of compact subsets, bounded deviations in homology from the rotation set, etc.

Our bounded deviation results may also be applied to prove the existence of invariant open sub-surfaces, \emph{e.g.}~a higher genus version of \cite{zbMATH06304088}.

\subsection*{Plan of the paper}

The proof of Theorem~\ref{ThBndedDevRat} is based on the forcing theory of Le Calvez and the second author \cite{lct1, lct2}. The preliminaries needed on the subject are developed in Section~\ref{SecPrelim}, the proof, which is here explained using the language of forcing theory, starts in Section~\ref{SecSpecClosed}. The goal of this section is to prove Proposition~\ref{LemRealizPeriod}, where we replace the orbit of a $\mu$-typical point $z\in S$ by the one of $z'\in S$ that is typical for $\mu'\in \M(f)$, having also $\gamma$ as a tracking geodesic, and such that the transverse trajectory $I^\Z_{\wt\F}(\wt z')$ of  one of its lifts $\wt z'\in \wt S$ is $\wt\F$-equivalent to a simple $T$-invariant\footnote{$T$ is a primitive deck transformation of $\wt S$ leaving $\wt\gamma$ invariant.} transverse path $\wt\alpha_0$ (and in particular is simple).

Section~\ref{sec:transverseintersections} deals with the rest of the proof of Theorem~\ref{ThBndedDevRat}: using repeatedly the traditional forcing lemma, we patch together pieces of the transverse trajectory of a lift of $y$ and of different copies of $\wt \alpha_0$ to obtain an admissible transverse path $\wt\beta$ and a deck transformation $T_1$ such that $\wt \beta$ and $T_1\wt \beta$ have an $\wt\F$-transverse intersection, and such that the axis of $T_1$ intersects the one of $T$. 

The whole section is split into two very different cases from the technical point: When considering the set $\wt\B$ of leaves met by $\wt\alpha_0$ and some deck transformation $R$, one cannot assume that $\wt \B$ is disjoint from $R\wt \B$. The trajectory of $\wt y$ will meet different copies of $\wt\B$, but we divide the proof into whether the trajectory of $\wt y$ stays in these different copies (Proposition~\ref{PropBndedDevRatCase2}), or if it crosses these copies of this set $\wt \B$ (Propositions~\ref{PropBndedDevRatCase1a} and \ref{PropBndedDevRatCase1}).

\section{Preliminaries on forcing theory}\label{SecPrelim}

Let us start with two results independent from forcing theory. 
The following lemma is a direct consequence of \cite[Lemma 3.1]{MR837985}. It implies that a loop of the annulus winding twice or more around it cannot be simple.

\begin{lemma}\label{LemBrown}
Let $v\in\R^2\setminus\{0\}$ and $K\subset \R^2$ be an arcwise connected set.
If $K\cap (K+iv)\neq\emptyset$ for some $i\in\Z\setminus\{0\}$, then  $K\cap (K+v) \neq \emptyset$. 
\end{lemma}

The following lemma expresses that if a path of $\wt S$ intersects a certain amount of lifts of a fixed closed geodesic, then up to considering a fixed fraction of these lifts, one can suppose that they are pairwise disjoint.

\begin{lemma}\label{LemUseResidFinite}
Let $\gamma$ be a closed geodesic on $S$. Then for any $M_0>0$ and any $R>0$, there exists $N_0\in\N$ such that for any path $\alpha : [0,1]\to S$ whose geometric intersection number with $\gamma$ is bigger than $N_0$, any lift $\wt\alpha$ of $\alpha$ to $\wt S$ crosses geometrically $M_0$ lifts of $\gamma$ that are pairwise disjoint, have the same orientation and are pairwise at a distance $\ge R$.
\end{lemma}

\begin{proof}
It is a classical result that there exists a finite cover $\check S$ of $S$ on which the lifts $\check\gamma_1,\dots,\check\gamma_k$ of $\gamma$ are simple (it is a consequence of the facts that $\pi_1(S)$ is residually finite and that finitely generated subgroups of $\pi_1(S)$ are separable \cite{ResidFinit, ResidFinit2}). 

For $1\le i\le k$, define $d_i>0$ as the minimum distance between two different lifts of $\check\gamma_i$ to $\wt S$. Let $d = \max_{1\le i\le k} d_i$.  

Let $M_0>0$ and suppose that the geometric intersection number of $\alpha$ and $\gamma$ is bigger than $2kM_0\lceil R/d\rceil$. By the pigeonhole principle, this implies that there exists $1\le i_0\le k$ such that a lift $\wt\alpha$ of $\alpha$ to $\wt S$ crosses geometrically at least $2M_0\lceil R/d\rceil$ lifts of $\check\gamma_{i_0}$. Because they lift a simple geodesic, these lifts are pairwise disjoint. Moreover, at least $M_0\lceil R/d\rceil$ of them have the same orientation. 

As these geodesics are ordered and pairwise at a distance $\ge d_{i_0}$, at least $M_0$ of them are pairwise at a distance $\ge R$. These are the geodesics that satisfy the conclusion of the lemma.
\end{proof}

\paragraph{Foliations and isotopies.}

Given an isotopy $I = \{f_t\}_{t\in[0,1]}$ from the identity to $f$, its \emph{fixed point set} is $\Fix(I) = \bigcap_{t\in[0,1]} \Fix(f_t)$, and its \emph{domain} is $\dom(I) := S \backslash \Fix(I)$. Note that $\dom(I)$ is an oriented boundaryless surface, not necessarily closed, not necessarily connected, not necessarily of finite type.	

In this section we will consider an oriented surface $\Sigma$ without boundary, not necessarily closed or connected (with the idea to apply it to $\textnormal{dom}(I)$), and a non singular oriented topological foliation $\F$ on $\Sigma$. We will denote
$\wh\Sigma$ the universal covering space of $\Sigma$, $\wh\pi: \wh\Sigma\to \Sigma$ the covering projection and $\wh{\F}$ the lift of $\F$ to $\wh{\Sigma}$.

For every $x\in\Sigma$, we denote $\phi_{x}$ the leaf of $\F$ that contains $x$. 
The complement of any simple injective proper path $\wh\gamma$ of $\wh\Sigma$ inside the connected component  of $\wh\Sigma$ containing $\wh\gamma$ has two connected components, that we denote $L(\wh\gamma)$ and $R(\wh\gamma)$, chosen according to some fixed orientation of $\wh\Sigma$ and the orientation of $\wh\gamma$. 
Given a simple injective oriented proper path $\wh\gamma$ of $\wh\Sigma$ and $\wh x\in \wh\gamma$, we denote $\wh\gamma^+_{\wh x}$ and $\wh\gamma^-_{\wh x}$ the connected components of $\wh\gamma\setminus\{\wh x\}$, chosen accordingly to the orientation of $\wh\gamma$; their respective projections on $\Sigma$ are denoted respectively $\gamma^+_{x}$ and $\gamma^-_{x}$.

\paragraph{$\F$-transverse paths and $\F$-transverse intersections.}

A path $\alpha:J\to\Sigma$ is called \emph{positively transverse}\footnote{In the sequel, ``transverse'' will mean ``positively transverse''.} to $\F$ if it locally crosses each leaf of $\F$ from left to right. 
Note that every lift $\wh\alpha:J\to\wh\Sigma$ of a positively transverse path $\alpha$ is positively transverse to $\wh{\F}$. Moreover, asking that $\alpha$ is transverse is equivalent to requiring that for any lift $\wh\alpha$ and for every $a<b$ in $J$, the path $\wh\alpha|_{[a,b]}$ meets once every leaf $\wh\phi$ of $\wh \F$ such that $L(\wh\phi_{\wh \alpha(a)})\subset L(\wh\phi)\subset L(\wh\phi_{\wh \alpha(b)})$, and $\wh\alpha|_{[a,b]}$ does not meet any other leaf.
We will say that two transverse paths $\wh\alpha_1:J_1\to\wh\Sigma$ and $\wh\alpha_2:J_2\to\wh\Sigma$ are \emph{$\wh\F$-equivalent} if they meet the same leaves of $\wh{\F}$. Two transverse paths $\alpha_1:J_1\to\Sigma$ and $\alpha_2:J_2\to\Sigma$ are said to be \emph{$\F$-equivalent} if they have lifts to $\wh \Sigma$ that are $\wh\F$-equivalent. When it is clear from the context, we will say that the paths are equivalent and not $\F$-equivalent.

\begin{figure}
\begin{center}

\tikzset{every picture/.style={line width=0.6pt}} 

\begin{tikzpicture}[x=0.75pt,y=0.75pt,yscale=-1,xscale=1]

\draw [color={rgb, 255:red, 245; green, 166; blue, 35 }  ,draw opacity=1 ]   (261.43,50.73) .. controls (259.25,85.76) and (271.02,112.03) .. (273.64,145.74) .. controls (276.25,179.45) and (300.23,213.16) .. (286.72,229.8) ;
\draw [shift={(266.37,101.91)}, rotate = 78.81] [fill={rgb, 255:red, 245; green, 166; blue, 35 }  ,fill opacity=1 ][line width=0.08]  [draw opacity=0] (8.04,-3.86) -- (0,0) -- (8.04,3.86) -- (5.34,0) -- cycle    ;
\draw [shift={(285.23,191.05)}, rotate = 71.75] [fill={rgb, 255:red, 245; green, 166; blue, 35 }  ,fill opacity=1 ][line width=0.08]  [draw opacity=0] (8.04,-3.86) -- (0,0) -- (8.04,3.86) -- (5.34,0) -- cycle    ;
\draw [color={rgb, 255:red, 245; green, 166; blue, 35 }  ,draw opacity=1 ]   (156.35,52.92) .. controls (164.64,69.56) and (165.07,76.12) .. (155.92,84.88) .. controls (146.76,93.64) and (130.19,100.2) .. (120.6,95.83) ;
\draw [shift={(162.66,72.57)}, rotate = 103] [fill={rgb, 255:red, 245; green, 166; blue, 35 }  ,fill opacity=1 ][line width=0.08]  [draw opacity=0] (8.04,-3.86) -- (0,0) -- (8.04,3.86) -- (5.34,0) -- cycle    ;
\draw [shift={(136.28,95.94)}, rotate = 172] [fill={rgb, 255:red, 245; green, 166; blue, 35 }  ,fill opacity=1 ][line width=0.08]  [draw opacity=0] (8.04,-3.86) -- (0,0) -- (8.04,3.86) -- (5.34,0) -- cycle    ;
\draw [color={rgb, 255:red, 245; green, 166; blue, 35 }  ,draw opacity=1 ]   (391,50.6) .. controls (378.79,61.98) and (385.69,73.93) .. (392.23,82.25) .. controls (398.77,90.57) and (420.57,95.39) .. (430.6,85.32) ;
\draw [shift={(385.1,69.68)}, rotate = 66] [fill={rgb, 255:red, 245; green, 166; blue, 35 }  ,fill opacity=1 ][line width=0.08]  [draw opacity=0] (8.04,-3.86) -- (0,0) -- (8.04,3.86) -- (5.34,0) -- cycle    ;
\draw [shift={(414.19,90.98)}, rotate = -5] [fill={rgb, 255:red, 245; green, 166; blue, 35 }  ,fill opacity=1 ][line width=0.08]  [draw opacity=0] (8.04,-3.86) -- (0,0) -- (8.04,3.86) -- (5.34,0) -- cycle    ;
\draw [color={rgb, 255:red, 245; green, 166; blue, 35 }  ,draw opacity=1 ]   (209.98,50.29) .. controls (206.49,112.03) and (133.24,101.96) .. (135.86,135.67) .. controls (138.48,169.38) and (228.73,165.88) .. (216.09,224.11) ;
\draw [shift={(176.16,100.54)}, rotate = 149.51] [fill={rgb, 255:red, 245; green, 166; blue, 35 }  ,fill opacity=1 ][line width=0.08]  [draw opacity=0] (8.04,-3.86) -- (0,0) -- (8.04,3.86) -- (5.34,0) -- cycle    ;
\draw [shift={(189.24,174.67)}, rotate = 28.12] [fill={rgb, 255:red, 245; green, 166; blue, 35 }  ,fill opacity=1 ][line width=0.08]  [draw opacity=0] (8.04,-3.86) -- (0,0) -- (8.04,3.86) -- (5.34,0) -- cycle    ;
\draw [color={rgb, 255:red, 245; green, 166; blue, 35 }  ,draw opacity=1 ]   (322.47,48.1) .. controls (322.47,103.27) and (410.46,94.79) .. (412.2,125) .. controls (413.94,155.21) and (323.34,172.01) .. (347.76,227.17) ;
\draw [shift={(364.02,94.87)}, rotate = 22.65] [fill={rgb, 255:red, 245; green, 166; blue, 35 }  ,fill opacity=1 ][line width=0.08]  [draw opacity=0] (8.04,-3.86) -- (0,0) -- (8.04,3.86) -- (5.34,0) -- cycle    ;
\draw [shift={(364.67,172.23)}, rotate = 140.57] [fill={rgb, 255:red, 245; green, 166; blue, 35 }  ,fill opacity=1 ][line width=0.08]  [draw opacity=0] (8.04,-3.86) -- (0,0) -- (8.04,3.86) -- (5.34,0) -- cycle    ;
\draw [color={rgb, 255:red, 84; green, 0; blue, 213 }  ,draw opacity=1 ][line width=1.5]    (155.92,84.88) .. controls (175.97,129.54) and (235.71,91.89) .. (273.64,145.74) .. controls (311.57,199.59) and (356.04,167.19) .. (381.77,186.02) ;
\draw [shift={(219.54,113.55)}, rotate = 188.99] [fill={rgb, 255:red, 84; green, 0; blue, 213 }  ,fill opacity=1 ][line width=0.08]  [draw opacity=0] (8.75,-4.2) -- (0,0) -- (8.75,4.2) -- (5.81,0) -- cycle    ;
\draw [shift={(327.1,178.63)}, rotate = 187.51] [fill={rgb, 255:red, 84; green, 0; blue, 213 }  ,fill opacity=1 ][line width=0.08]  [draw opacity=0] (8.75,-4.2) -- (0,0) -- (8.75,4.2) -- (5.81,0) -- cycle    ;
\draw [color={rgb, 255:red, 0; green, 116; blue, 201 }  ,draw opacity=1 ][line width=1.5]    (151.12,196.96) .. controls (196.46,209.66) and (233.96,202.65) .. (273.64,145.74) .. controls (313.31,88.82) and (366.07,123.85) .. (392.23,82.25) ;
\draw [shift={(225.84,192.34)}, rotate = 153.5] [fill={rgb, 255:red, 0; green, 116; blue, 201 }  ,fill opacity=1 ][line width=0.08]  [draw opacity=0] (8.75,-4.2) -- (0,0) -- (8.75,4.2) -- (5.81,0) -- cycle    ;
\draw [shift={(335.43,108.69)}, rotate = 169.41] [fill={rgb, 255:red, 0; green, 116; blue, 201 }  ,fill opacity=1 ][line width=0.08]  [draw opacity=0] (8.75,-4.2) -- (0,0) -- (8.75,4.2) -- (5.81,0) -- cycle    ;
\draw [color={rgb, 255:red, 245; green, 166; blue, 35 }  ,draw opacity=1 ]   (121.47,179.01) .. controls (134.99,171.57) and (152.43,182.51) .. (151.12,196.96) .. controls (149.81,211.41) and (137.17,224.98) .. (125.4,224.55) ;
\draw [shift={(144.21,181.45)}, rotate = 45.25] [fill={rgb, 255:red, 245; green, 166; blue, 35 }  ,fill opacity=1 ][line width=0.08]  [draw opacity=0] (8.04,-3.86) -- (0,0) -- (8.04,3.86) -- (5.34,0) -- cycle    ;
\draw [shift={(140.71,217.82)}, rotate = 141.55] [fill={rgb, 255:red, 245; green, 166; blue, 35 }  ,fill opacity=1 ][line width=0.08]  [draw opacity=0] (8.04,-3.86) -- (0,0) -- (8.04,3.86) -- (5.34,0) -- cycle    ;
\draw [color={rgb, 255:red, 245; green, 166; blue, 35 }  ,draw opacity=1 ]   (428.42,198.71) .. controls (412.72,179.45) and (387.87,180.76) .. (381.77,186.02) .. controls (375.66,191.27) and (355.61,205.72) .. (373.92,221.48) ;
\draw [shift={(403.94,183.63)}, rotate = 185.09] [fill={rgb, 255:red, 245; green, 166; blue, 35 }  ,fill opacity=1 ][line width=0.08]  [draw opacity=0] (8.04,-3.86) -- (0,0) -- (8.04,3.86) -- (5.34,0) -- cycle    ;
\draw [shift={(366.62,205.34)}, rotate = 88.52] [fill={rgb, 255:red, 245; green, 166; blue, 35 }  ,fill opacity=1 ][line width=0.08]  [draw opacity=0] (8.04,-3.86) -- (0,0) -- (8.04,3.86) -- (5.34,0) -- cycle    ;
\draw  [fill={rgb, 255:red, 0; green, 0; blue, 0 }  ,fill opacity=1 ] (270.32,145.74) .. controls (270.32,143.9) and (271.81,142.41) .. (273.64,142.41) .. controls (275.47,142.41) and (276.95,143.9) .. (276.95,145.74) .. controls (276.95,147.58) and (275.47,149.07) .. (273.64,149.07) .. controls (271.81,149.07) and (270.32,147.58) .. (270.32,145.74) -- cycle ;
\draw  [fill={rgb, 255:red, 0; green, 0; blue, 0 }  ,fill opacity=1 ] (388.92,82.25) .. controls (388.92,80.41) and (390.4,78.92) .. (392.23,78.92) .. controls (394.06,78.92) and (395.55,80.41) .. (395.55,82.25) .. controls (395.55,84.09) and (394.06,85.58) .. (392.23,85.58) .. controls (390.4,85.58) and (388.92,84.09) .. (388.92,82.25) -- cycle ;
\draw  [fill={rgb, 255:red, 0; green, 0; blue, 0 }  ,fill opacity=1 ] (378.45,186.02) .. controls (378.45,184.18) and (379.94,182.69) .. (381.77,182.69) .. controls (383.6,182.69) and (385.08,184.18) .. (385.08,186.02) .. controls (385.08,187.86) and (383.6,189.35) .. (381.77,189.35) .. controls (379.94,189.35) and (378.45,187.86) .. (378.45,186.02) -- cycle ;
\draw  [fill={rgb, 255:red, 0; green, 0; blue, 0 }  ,fill opacity=1 ] (147.8,196.96) .. controls (147.8,195.12) and (149.29,193.63) .. (151.12,193.63) .. controls (152.95,193.63) and (154.44,195.12) .. (154.44,196.96) .. controls (154.44,198.8) and (152.95,200.29) .. (151.12,200.29) .. controls (149.29,200.29) and (147.8,198.8) .. (147.8,196.96) -- cycle ;
\draw  [fill={rgb, 255:red, 0; green, 0; blue, 0 }  ,fill opacity=1 ] (152.6,84.88) .. controls (152.6,83.04) and (154.09,81.55) .. (155.92,81.55) .. controls (157.75,81.55) and (159.23,83.04) .. (159.23,84.88) .. controls (159.23,86.72) and (157.75,88.21) .. (155.92,88.21) .. controls (154.09,88.21) and (152.6,86.72) .. (152.6,84.88) -- cycle ;

\draw (284.07,145.38) node [anchor=west] [inner sep=0.75pt]  [font=\small]  {$\wh{\alpha }_{1}( t_{1}) =\wh{\alpha }_{2}( t_{2})$};
\draw (205.68,114.4) node [anchor=north] [inner sep=0.75pt]  [color={rgb, 255:red, 84; green, 0; blue, 213 }  ,opacity=1 ]  {$\wh{\alpha }_{1}$};
\draw (244.15,176.67) node [anchor=south east] [inner sep=0.75pt]  [color={rgb, 255:red, 0; green, 116; blue, 201 }  ,opacity=1 ]  {$\wh{\alpha }_{2}$};
\draw (394.23,78.85) node [anchor=south west] [inner sep=0.75pt]  [font=\small]  {$\wh{\alpha }_{2}( b_{2})$};
\draw (383.77,189.42) node [anchor=north west][inner sep=0.75pt]  [font=\small]  {$\wh{\alpha }_{1}( b_{1})$};
\draw (145.8,196.96) node [anchor=east] [inner sep=0.75pt]  [font=\small]  {$\wh{\alpha }_{2}( a_{2})$};
\draw (153.92,81.48) node [anchor=south east] [inner sep=0.75pt]  [font=\small]  {$\wh{\alpha }_{1}( a_{1})$};

\end{tikzpicture}

\caption{Example of $\wh{\F}$-transverse intersection.\label{Fig:extransverse}}
\end{center}
\end{figure}
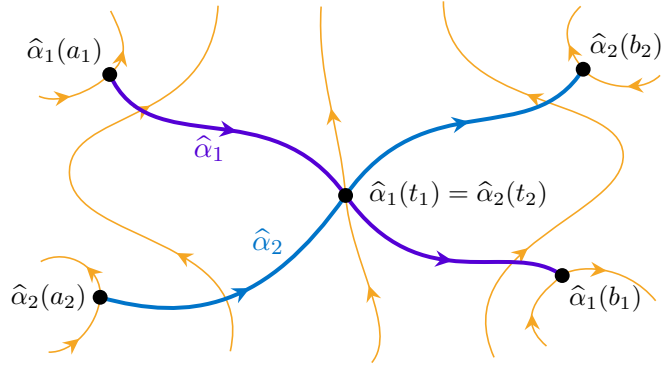

\begin{definition}\label{DefInterTrans}
Let $\wh \phi_1$, $\wh\phi_2$ and $\wh\phi_3$ be three leaves of $\wh \F$. We say that $\wh\phi_1$ \emph{is above $\wh \phi_2$ relative to $\wh\phi_3$} if there exist disjoint paths $\wh \delta_1$ and $\wh\delta_2$ linking $\wh\phi_1$, respectively $\wh\phi_2$, to $\wh\phi_3$, disjoint from these leaves but at their extremities, and such that $\wh\delta_1\cap\wh\phi_3$ is after $\wh\delta_2\cap\wh\phi_3$ for the order on $\wh\phi_3$.

Let $\wh\alpha_1:J_1\to \wh\Sigma$ and $\wh\alpha_2:J_2\to \wh \Sigma$ be two transverse paths such that there exist $t_1\in J_1$ and $t_2\in J_2$ satisfying $\wh\alpha_1(t_1)=\wh\alpha_2(t_2)$. We say that $\wh\alpha_1$ and $\wh\alpha_2$ have an \emph{$\wh{\F}$-transverse intersection} at $\wh\alpha_1(t_1)=\wh\alpha_2(t_2)$ (see Figure~\ref{Fig:extransverse}) 
if there exist $a_1, b_1\in J_1$ with $a_1<t_1<b_1$, and $a_2, b_2\in J_2$ with $a_2<t_2<b_2$, such that
$\wh\phi_{\wh\alpha_1(a_1)}$ is above $\wh\phi_{\wh\alpha_2(a_2)}$ relative to $\wh\phi_{\wh\alpha_2(t_2)}$, and $\wh\phi_{\wh\alpha_1(b_1)}$ is below $\wh\phi_{\wh\alpha_2(b_2)}$ relative to $\wh\phi_{\wh\alpha_2(t_2)}$.
\end{definition}

A transverse intersection means that there is a ``crossing'' between the two paths naturally defined by $\hat\alpha_1$ and $\hat\alpha_2$ in the space of leaves of $\widehat{\F}$, which is a one-dimensional topological manifold, usually non Hausdorff.

\bigskip
We say that two transverse paths $\alpha_1$ and $\alpha_2$ of $\Sigma$ have an \emph{$\F$-transverse intersection} if they have lifts to $\wh\Sigma$ having an $\wh\F$-transverse intersection.
If $\alpha_1=\alpha_2$ one speaks of a \emph{$\F$-transverse self-intersection}. In this case, if $\widehat \alpha_1$ is a lift of $\alpha_1$, then there exists a deck transformation $T\in\mathcal G$ such that $\widehat\alpha_1$ and $T\widehat\alpha_1$ have an $\widehat{\F}$-transverse intersection.

\paragraph{Recurrence and equivalence for $\F$.}

We will say a transverse path $\alpha:\R\to \Sigma$ is \emph{positively recurrent} (resp. \emph{negatively recurrent}) if, for every $a<b$, there exist $c<d$, with $b<c$ (resp. with $d<a$), such that $\alpha|_{[a,b]}$ and $\alpha|_{[c,d]}$  are equivalent.
Finally, $\alpha$ is \emph{recurrent} if it is both positively and negatively recurrent.

Two transverse paths $\alpha_1:\R\to \Sigma$ and $\alpha_2:\R\to \Sigma$ are said to be \emph{equivalent at $+\infty$} (denoted $\alpha_1\sim_{+\infty}\alpha_2$) if there exist $a_1,a_2\in \R$ such that $\alpha_1{}|_{[a_1,+\infty)}$ and  $\alpha_2{}|_{[a_2,+\infty)}$ are equivalent. Similarly $\alpha_1$ and $\alpha_2$ are said \emph{equivalent at $-\infty$} (denoted $\alpha_1\sim_{-\infty}\alpha_2$) if there exist $b_1,b_2\in \R$ such that $\alpha_1{}|_{(-\infty,b_1]}$ and  $\alpha_2{}|_{(-\infty,b_2]}$ are equivalent.

\subsection{Accumulation property}

We say that a transverse path $\alpha_1:\R\to \Sigma$ \emph{accumulates positively} on the transverse path $\alpha_2:\R\to \Sigma$ if there exist real numbers $a_1$ and $a_2<b_2$ such that  $\alpha_1{}|_{[a_1,+\infty)}$ and $\alpha_2{}|_{[a_2,b_2)}$ are $\F$-equivalent. Similarly, $\alpha_1$ \emph{accumulates negatively} on $\alpha_2$ if there exist real numbers $b_1$ and $a_2<b_2$ such that  $\alpha_1{}|_{(-\infty,b_1]}$ and $\alpha_2{}|_{(a_2,b_2]}$ are $\F$-equivalent. 
Finally, we say that $\alpha_1$ \emph{accumulates} on $\alpha_2$ if it accumulates either positively or negatively on $\alpha_2$.

The following is \cite[Lemme 2.1.3]{lellouch}, see also \cite[Corollary 3.10]{guiheneuf2023area}:

\begin{prop}\label{Prop2.1.3Lellouch}
If $\alpha : \R\to \Sigma$ is a transverse recurrent path, then it cannot accumulate on itself.
\end{prop}

The following property asserts that the accumulation of a recurrent transverse path on another transverse path only occurs in a very specific configuration \cite[Proposition 3.3]{guiheneuf2023area}:

\begin{prop}\label{Prop3.3GLCP}
Suppose that $\alpha_1:\R\to \Sigma$ is a positively recurrent transverse path that accumulates positively on a transverse path $\alpha_2:\R\to \Sigma$. 
Then, there exists a transverse simple loop $\Gamma_*\subset \Sigma$ with the following properties.
\begin{enumerate}
\item The set $B$ of leaves met by $\Gamma_*$ is an open annulus of $\Sigma$.
\item The path $\alpha_1$ stays in $B$ and is equivalent to the natural lift of $\Gamma_*$.
\item If $\wh\alpha_1$, $\wh\alpha_2$ are lifts of $\alpha_1$, $\alpha_2$ to the universal covering space $\wh \Sigma$ such that $\wh\alpha_1\vert_{[a_1,+\infty)}$ is equivalent to $\wh\alpha_2\vert_{[a_2,b_2)}$ and if $\wh B$ is the lift of $B$ that contains $\wh\alpha_1$, then one of the inclusions $\wh\phi_{\wh\alpha_2(b_2)}\subset \partial \wh B^R$, $\wh\phi_{\wh\alpha_2(b_2)}\subset \partial \wh B^L$ holds. In the first case, we have $\wh B\subset L(\wh \phi)$ for every $\wh \phi\subset \partial \wh B^R$ and in the second case, we have $\wh B\subset R(\wh \phi)$ for every $\wh \phi\subset \partial \wh B^L$. 
\end{enumerate}
\end{prop}

Let us get some additional properties of this configuration, that will be used in the sequel.

\begin{lemma}\label{LemAccumul}
Suppose that $\alpha_0 : \R\to \Sigma$ is a positively recurrent transverse path that accumulates positively on a transverse path $\alpha_1 : J_1\to \Sigma$, where $J_1$ is an interval of $\R$.
Then, $\alpha_0$ accumulates positively on any transverse path $\alpha_2 : J_2\to \Sigma$ that crosses $B$.
\end{lemma}

\begin{proof}
Let us consider a lift $\wh \alpha_0$ of $\alpha_0$ to $\wh \Sigma$, the $T$-band $\wh B$ of leaves of $\wh \F$ met by $\wh\alpha_0$, and a lift $\wh\alpha_1$ of $\alpha_1$ to $\wh S$ such that $\wh\alpha_0$ accumulates positively in $\wh\alpha_1$. Let $\wh\alpha_2 : J_2\to\wh S$ be a transverse path that crosses $\wh B$.

By construction the boundary of $\wh B$ is made of leaves, some of them on the left of $\wh\alpha_0$ (and their union is denoted by $\partial \wh B^L$) and some of them on the right of $\wh\alpha_0$ (and their union is denoted by $\partial \wh B^R$). 
By hypothesis, $\wh\alpha_0$ accumulates positively in $\wh\alpha_1$: there exist $a_0$ and $a_1<b_1$ such that $\wh\alpha_0|_{[a_0,+\infty)}$ is equivalent to $\wh\alpha_1|_{[a_1,b_1)}$. Without loss of generality we can suppose that $\wh\phi_{\wh\alpha_1(b_1)} \in\partial\wh B^L$ (the other case being symmetric).

By Proposition~\ref{Prop3.3GLCP}.3, any transverse trajectory crossing $\partial\wh B^L$ has to cross it from right to left. Let us replace $\alpha_1$ by a transverse path $\alpha'_1$ that crosses $B$ in the following way (see Figure~\ref{FigLemAccumul}, left). Let us consider $a_2<b_2$ such that $\alpha_2((a_2,b_2))\subset B$ and $\alpha_2(a_2)\in \partial B^R$ and $\alpha_2(b_2)\in \partial B^L$. 
There exist $c_2\in (a_2,b_2)$ and $c_0\in\R$ such that $\wh\alpha_0(c_0) = \wh\alpha_2(c_2)$. Finally, there exists $c'_0>c_0$ and $c_1\in (a_1,b_1)$ such that $\wh\phi_{\wh\alpha_1(c_1)} = \wh\phi_{\wh\alpha_0(c'_0)}$. We then consider a transverse loop $\wh\alpha'_1$ that is $\wh\F$-equivalent to $\wh\alpha_2|_{[a_2,c_2]} \wh\alpha_0|_{[c_0, c'_0]}\wh\alpha_0|_{[c_1,b_1]}$. This path $\wh\alpha'_1$ -- as $\wh\alpha_1$ -- has the property that $\wh\alpha_0$ accumulates in it, with the additional property that it crosses $\wh B$. 

The previous property means that $\alpha_2$ crosses $B$ in the same direction as $\alpha'_1$. Now, by restricting $\alpha_2$ if necessary, we can suppose that $J_2$ is compact (and hence bounded).

\begin{figure}
\begin{center}

\tikzset{every picture/.style={line width=0.75pt}} 

\begin{tikzpicture}[x=0.75pt,y=0.75pt,yscale=-.92,xscale=.9]

\draw  [draw opacity=0][fill={rgb, 255:red, 245; green, 166; blue, 35 }  ,fill opacity=0.15 ] (86,96) .. controls (86,76.12) and (102.12,60) .. (122,60) -- (270,60) .. controls (289.88,60) and (306,76.12) .. (306,96) -- (306,204) .. controls (306,223.88) and (289.88,240) .. (270,240) -- (122,240) .. controls (102.12,240) and (86,223.88) .. (86,204) -- cycle ;
\draw [color={rgb, 255:red, 74; green, 144; blue, 226 }  ,draw opacity=1 ]   (86,150) .. controls (126,120) and (146,180) .. (186,150) ;
\draw [color={rgb, 255:red, 74; green, 144; blue, 226 }  ,draw opacity=1 ]   (286,150) .. controls (296.32,142.76) and (302.46,141.48) .. (306.04,141.19) ;
\draw [color={rgb, 255:red, 245; green, 166; blue, 35 }  ,draw opacity=1 ][fill={rgb, 255:red, 255; green, 255; blue, 255 }  ,fill opacity=1 ]   (186,60) .. controls (207.83,85.46) and (232.33,93.96) .. (236,60) ;
\draw [color={rgb, 255:red, 245; green, 166; blue, 35 }  ,draw opacity=1 ][fill={rgb, 255:red, 255; green, 255; blue, 255 }  ,fill opacity=1 ]   (246,60) .. controls (244.83,83.46) and (263.83,93.96) .. (276,60) ;
\draw [color={rgb, 255:red, 245; green, 166; blue, 35 }  ,draw opacity=1 ][fill={rgb, 255:red, 255; green, 255; blue, 255 }  ,fill opacity=1 ]   (146,60) .. controls (144.83,83.46) and (163.83,93.96) .. (176,60) ;
\draw [color={rgb, 255:red, 245; green, 166; blue, 35 }  ,draw opacity=1 ][fill={rgb, 255:red, 255; green, 255; blue, 255 }  ,fill opacity=1 ]   (86,60) .. controls (107.83,85.46) and (132.33,93.96) .. (136,60) ;
\draw  [draw opacity=0][fill={rgb, 255:red, 255; green, 255; blue, 255 }  ,fill opacity=1 ] (80.51,63.22) .. controls (80.51,58.25) and (84.54,54.22) .. (89.51,54.22) .. controls (94.48,54.22) and (98.51,58.25) .. (98.51,63.22) .. controls (98.51,68.19) and (94.48,72.22) .. (89.51,72.22) .. controls (84.54,72.22) and (80.51,68.19) .. (80.51,63.22) -- cycle ;
\draw  [draw opacity=0][fill={rgb, 255:red, 255; green, 255; blue, 255 }  ,fill opacity=1 ] (353.8,59) .. controls (353.8,54.03) and (357.83,50) .. (362.8,50) .. controls (367.77,50) and (371.8,54.03) .. (371.8,59) .. controls (371.8,63.97) and (367.77,68) .. (362.8,68) .. controls (357.83,68) and (353.8,63.97) .. (353.8,59) -- cycle ;
\draw [color={rgb, 255:red, 245; green, 166; blue, 35 }  ,draw opacity=1 ][fill={rgb, 255:red, 255; green, 255; blue, 255 }  ,fill opacity=1 ]   (216,240) .. controls (206.14,202.67) and (246.81,214.33) .. (266,240) ;
\draw [color={rgb, 255:red, 245; green, 166; blue, 35 }  ,draw opacity=1 ][fill={rgb, 255:red, 255; green, 255; blue, 255 }  ,fill opacity=1 ]   (106,240) .. controls (96.14,202.67) and (136.81,214.33) .. (156,240) ;
\draw  [draw opacity=0][fill={rgb, 255:red, 255; green, 255; blue, 255 }  ,fill opacity=1 ] (347.9,241.02) .. controls (347.9,238.84) and (349.67,237.08) .. (351.84,237.08) .. controls (354.02,237.08) and (355.79,238.84) .. (355.79,241.02) .. controls (355.79,243.2) and (354.02,244.96) .. (351.84,244.96) .. controls (349.67,244.96) and (347.9,243.2) .. (347.9,241.02) -- cycle ;
\draw  [draw opacity=0][fill={rgb, 255:red, 255; green, 255; blue, 255 }  ,fill opacity=1 ] (99.7,238.16) .. controls (99.7,236.06) and (101.4,234.36) .. (103.51,234.36) .. controls (105.61,234.36) and (107.31,236.06) .. (107.31,238.16) .. controls (107.31,240.26) and (105.61,241.97) .. (103.51,241.97) .. controls (101.4,241.97) and (99.7,240.26) .. (99.7,238.16) -- cycle ;
\draw [color={rgb, 255:red, 39; green, 179; blue, 219 }  ,draw opacity=1 ]   (160.05,114.05) .. controls (172.05,86.45) and (229.65,119.65) .. (214.45,63.25) ;
\draw [shift={(203.47,98.15)}, rotate = 167.85] [fill={rgb, 255:red, 39; green, 179; blue, 219 }  ,fill opacity=1 ][line width=0.08]  [draw opacity=0] (8.04,-3.86) -- (0,0) -- (8.04,3.86) -- (5.34,0) -- cycle    ;
\draw [color={rgb, 255:red, 39; green, 179; blue, 219 }  ,draw opacity=1 ] [dash pattern={on 0.84pt off 2.51pt}]  (151.62,125.18) .. controls (155.99,123.52) and (157.62,119.38) .. (160.05,114.05) ;
\draw [color={rgb, 255:red, 245; green, 166; blue, 35 }  ,draw opacity=1 ]   (180.75,60.54) .. controls (178.84,93.29) and (305,125.08) .. (287.5,235.58) ;
\draw [color={rgb, 255:red, 21; green, 146; blue, 0 }  ,draw opacity=1 ]   (133.15,220.03) .. controls (155.51,190.62) and (173.15,178.15) .. (210.56,142.5) ;
\draw [color={rgb, 255:red, 21; green, 146; blue, 0 }  ,draw opacity=1 ]   (210.56,142.5) .. controls (234.74,144.42) and (244.55,160.99) .. (266.37,159.58) ;
\draw [color={rgb, 255:red, 21; green, 146; blue, 0 }  ,draw opacity=1 ]   (266.37,159.58) .. controls (249.74,126.65) and (232.07,116.22) .. (213.45,93.14) ;
\draw [color={rgb, 255:red, 21; green, 146; blue, 0 }  ,draw opacity=1 ]   (213.45,93.14) .. controls (216.27,90.05) and (217.18,86.69) .. (217.45,82.96) ;
\draw [color={rgb, 255:red, 136; green, 74; blue, 226 }  ,draw opacity=1 ]   (126,230) .. controls (151.25,180.85) and (251.65,124.45) .. (261.65,65.65) ;
\draw [shift={(202.19,148.95)}, rotate = 136.88] [fill={rgb, 255:red, 136; green, 74; blue, 226 }  ,fill opacity=1 ][line width=0.08]  [draw opacity=0] (8.04,-3.86) -- (0,0) -- (8.04,3.86) -- (5.34,0) -- cycle    ;
\draw [color={rgb, 255:red, 74; green, 144; blue, 226 }  ,draw opacity=1 ]   (186,150) .. controls (226,120) and (246,180) .. (286,150) ;
\draw [shift={(239.15,151.57)}, rotate = 206.47] [fill={rgb, 255:red, 74; green, 144; blue, 226 }  ,fill opacity=1 ][line width=0.08]  [draw opacity=0] (8.04,-3.86) -- (0,0) -- (8.04,3.86) -- (5.34,0) -- cycle    ;
\draw  [draw opacity=0][fill={rgb, 255:red, 245; green, 166; blue, 35 }  ,fill opacity=0.15 ] (345.49,96) .. controls (345.49,76.12) and (361.6,60) .. (381.49,60) -- (609.49,60) .. controls (629.37,60) and (645.49,76.12) .. (645.49,96) -- (645.49,204) .. controls (645.49,223.88) and (629.37,240) .. (609.49,240) -- (381.49,240) .. controls (361.6,240) and (345.49,223.88) .. (345.49,204) -- cycle ;
\draw [color={rgb, 255:red, 74; green, 144; blue, 226 }  ,draw opacity=1 ]   (345.49,150.64) .. controls (385.49,120.64) and (405.49,180.64) .. (445.49,150.64) ;
\draw [color={rgb, 255:red, 245; green, 166; blue, 35 }  ,draw opacity=1 ][fill={rgb, 255:red, 255; green, 255; blue, 255 }  ,fill opacity=1 ]   (445.49,60) .. controls (467.32,85.46) and (491.82,93.96) .. (495.49,60) ;
\draw [color={rgb, 255:red, 245; green, 166; blue, 35 }  ,draw opacity=1 ][fill={rgb, 255:red, 255; green, 255; blue, 255 }  ,fill opacity=1 ]   (505.49,60.64) .. controls (504.32,84.1) and (523.32,94.6) .. (535.49,60.64) ;
\draw [color={rgb, 255:red, 245; green, 166; blue, 35 }  ,draw opacity=1 ][fill={rgb, 255:red, 255; green, 255; blue, 255 }  ,fill opacity=1 ]   (405.49,60.64) .. controls (404.32,84.1) and (423.32,94.6) .. (435.49,60.64) ;
\draw [color={rgb, 255:red, 245; green, 166; blue, 35 }  ,draw opacity=1 ][fill={rgb, 255:red, 255; green, 255; blue, 255 }  ,fill opacity=1 ]   (345.49,60.64) .. controls (367.32,86.1) and (391.82,94.6) .. (395.49,60.64) ;
\draw  [draw opacity=0][fill={rgb, 255:red, 255; green, 255; blue, 255 }  ,fill opacity=1 ] (340,63.87) .. controls (340,58.9) and (344.03,54.87) .. (349,54.87) .. controls (353.97,54.87) and (358,58.9) .. (358,63.87) .. controls (358,68.84) and (353.97,72.87) .. (349,72.87) .. controls (344.03,72.87) and (340,68.84) .. (340,63.87) -- cycle ;
\draw [color={rgb, 255:red, 245; green, 166; blue, 35 }  ,draw opacity=1 ][fill={rgb, 255:red, 255; green, 255; blue, 255 }  ,fill opacity=1 ]   (475.99,240.64) .. controls (466.13,203.31) and (506.79,214.98) .. (525.99,240.64) ;
\draw [color={rgb, 255:red, 245; green, 166; blue, 35 }  ,draw opacity=1 ][fill={rgb, 255:red, 255; green, 255; blue, 255 }  ,fill opacity=1 ]   (365.49,240.64) .. controls (355.63,203.31) and (396.29,214.98) .. (415.49,240.64) ;
\draw  [draw opacity=0][fill={rgb, 255:red, 255; green, 255; blue, 255 }  ,fill opacity=1 ] (359.19,238.81) .. controls (359.19,236.7) and (360.89,235) .. (362.99,235) .. controls (365.09,235) and (366.8,236.7) .. (366.8,238.81) .. controls (366.8,240.91) and (365.09,242.61) .. (362.99,242.61) .. controls (360.89,242.61) and (359.19,240.91) .. (359.19,238.81) -- cycle ;
\draw [color={rgb, 255:red, 21; green, 146; blue, 0 }  ,draw opacity=1 ]   (372.3,230.22) .. controls (416.59,103.36) and (489.13,120.29) .. (473.93,63.89) ;
\draw [shift={(424.21,140.55)}, rotate = 132.59] [fill={rgb, 255:red, 21; green, 146; blue, 0 }  ,fill opacity=1 ][line width=0.08]  [draw opacity=0] (8.04,-3.86) -- (0,0) -- (8.04,3.86) -- (5.34,0) -- cycle    ;
\draw [color={rgb, 255:red, 74; green, 144; blue, 226 }  ,draw opacity=1 ]   (445.49,150.64) .. controls (485.49,120.64) and (505.49,180.64) .. (545.49,150.64) ;
\draw [shift={(498.64,152.21)}, rotate = 206.47] [fill={rgb, 255:red, 74; green, 144; blue, 226 }  ,fill opacity=1 ][line width=0.08]  [draw opacity=0] (8.04,-3.86) -- (0,0) -- (8.04,3.86) -- (5.34,0) -- cycle    ;
\draw [color={rgb, 255:red, 74; green, 144; blue, 226 }  ,draw opacity=1 ]   (545.49,150.64) .. controls (585.49,120.64) and (605.49,180.64) .. (645.49,150.64) ;
\draw [color={rgb, 255:red, 245; green, 166; blue, 35 }  ,draw opacity=1 ][fill={rgb, 255:red, 255; green, 255; blue, 255 }  ,fill opacity=1 ]   (545.49,60) .. controls (567.32,85.46) and (591.82,93.96) .. (595.49,60) ;
\draw [color={rgb, 255:red, 245; green, 166; blue, 35 }  ,draw opacity=1 ][fill={rgb, 255:red, 255; green, 255; blue, 255 }  ,fill opacity=1 ]   (605.49,60) .. controls (604.32,83.46) and (623.32,93.96) .. (635.49,60) ;
\draw  [draw opacity=0][fill={rgb, 255:red, 255; green, 255; blue, 255 }  ,fill opacity=1 ] (626.49,60) .. controls (626.49,55.03) and (630.52,51) .. (635.49,51) .. controls (640.46,51) and (644.49,55.03) .. (644.49,60) .. controls (644.49,64.97) and (640.46,69) .. (635.49,69) .. controls (630.52,69) and (626.49,64.97) .. (626.49,60) -- cycle ;
\draw [color={rgb, 255:red, 245; green, 166; blue, 35 }  ,draw opacity=1 ][fill={rgb, 255:red, 255; green, 255; blue, 255 }  ,fill opacity=1 ]   (585.49,240) .. controls (575.63,202.67) and (616.29,214.33) .. (635.49,240) ;
\draw  [draw opacity=0][fill={rgb, 255:red, 255; green, 255; blue, 255 }  ,fill opacity=1 ] (626.49,240) .. controls (626.49,235.03) and (630.52,231) .. (635.49,231) .. controls (640.46,231) and (644.49,235.03) .. (644.49,240) .. controls (644.49,244.97) and (640.46,249) .. (635.49,249) .. controls (630.52,249) and (626.49,244.97) .. (626.49,240) -- cycle ;
\draw [color={rgb, 255:red, 136; green, 74; blue, 226 }  ,draw opacity=1 ]   (485.49,230) .. controls (510.73,180.85) and (611.13,124.45) .. (621.13,65.65) ;
\draw [shift={(561.68,148.95)}, rotate = 136.88] [fill={rgb, 255:red, 136; green, 74; blue, 226 }  ,fill opacity=1 ][line width=0.08]  [draw opacity=0] (8.04,-3.86) -- (0,0) -- (8.04,3.86) -- (5.34,0) -- cycle    ;
\draw  [draw opacity=0][fill={rgb, 255:red, 74; green, 144; blue, 226 }  ,fill opacity=1 ] (604.44,156.51) .. controls (604.44,154.88) and (605.76,153.56) .. (607.39,153.56) .. controls (609.02,153.56) and (610.34,154.88) .. (610.34,156.51) .. controls (610.34,158.14) and (609.02,159.46) .. (607.39,159.46) .. controls (605.76,159.46) and (604.44,158.14) .. (604.44,156.51) -- cycle ;
\draw [color={rgb, 255:red, 245; green, 166; blue, 35 }  ,draw opacity=1 ]   (348.55,81.81) .. controls (398.8,97.56) and (390.5,74.22) .. (400.1,73.82) .. controls (409.7,73.42) and (406.15,89.23) .. (413.3,91.02) .. controls (420.44,92.81) and (433.7,71.82) .. (444.1,71.42) .. controls (454.5,71.02) and (449.07,101.28) .. (552.74,100.94) .. controls (656.4,100.61) and (634.4,189.61) .. (646.07,198.94) ;

\draw (173.33,102.4) node [anchor=north west][inner sep=0.75pt]  [color={rgb, 255:red, 30; green, 149; blue, 182 }  ,opacity=1 ,xscale=1.2,yscale=1.2]  {$\wh{\alpha }_{1}$};
\draw (137.9,146.4) node [anchor=north east] [inner sep=0.75pt]  [color={rgb, 255:red, 74; green, 144; blue, 226 }  ,opacity=1 ,xscale=1.2,yscale=1.2]  {$\wh{\alpha }_{0}$};
\draw (256.83,85.9) node [anchor=north west][inner sep=0.75pt]  [color={rgb, 255:red, 120; green, 66; blue, 199 }  ,opacity=1 ,xscale=1.2,yscale=1.2]  {$\wh{\alpha }_{2}$};
\draw (173.33,177.9) node [anchor=north west][inner sep=0.75pt]  [color={rgb, 255:red, 21; green, 146; blue, 0 }  ,opacity=1 ,xscale=1.2,yscale=1.2]  {$\wh{\alpha } '_{1}$};
\draw (93.33,96.9) node [anchor=north west][inner sep=0.75pt]  [color={rgb, 255:red, 199; green, 135; blue, 28 }  ,opacity=1 ,xscale=1.2,yscale=1.2]  {$\wh{B}$};
\draw (435,132.38) node [anchor=south east] [inner sep=0.75pt]  [color={rgb, 255:red, 21; green, 146; blue, 0 }  ,opacity=1 ,xscale=1.2,yscale=1.2]  {$T^{k_{0}}\wh{\alpha } '_{1}$};
\draw (382.39,144) node [anchor=north east] [inner sep=0.75pt]  [color={rgb, 255:red, 74; green, 144; blue, 226 }  ,opacity=1 ,xscale=1.2,yscale=1.2]  {$\wh{\alpha }_{0}$};
\draw (509.55,194.43) node [anchor=south east] [inner sep=0.75pt]  [color={rgb, 255:red, 120; green, 66; blue, 199 }  ,opacity=1 ,xscale=1.2,yscale=1.2]  {$\wh{\alpha }_{2}$};
\draw (350,97.54) node [anchor=north west][inner sep=0.75pt]  [color={rgb, 255:red, 199; green, 135; blue, 28 }  ,opacity=1 ,xscale=1.2,yscale=1.2]  {$\wh{B}$};
\draw (590.48,160) node [anchor=north] [inner sep=0.75pt]  [color={rgb, 255:red, 74; green, 144; blue, 226 }  ,opacity=1 ,xscale=1.2,yscale=1.2]  {$\scriptstyle \wh{\alpha }_{0}( a_{0} +k_{0})$};
\draw (529.02,100.37) node [anchor=north east] [inner sep=0.75pt]  [color={rgb, 255:red, 218; green, 147; blue, 29 }  ,opacity=1 ,xscale=1.2,yscale=1.2]  {$\wh{\phi }_{t}$};

\end{tikzpicture}
\caption{proof of Lemma~\ref{LemAccumul}. Left: construction of the path $\wh\alpha'_1$. Right: final argument of the proof. Leaves of $\wh \F$ are in orange.}\label{FigLemAccumul}
\end{center}
\end{figure}
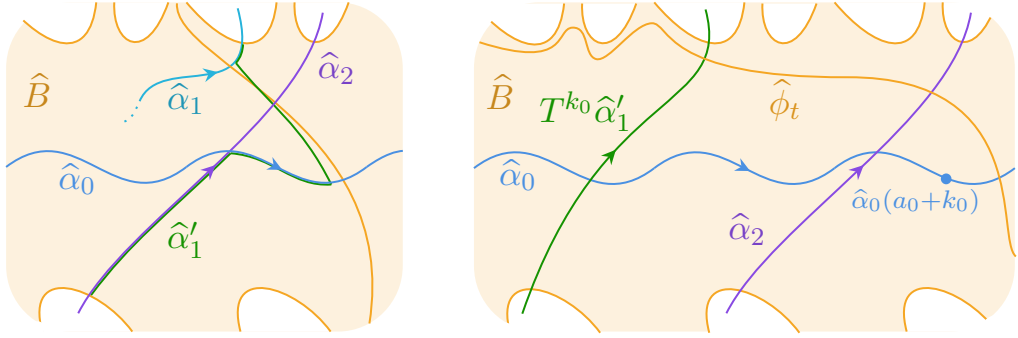

Recall that by Proposition~\ref{Prop3.3GLCP}, $\alpha_0$ is $\F$-equivalent to a $T$-invariant $\F$-transverse path, so $\wt\alpha_0$ also accumulates in $T^{k}\wt\alpha_1$ for any $k\in\Z$: there exist $(t_k)$ such that the path $\wt\alpha_0|_{[a_0+t_k,+\infty)}$ is equivalent to $T^{k}\wt\alpha'_1|_{[a_1,b_1)}$. 
As $\wt\alpha'_1$ and $\wt\alpha_2$ are simple and separate $\wt B$, their respective complements in $\wt B$ are made of two connected components, one on their left and one on their right.
As $\wt\alpha'_1$ and $\wt\alpha_2$ are compact, there exists $-k_0$ large enough such that $T^{k_0} \wt\alpha_1$ sits on the left of $\wt\alpha_2$ in $\wt B$. By choosing a bigger $a_0$ if necessary, we can suppose that $\wt\alpha_0|_{[a_0+t_{k_0},+\infty)}$ is on the right of $\wt\alpha_2$ in $\wt B$.
Because $\wt\alpha_0|_{[a_0+t_{k_0},+\infty)}$ is equivalent to $T^{k_0}\wt\alpha'_1|_{[a_1,b_1)}$, for any $t\ge a_0+t_{k_0}$ there is a leaf segment $\wt\phi_t$ linking $\wt\alpha_0(t)$ to $T^{k_0}\wt\alpha_1$. 

As the leaf segments $\wt\phi_t$ link points from the left of $\wt\alpha_2$ (in $\wt B$) to the right of $\wt \alpha_2$ (in $\wt B$) and stay in $\wt B$, they have to cross the transverse path $\wt\alpha_2$ at a unique point; for the order on $\wt\alpha_2$, this point varies $C^0$ and monotonically in $t$. Hence, the path $\wt\alpha_0|_{[a_0+t_{k_0},+\infty)}$ is equivalent to a subpath of $\wt\alpha_2$; this proves the lemma. 
\end{proof}

\begin{lemma}\label{LemAccumul2}
Suppose that $\alpha_0 : \R\to \Sigma$ is a positively recurrent transverse path that accumulates positively on a transverse path $\alpha_1 : J_1\to \Sigma$, where $J_1$ is an interval of $\R$. Denote $\wh\alpha_0$ and $\wh\alpha_1$ some lifts of these paths to $\wh\Sigma$ such that $\wh\alpha_0$ accumulates positively in $\wh\alpha_1$.
Suppose that the accumulated leaf belongs to $\partial \wh B^L$, and that $\wh\alpha_2 : J_2\to\wh S$, with $J_2$ a compact interval, is a transverse path that enters in $\wh B$ but does not meet $\partial\wh B^L$. 
Then, for $t$ large enough, the leaf $\wh\phi_{\wh\alpha_0(t)}$ does not meet $\wh\alpha_2$.
\end{lemma}


\begin{proof}
Let us first show that the path $\wh\alpha_2$ goes in and out of $\wh B$ a finite number of times

For any $t\in J_2$, there exists an open interval $I_t\ni t$ such that:
\begin{itemize}
\item either the path $\wh\alpha_2$ is in $\wh B$ in a neighbourhood of $t$: for any $s\in I_t\cap J_2$, we have $\wh\alpha_2(s)\in\wh B$;
\item or the path $\wh\alpha_2$ is out of $\wh B$ in a neighbourhood of $t$: for any $s\in I_t\cap J_2$, we have $\wh\alpha_2(s)\notin\wh B$;
\item or the path $\wh\alpha_2$ goes out of $\wh B$ at time $t$: for any $s\in I_t\cap J_2\cap (-\infty,t)$, we have $\wh\alpha_2(s)\in\wh B$ and for any $s\in I_t\cap J_2\cap [t,+\infty)$, we have $\wh\alpha_2(s)\notin\wh B$;
\item or the path $\wh\alpha_2$ enters $\wh B$ at time $t$: for any $s\in I_t\cap J_2\cap (-\infty,t]$, we have $\wh\alpha_2(s)\notin\wh B$ and for any $s\in I_t\cap J_2\cap (t,+\infty)$, we have $\wh\alpha_2(s)\in\wh B$.
\end{itemize}
By compactness of $J_2$, one can cover the interval $J_2$ with a finite number of such intervals $I_t$; in particular the path $\wh\alpha_2$ goes in and out of $\wh B$ a finite number of times. Hence, to prove the lemma, we do not lose generality by supposing that $\wh\alpha_2|_{\inte J_2}\subset \wh B$ and that $\wh\alpha_2(\min J_2)\in \partial \wh B^R$ (as $\wh\alpha_2$ can only enter $\wh B$ from the right).
\bigskip

Consider now the ``projection on $\wh\alpha_0$ along the leaves'' $H:\wh B\to \R$ such that $\wh\phi_{\wh{p}}=\wh\phi_{\wh\alpha_0(H(\wh{p}))}$, which is continuous.
Since $\wh \alpha_2(t)$ is contained in $\wh B$ if $t\in\inte J_2$, one can consider the function $W:\inte J_2\to\R$ such that $W(t)=H(\wh \alpha_2(t))$, which is continuous, and as $\wh \alpha_2$ is transverse to $\wh \F$, $W$ must be monotone increasing. If $\wh\alpha_2(\max J_2)$ belongs to $\wh B$, then $W$ extends continuously in its right endpoint to $W(\max J_2)=H(\wh\alpha_2(\max J_2))=L<\infty$ and of course the result holds for any $t>L$. If not, then as $\wh \alpha_2(\max J_2)$ does not belong to $\partial\wh B^L$ it must lie on $\partial\wh B^R$. We claim that in this situation, we still have that $\sup_{t\in\inte J2} W(t)<+\infty$, which shows the lemma. Indeed, if that was not the case, then picking some $t_2\in\inte J_2$, we would have that $\wh \alpha_2:[t_2,\max J_2)$ is $\wh \F$ equivalent to $\wh \alpha_0$, but this would imply, by Proposition~\ref{Prop3.3GLCP}, that no path can enter $\wh B$ from the right either, a contradiction.
\end{proof}

\subsection{Brouwer-Le Calvez foliations and forcing theory}

Let $\F$ be a singular foliation of a surface $S$; we denote $\mathrm{Sing}(\F)$ its set of singularities and $\dom(\F) := S \setminus \mathrm{Sing}(\F)$. 
The forcing theory is grounded on the following result of existence of transverse foliations, which can be obtained as a combination of \cite{lecalvezfoliations} and \cite{bguin2016fixed}.

\begin{theorem}\label{ThExistIstop}
Let $S$ be a surface and $f\in\Homeo_0(S)$.
Then there exist an identity isotopy $I$ for $f$ and a transverse topological oriented singular foliation $\F$ of $S$ with $\dom(\F) = \dom(I)$, such that:

	For any $z\in \dom(\F)$, there exists an $\F$-transverse path denoted by $\big(I_\F^t(z)\big)_{t\in[0,1]}$, linking $z$ to $f(z)$, that is homotopic in $\dom(\F)$, relative to its endpoints, to the arc $(I^t(z))_{t\in[0,1]}$.
\end{theorem}

This allows us to define the path $I_{\F}^\Z (x)$ as the concatenation of the paths $\big(I_\F^t(f^n(z))\big)_{t\in[0,1]}$ for $n\in\Z$.

In \cite{paper2PAF} we will need the following lemma:

\begin{lemma}\label{lemma:Maxsize}
If $S$ is a closed surface and $\Fix(I)$ is contained in a topological disc, then there exists $M>0$ such that, for every $z\in \dom(\F)$, one can choose $\big(I_\F^t(z)\big)_{t\in[0,1]}$ such that, if $\big(I_{\wt\F}^t(\wt z)\big)_{t\in[0,1]}$ is a lift of $\big(I_\F^t(z)\big)_{t\in[0,1]}$ to $\widetilde S$, then the diameter of $\big(I_{\wt\F}^t(\wt z)\big)_{t\in[0,1]}$ is at most $M$.
\end{lemma}

\begin{proof}
Let $U_0$ be an open topological disc containing $\Fix(I)$. Since $\Fix(I)$ is compact, it is at a positive distance from the boundary of $U_0$ and we may assume therefore that $\overline{U_0}$ is a closed topological disc, which implies that if $\wt U_0$ is a connected component of the lift of $U_0$ to $\wt{S}$, then $\diam(\wt{U_0}) \le M_0< +\infty$. By \cite[Proposition 5.7]{alepablo}, there exists $V_0$ an open neighbourhood of $\Fix(I)$ such that, if $z\in V_0\setminus \Fix(I)=V_0\cap \dom(\F)$, then $I_{\F}^{[0,1]}(z)$ can be chosen having its image in $U_0$. 
So it suffices to prove the result for all $z\notin V_0$. We claim that, for any $z \notin V_0$, there exists an open set $U_z\ni z$ and some $M_z>0$ such that, for any $z'$ in $U_z$ one can choose $I_{\F}^{[0,1]}(z')$ such that the diameter of a lift $I_{\wt \F}^{[0,1]}(\wt z')$ is smaller than $M_z$. This will finish the result by compactness of $S\setminus V_0$.

To see the claim, choose $W_z$ and $W_{f(z)}$ trivialisation neighbourhoods of $\F$ containing $z$ and $f(z)$ respectively, and a path $I_{\F}^{[0,1]}(z)$. Let $\delta>0$ be such that $I_{\F}^{[0,\delta]}(z)$ is contained in $W_z$ and $I_{\F}^{[1-\delta,1]}(z)$ is contained in $W_{f(z)}$. Note that, by continuity, if $z'$ is sufficiently close to $z$, then the path $\beta_{z'}=\sigma_{f(z')}I_{\F}^{[\delta, 1-\delta]}(z)\sigma_{z'}$, where $\sigma_{z'}$ is a transverse path in $W_z$ connecting $z'$ to $I_{\F}^{\delta}(z)$ and $\sigma_{f(z')}$ is a transverse path in $W_{f(z)}$ connecting $I_{\F}^{ 1-\delta}(z)$ to $f(z')$, is homotopic with fixed endpoints to $I^{[0,1]}(z')$, which implies the claim.    
\end{proof}

We will say that a transverse path $\alpha : [a, b] \to \dom(I )$ is \emph{admissible of order $n$} if it is $\F$-equivalent to a path $I^{[0,n]}_\F(z)$ for some $z\in\dom(I)$.

The following is \cite[Lemma~17]{lct1}, it is a straightforward consequence of continuity properties of $\F$. 

\begin{lemma}\label{Lem17Lct1}
Let $z\in\dom(\F)$ and $n\ge 1$. 
Then there exists a neighbourhood $W$ of $z$ such that, for every $z',z''\in W$, the path $I^n_\F(z')$ is $\F$-equivalent to a subpath of $I^{n+2}_\F(f^{-1}(z''))$.
\end{lemma}

The following statement is a reformulation of the main technical result of the forcing theory \cite{lct1} (Proposition~20):

\begin{prop}\label{propFondalct1}
Suppose that $I^{[t, t']}_\F(z)$ and $I^{[s,s']}_\F(z')$ intersect $\F$-transversally at $I^{t''}_\F(z) = I^{s''}_\F(z')$. Then the path $I^{[t, t'']}_\F(z) I^{[s'',s']}_\F(z')$ is $f$-admissible of order $\lceil t'-t\rceil+ \lceil s'-s\rceil$.
\end{prop}

Another important (but much more technical) result of the forcing theory is a simple criterion (in terms of transverse intersections of paths) of existence of horseshoes. It is the main technical result of \cite{lct2} (Theorem M):

\begin{theorem}\label{thmMlct2}
Suppose that  $\gamma:[a,b]\to \mathrm{dom}(I)$ is an admissible path of order $r$. 
Let $\wh \gamma$ be a lift of $\gamma$ to the universal covering space $\wh{\dom}(\F)$ and suppose there exists a covering automorphism $T$ such that $\wh\gamma$ and $T(\wh\gamma)$ have an $\wh{\mathcal F}$-transverse intersection at $\wh\gamma(t)=T(\wh\gamma)(s)$, with $s<t$. 

Then there exists a point $\wh z\in \wh\dom(\F)$ such that $\wh f^{r}(\wh z)=T(\wh z)$.
\end{theorem}

\subsection{Band defined by a transverse loop}\label{SecBand}

Fix $f\in\Homeo_0(S)$; let $I$ and $\F$ be the isotopy and the foliation given by Theorem~\ref{ThExistIstop}. 
In the sequel we will denote with $\tilde{}$ the lifts to the universal cover $\wt S$ of $S$, and with $\hat{}$ the lifts to the universal cover $\wh\dom(\F)$ of $\dom(\F)$. In particular, $\wt\F$ and $\wh\F$ will be the lifts of $\F$ to respectively $\wt S$ and $\wh\dom(\F)$. 

Let $\wh\beta :\R \to \wh\dom(\F)$ be an $\F$-transverse path (note that this implies that $\wh\beta$ is a topological line).
We say that $\wh\beta$ is a \emph{$T$-loop} if it is invariant under $T\in\G$. The union of leaves of $\wh\F$ met by $\wh\beta$, denoted $\wh B$, is called the \emph{band} or the \emph{$T$-band} defined by $\wh\beta$. The frontier $\partial \wh B$ of $\wh B$ is a (possibly empty) union of leaves which can be written $\partial\wh B=\partial\wh B^R\sqcup\partial\wh B^L$, with
\[\partial\wh B^R=\partial\wh B\cap R(\wh\beta)
\qquad \textrm{and}\qquad 
\partial\wh B^L=\partial\wh B\cap L(\wh\beta).\]

Let $\wh\alpha : \wh\dom(\F)\to \R$ be a transverse path, and suppose that 
\[\big\{ t\in\R \mid \wh\alpha(t)\in\wh B\big\}=(a,b),\]
where $-\infty\leq a<b\leq\infty$. We say that
\begin{itemize}
\item $\wh\alpha$ {\it draws $\wh B$} if there exist $t<t'$ in $(a,b)$ such that $\wh\phi_{\wh\alpha(t')}=T\wh\phi_{\wh\alpha(t)}$.
\end{itemize}
If, moreover, we suppose that $-\infty<a<b<+\infty$, say that:
\begin{itemize}
\item $\wh\alpha$ \emph{crosses $\wh B$ from right to left} if $\wh\alpha(a)\in \partial\wh B^R$ and $\wh\alpha(b)\in \partial\wh B^L$;
\item $\wh\alpha$ \emph{crosses $\wh B$ from left to right} if $\wh\alpha(a)\in \partial\wh B^L$ and $\wh\alpha(b)\in \partial\wh B^R$;
\item $\wh\alpha$ \emph{visits $\wh B$ on the right} if $\wh\alpha(a)\in \partial\wh B^R$ and $\wh\alpha(b)\in \partial\wh B^R$;
\item $\wh\alpha$ \emph{visits $\wh B$ on the left} if $\wh\alpha(a)\in \partial\wh B^L$ and $\wh\alpha(b)\in \partial\wh B^L$.
\end{itemize}
We say that $\wh\alpha$ \emph{crosses} $\wh B$ if it crosses it from right to left or from left to right. Similarly, $\wh\alpha$ \emph{visits} $\wh B$ if it visits it on the right or on the left. In the case where $\wh\alpha$ draws, crosses or visits $\wh B$, we will say that $(a,b)$ is a \emph{drawing, crossing} or \emph{visiting component} in $\wh B$.  

\begin{rem}\label{RemPossibBand}
A transverse path $\wh\alpha$ drawing $\wh B$ either crosses or visits $\wh B$, or accumulates in $\wh\gamma$, or is equivalent to $\wh \gamma$ at $+\infty$ or $-\infty$.
\end{rem}

If $\alpha$ is a transverse path meeting both a leaf $\wh\phi$ and its image $T\wh\phi$ by $T\in\G$, and if $\gamma$ is a $T$-loop meeting $\wh\phi$, then we say that $\wh\gamma$ is an \emph{approximation} of $\wh\alpha$.
\bigskip

We first give a criterion for a trajectory to stay in a band \cite[Proposition 2.1.17]{lellouch}.

\begin{prop}\label{Prop2.1.17lellouch}
Let $\alpha :  \R\to\dom(\F)$ be a transverse recurrent path, and $\wt\alpha$ be a lift of $\alpha$ to $\wt S$. If there exists $T\in\G$ such that $T\wt\alpha\sim_{+\infty}\wt\alpha$ (resp. $T\wt\alpha\sim_{-\infty}\wt\alpha$), then there exists a transverse $T$-loop $\wt \beta : \R\to \wt S$ such that $\wt\alpha\sim_{+\infty}\wt\beta$ (resp. $\wt\alpha\sim_{-\infty}\wt\beta$).
\end{prop}

Let us give two criteria of existence of transverse intersections in terms of the notions we just defined.
The first one is \cite[Proposition 2.1.15]{lellouch}\footnote{The fact that the transverse intersection occurs with the $T$-translate is not contained in the lemma's statement but is stated at the end of the proof.}.

\begin{prop}\label{Prop2.1.15Lellouch}
Let $\wh\alpha : \R\to\wh\dom(\F)$ be a transverse path and $\wh\gamma$ an approximation of $\wh\alpha$ that is a $T$-loop. 
If $\wh\alpha$ visits the band $\wh B$ defined by $\wh\gamma$, then $\wh\alpha$ and $T\wh\alpha$ intersect $\wh\F$-transversally.
\end{prop}

The second one is \cite[Proposition 2.1.16]{lellouch}.

\begin{prop}\label{Prop2.1.16Lellouch}
Let $\wh\alpha, \wh\beta : \R\to\wh\dom(\F)$ be two transverse paths and $\wh\gamma$ an approximation of $\wh\alpha$ that is a $T$-loop. 

If $\wh\alpha$ crosses the band $\wh B$ defined by $\wh\gamma$ from left to right, and $\wh\beta$ crosses the band $\wh B$ from right to left, then there exists $n\in\Z$ such that $\wh\alpha$ and $T^n\wh\beta$ intersect $\wh\F$-transversally.

Similarly, if $\wh\alpha$ crosses the band $\wh B$ defined by $\wh\gamma$ from right to left, and $\wh\beta$ crosses the band $\wh B$ from left to right, then there exists $n\in\Z$ such that $\wh\alpha$ and $T^n\wh\beta$ intersect $\wh\F$-transversally.
\end{prop}

The following lemma is straightforward.

\begin{lemma}\label{LemPasAccImplTrans}
Let $\wh\alpha, \wh\beta : \R\to\wh\dom(\F)$ be two transverse paths such that $\wh\beta$ is a $T$-loop for some $T\in\G\setminus\{\Id\}$. 

If $\wh\alpha$ crosses the band $\wh B$ defined by $\wh\beta$ and $\wh\beta$ does not accumulate in $\alpha$, then the transverse paths $\wh\beta$ and $\wh\alpha$ intersect $\wh\F$-transversally.
\end{lemma}

The following specifies the rotational properties of a $\mu$-typical point whose trajectory is $\F$-equivalent to a transverse loop \cite[Proposition 2.2.13]{lellouch}.

\begin{prop}\label{Prop2.2.13Lellouch}
Let $\mu\in\Me(f)$ be a measure that is not supported in a single fixed point of the isotopy $I$, and $z\in\dom(\F)$ be a $\mu$-typical point. Let $\alpha : \R\to S$ be an $\F$-transverse loop.
If $I^\Z_{\F}(z)$ and $\alpha$ are $\F$-equivalent (at $+\infty$), then for any lift $\wh z$ of $z$ to $\wh\dom (\F)$ there exists a bounded neighbourhood $\wh W$ of $\wh z$ and two increasing sequences $(\ell_n)$ and $(q_n)$ such that for any $n\in\N$, one has $\wh f^{\ell_n}(\wh z) \in T^{q_n}(\wh W)$. 
\end{prop}

Let us give a criterion of $\wt\F$-transverse intersection in terms of drawing components. It is based on the following result \cite[Proposition~24]{lct2}:

\begin{lemma}\label{LemProp24Lct2}
Suppose that $\wt\alpha : J \to \wt S$ is a transverse path with no $\wt\F$-transverse self-intersection. Suppose $\wt\alpha$ draws a transverse simple loop $\wt\Gamma_0$. Then:
\begin{enumerate}[label=(\arabic*)]
\item there exists a unique drawing component of $\gamma$ in the band defined by $\Gamma_0$;
\item if $\wt\alpha$ crosses the loop $\wt\Gamma_0$, then there exists a unique crossing component of $\wt\alpha$ in $\wt\Gamma_0$.
\item if $\wt\alpha$ does not cross $\wt\Gamma_0$, then the drawing component
contains a neighbourhood of at least one end of $J$;
\item if $\wt\alpha$ draws two non-equivalent transverse simple loops $\wt\Gamma_0$ and $\wt\Gamma_1$, then the drawing component of $\wt\alpha$ in $\wt\Gamma_0$ is on the right of the drawing component of $\wt\alpha$ in $\wt\Gamma_1$ or on its left.
\end{enumerate}
\end{lemma}

From this lemma we deduce the following:

\begin{lemma}\label{LemAdaptProp24Lct2}
Let $\wt\Gamma_0, \wt\Gamma_1,\wt\Gamma_2 : \Sp^1\to\wt S$ be three non-equivalent simple $\wt\F$-transverse loops, such that none of them is included in the bounded connected component of the complement of one other. If a transverse loop $\wt\alpha : J\to\wt S$ draws $\wt\Gamma_0, \wt\Gamma_1$ and $\wt\Gamma_2$, then $\wt\alpha$ has an $\wt\F$-transverse self-intersection. 
\end{lemma}

\begin{proof}
By contradiction, suppose that $\wt\alpha$ has no $\wt\F$-transverse self-intersection. 
Consider $J_0,J_1,J_2\subset J$ some drawing components of $\wt\alpha$ in respectively $\wt\Gamma_0$, $\wt\Gamma_1$ and $\wt\Gamma_2$.
By Lemma~\ref{LemProp24Lct2}.(4), by permuting the $\wt\Gamma_i$ if necessary, one can suppose that $\inf J_0\le \inf J_1\le \inf J_2$ and $\sup J_0\le \sup J_1\le \sup J_2$. 
More precisely, as the $\wt\Gamma_i$ are not equivalent, we have $\inf J_0 < \inf J_1$ and $\sup J_1 < \sup J_2$. By Lemma~\ref{LemProp24Lct2}.(3), this implies that $\wt\alpha$ crosses the band $B(\wt\Gamma_1)$ defined by $\wt\Gamma_1$. 
Note that, by the hypothesis that for $i\neq j$, $\wt\Gamma_i$ is not included in the bounded connected component of the complement of $\wt\Gamma_j$, we have that both $\wt\alpha(\inf J_0)$ and $\wt\alpha(\sup J_2)$ belong to the unbounded connected component of the complement of $\wt\Gamma_1$.
This implies that there is another crossing component in $B(\wt\Gamma_1)$, contradicting Lemma~\ref{LemProp24Lct2}.(2).
\end{proof}

\section{A special orbit having $\gamma$ as a tracking geodesic}\label{SecSpecClosed}

In the following, we fix a geodesic $\gamma$ as in Theorem~\ref{ThBndedDevRat}: $\gamma$ is the tracking geodesic (defined in Theorem~\ref{DefTrackGeod}) of some $\mu\in \Merg(f)$ (defined after Lemma~\ref{LemErgoRotSpeed}).
By convention, $z$ is a $\mu$-typical point and $y$ is a point whose orbit is supposed to have big deviation.

Fix a lift $\wt\gamma$ of $\gamma$ to the universal cover $\wt S$ of $S$.
We denote by $T$ the primitive deck transformation of $\wt S$ such that the geodesic $\wt\gamma$ is $T$-invariant. 

The goal of this section is to prove the following proposition:

\begin{prop}\label{LemRealizPeriod}
There exists $\mu'\in \Merg(f)$, and a $\mu'$-typical point $z'$, having one lift $\wt z'$ to $\wt S$ whose orbit stays at a finite distance from $\wt\gamma$, and such that $I^\Z_{\wt\F}(\wt z')$ is $\wt\F$-equivalent to a simple $T$-invariant transverse path $\wt \alpha_0 \subset \wt S$.
\end{prop}

By hypothesis, $\wt \gamma$ is the tracking geodesic of a point $\wt z \in \wt S$ whose projection $z$ on $S$ is typical for some $f$-ergodic measure. 

By \cite[Proposition~4.3]{pa} (based on \cite[Lemma 2.1 p.343]{zbMATH04196929}), either the orbit of $z$ stays at finite distance to $\wt\gamma$, or there exists a periodic point having $\wt\gamma$ as a tracking geodesic. Hence, by changing the measure $\mu$ by another $f$-ergodic measure $\nu$ (supported on a periodic orbit) and changing $z$ for another $\nu$-typical point if necessary, one can suppose that the orbit of $\wt z$ stays at finite distance to $\wt\gamma$.

\begin{lemma}\label{LemRecur}
Up to changing $z$ to another $\mu$-typical point, the following is true.
Let $U$ be a topological disc containing $z$. Denote $\wt U$ the lift of $U$ that contains $\wt z$.
Then there exist two sequences $(i_k)_{k\ge 1}$ and $(m_k)_{k\ge 1}$ of integers, $m_k$ tending to $+\infty$, such that $\wt f^{m_k}(\wt z) \in T^{i_k} \wt U$.
\end{lemma}

\begin{proof}[Proof of Lemma~\ref{LemRecur}]
Let $\wc z$ be the projection of $\wt z$ on the open annulus $\wt S/T$, and $\wc f$ the projection of $\wt f$ on $\wt f/T$.

Let us apply the Krylov-Bogolyubov procedure: as the orbit of $\wc z$ is bounded in the open annulus $\wt S/T$, the sequence of measures $\frac 1n \sum_{k=0}^{n-1} \delta_{\wc f^k(\wc z)}$ has a subsequence converging for the weak-$*$ topology, to an $\wc f$-invariant measure we call $\wc \mu$. One easily checks that the projection of this measure on $S$ is equal to $\mu$, hence there is a set of points of $\wc\mu$-measure 1 that are recurrent and whose projection on $S$ are $\mu$-typical; moreover the orbit of any point in the support of $\wc\mu$ is bounded.

One can replace $\wc z$ by another of these points $\wc z'$. It is recurrent in $\wt S/T$ and has $\wt\gamma$ as a tracking geodesic (it stays at a finite distance from $\wt\gamma$ as is $\supp\mu$, and has a positive speed of escape to infinity). Taking $z'$ as the projection of $\wc z'$ and $\wt z'$ as a lift of $\wc z'$ to $\wt S$ proves the lemma.
\end{proof}

From now on, we replace the point $z$ by $z'$ given by Lemma~\ref{LemRecur}.

\begin{lemma}\label{LemAlphaSimple2}
Either there exists a periodic orbit whose tracking geodesic is $\gamma$, or the transverse trajectory $I^\Z_{\wt \F}(\wt z)$ meets each leaf of $\wt \F$ at most once (and in particular is simple).
\end{lemma}

\begin{proof}
If the transverse trajectory $I^\Z_{\wt \F}(\wt z)$ crosses one leaf of $\wt\F$ twice, then it draws a simple transverse loop $\wt\Gamma$. By recurrence of $z$, the fact that the trajectory $I^\Z_{\wt \F}(\wt z)$  is proper in $\wt S$ and Lemma~\ref{Lem17Lct1}, there exist two deck transformations $T', T''$ such that $I^\Z_{\wt \F}(\wt z)$ also draws $T'\wt\Gamma$ and $T''\wt\Gamma$; moreover we can choose $T', T''$ such that $\wt\Gamma$, $T'\wt\Gamma$ and $T''\wt\Gamma$ do not intersect, and that neither of them is included in the bounded connected component of the complement of one other. 

By Lemma~\ref{LemAdaptProp24Lct2}, this implies that the transverse trajectory $I^\Z_{\wt \F}(\wt z)$ has an $\wt\F$-transverse intersection: there exist $t_1<t_2$ such that $I^{[t_1,t_2]}_{\wt \F}(\wt z)$ has an $\wt\F$-transverse intersection. Choose a small trivialising (for $\F$) neighbourhood $U$ of $z$, that is a topological disc. Denote $\wt U$ the lift of $U$ that contains $\wt z$. By Lemma~\ref{Lem17Lct1}, if $U$ is small enough, then any point $\wt x\in \wt U$ is such that $I^{[t_1-1,t_2+1]}_{\wt \F}(\wt x)$ has an $\wt\F$-transverse intersection.

By Lemma~\ref{LemRecur}, there exist $m>0$ and $i>0$ such that $\wt f^m(\wt z)\in T^i \wt U$. Hence, $I^\Z_{\wt \F}(T^i \wt z)$ and $I^\Z_{\wt \F}(\wt z)$ intersect $\wt\F$-transversally. By Theorem~\ref{thmMlct2}, this implies that there is an $f$-periodic point whose tracking geodesic is $\gamma$. 
\end{proof}

\begin{lemma}\label{LemPerPtSimple}
If there is a periodic point having $\gamma$ as a tracking geodesic, then there exists a periodic point having $\gamma$ as a tracking geodesic and whose transverse trajectory in $\wt S$ is simple.
\end{lemma}

\begin{proof}
Suppose that $\wt z$ is the lift of an $f$-periodic point having $\gamma$ as a tracking geodesic. If the path $\wt\alpha_1 := I^\Z_{\wt\F}(\wt z)$ is simple, the lemma is proved. If not, up to reparametrizing $\wt\alpha_1$ if necessary, one can suppose that there exists $i>0$ such that for any $t\in \R$ and any $k\in\Z$, we have $\wt\alpha_1(t+k) = T^{ik} \wt\alpha_1(t)$. 
Suppose that along all transverse trajectories of periodic points having $\gamma$ as a tracking geodesic, the number of intersections of $\wt\alpha_1|_{[0,1]}$ with $\wt\alpha_1$ is minimal (this number is finite because $\wt\alpha_1$ is proper and the intersections can be supposed to be locally discrete). 

Consider $t_0<t_1$ such that $\wt\alpha_1(t_0) = \wt\alpha_1(t_1)$ and such that $\wt\alpha_1|_{(t_0,t_1)}$ is simple.

We now repeat the arguments of the beginning of the proof of Lemma~\ref{LemAlphaSimple2} to get that $\wt\alpha_1$ has a transverse self-intersection at $\wt\alpha_1(t_0) = \wt\alpha_1(t_1)$. Indeed, let $\wt\Gamma$ be the 1-periodic $\wt\F$-transverse trajectory defined by $\wt\alpha_1|_{[t_0,t_1]}$.
By the fact that the trajectory of $\wt z$ is proper in $\wt S$, we deduce that there exists two deck transformations $T_1,T_2$ such that $I^\Z_{\wt \F}(\wt z)$ also draws $T_1\wt\Gamma$ and $T_2\wt\Gamma$, and we can choose $T_1,T_2$ such that $\wt\Gamma$, $T_1\wt\Gamma$ and $T_2\wt\Gamma$ do not intersect, and that neither of them is included in the bounded connected component of the complement of one other. By Lemma~\ref{LemAdaptProp24Lct2}, this implies that the transverse trajectory $\wt\alpha_1$ has a transverse self-intersection at $\wt\alpha_1(t_0) = \wt\alpha_1(t_1)$, equivalently the paths $\wt\alpha_1$ and ${T^i}\wt\alpha_1$ intersect $\wt\F$-transversally. 

By Theorem~\ref{thmMlct2}, we deduce that there is an $f$-periodic orbit whose transverse path is $T^j$-invariant (for some $j>0$) and $\F$-equivalent to $\alpha_1|_{[t_1,t_0+1]}$ on one of its fundamental domains. This path has less self-intersections than $\alpha_1$, this is a contradiction with the hypothesis that $\alpha_1$ is minimal.
\end{proof}

\begin{lemma}\label{LemPerPtSimple2}
If there is a periodic point $z$ having $\gamma$ as a tracking geodesic, then there exists a periodic point having $\gamma$ as a tracking geodesic, whose transverse trajectory in $\wt S$ is simple and that is $\wt\F$-equivalent to a $T$-invariant simple path $\wt\alpha_0\subset\wt S$.
\end{lemma}

\begin{proof}
Consider a periodic point $z$, given by Lemma~\ref{LemPerPtSimple}, having $\gamma$ as a tracking geodesic and whose transverse trajectory $I^\Z_{\wt\F}(\wt z)$ in $\wt S$ is simple. Hence, $I^\Z_{\wt\F}(\wt z)$ is $T^i$-invariant for some $i>0$; let us show that one can suppose that it is $\wt\F$-equivalent to a $T$-invariant transverse trajectory.

Let us denote $\wc S = \wt S/T$. This is an open annulus in which $\gamma$ projects into a simple closed geodesic $\wc \gamma$.
Let $\wc z$ be a lift of $z$ to $\wc S$, and $\wc\F$ the lift of the foliation $\F$ to $\wc S$. 


Consider $t_0<t_1$ such that $I^{t_0}_{\wc\F}(\wc z) = I^{t_1}_{\wc\F}(\wc z)$, and such that the path $I^{(t_0,t_1)}_{\wc\F}(\wc z)$ is simple. Define $\wc\Gamma_0$ as the periodic $\wc\F$-transverse path defined by $I^{[t_0,t_1]}_{\wc\F}(\wc z)$. This path is essential (because the transverse trajectory of $\wt z$ in $\wt S$ is simple) and simple. 
Moreover, Lemma~\ref{LemBrown} asserts that the lift $\wt\alpha_0$ of $\wc \Gamma_0$ is $T$-invariant (we know it is $T^i$-invariant for some $i>0$, and if $i\ge 2$ then it forces $\wc\Gamma_0$ not to be simple).
Denote $\wc B$ the set of leaves of $\wc \F$ met by $\wc\Gamma_0$. This is an open essential annulus of $\wc S$.

If $I^\Z_{\wc\F}(\wc z)$ stays is $\wc B$, the lemma is proved. So we suppose that $I^\Z_{\wc\F}(\wc z)$ does not stay is $\wc B$, and hence (because it is periodic) it has to get in and out of $\wc B$ an infinite number of times.

There are two possibilities, given by Remark~\ref{RemPossibBand}.

Either the path $I^{\Z}_{\wc \F}(\wc z)$ draws and visits $\wc B$, which implies that $I^{\Z}_{\wc \F}(\wc z)$ intersects $\wc\F$-transversally $TI^{\Z}_{\wc \F}(\wc z)$ (by Proposition~\ref{Prop2.1.15Lellouch}). By Theorem~\ref{thmMlct2}, this implies that there is an $f$-periodic point $\wt z'$ such that $\wt f^p(\wt z') = T\wt z'$ and whose transverse trajectory stays in $\wc B$. This proves the lemma.

Or the path $\wc\beta_0:=I^{\Z}_{\wc \F}(\wc z)$ draws and crosses $\wc B$ ($\wc\beta_0$ is periodic, and by renormalising it if necessary we suppose it is $1$-periodic). Say it crosses it from left to right. As it is periodic, it also crosses it from right to left, and this crossing component has to meet the drawing and crossing from left to right component. In particular, there exists $t_0<t_1<t_2<t_0+1$ such that $\wc\beta_0(t_0) = \wc\beta_0(t_2)\in \wc B$, that $\wc\beta_0|_{[t_0,t_2]} \subset \wc B\cup L(\wc B)$ and that $\wc\beta_0(t_1)\in L(\wc B)$.
We can then replace $\wc\beta_0$ by $\wc\beta_0|_{[t_0,t_2]}$ and repeat the above process. 
Ultimately, this process stops (because the homotopy types of the self-intersections of $\wt\beta_0$ in $\wh\dom(\F)$ are locally discrete, as $\wt\beta_0$ is simple), and we get $t'_0<t'_2$ such that $\wc\beta_0|_{[t'_0,t'_2]}$ is simple, and that $\wc\beta_0$ draws and visits the band defined by $\wc\beta_0|_{[t'_0,t'_2]}$. 
We are reduced to the previous case:  $I^{\Z}_{\wc \F}(\wc z)$ intersects $\wc\F$-transversally $TI^{\Z}_{\wc \F}(\wc z)$ (by Proposition~\ref{Prop2.1.15Lellouch}), which proves, by Theorem~\ref{thmMlct2}, that there is an $f$-periodic point $\wt z'$ such that $\wt f^p(\wt z') = T\wt z'$ and whose transverse trajectory stays in $\wc B$, proving the lemma.
\end{proof}

\begin{proof}[Proof of Proposition~\ref{LemRealizPeriod}]
By Lemma~\ref{LemPerPtSimple2}, if $\gamma$ is the tracking geodesic of a periodic point, then the proposition is proved. Hence, we suppose that it is not the tracking geodesic of a periodic point.

Take $U$ a sufficiently small neighbourhood of $z$ such that, if $\wt U$ is a lift of $U$ containing $\wt z$, then for all $\wt x$ in $\wt U$, $I^{\Z}_{\wt\F}(\wt{f}^{-1}(\wt x))$ meets $\wt\phi_{\wt z}$. The return $f^{m_1}(z)$ of the orbit in $U$ given by Lemma~\ref{LemRecur} allows us to build an approximation $\alpha_1$ of $I^{\Z}_{ \F}(z)$ associated to $T^{i_1}$: there is a transverse path $\wt\alpha_1$ in $\wt S$ that is $T^{i_1}$-periodic in the sense that for any $j\in \Z$ and $t\in\R$, we have $\wt\alpha_1(t+j) = T^{ji_1}\wt\alpha_1(t)$, and such that $\wt\alpha_1|_{[0,1]}$ is $\F$-equivalent to a subpath of $I^{[-1,m_1+1]}_{\wt \F}(\wt z)$. 
By Lemma~\ref{LemAlphaSimple2}, this path is simple in $\wt S$. Moreover, Lemma~\ref{LemBrown} implies that $\wt\alpha_1|_{[0,1]} \cap T\wt\alpha_1|_{[0,1]} \neq\emptyset$. Hence, there is a path $\wt\alpha_0 = \bigcup_{i\in\Z}T^i \wt\beta$, where $\wt\beta$ is a piece of $\wt\alpha_1$, that is simple and $T$-invariant (and not only, as $\wt\alpha_1$, $T^{i_1}$-invariant).

Let us denote by $\wt B \subset \wt S$ the set of leaves of $\wt\F$ met by $\wt\alpha_0$; it is a $T$-invariant plane of leaves. 

If $I^\Z_{\wt \F}(\wt z)$ stays in $\wt B$, then it is $\wt\F$-equivalent to a subpath of $\wt\alpha_0$. If it is not equivalent to $\wt\alpha_0$, then it has to accumulate in $\wt\alpha_0$, hence  $I^\Z_{\F}(z)$ accumulates in itself, which is impossible by Proposition~\ref{Prop2.1.3Lellouch}. 
So $I^\Z_{\wt \F}(\wt z)$ is $\wt\F$-equivalent to $\wt\alpha_0$

If $I^\Z_{\wt \F}(\wt z)$ does not stay in $\wt B$, let us first prove that it goes in and out of $\wt B$ both in positive and negative times.
There exist $t_0, t_1\in\R$ such that $I^{t_0}_{\wt \F}(\wt z) \in \wt B$ and $I^{t_1}_{\wt \F}(\wt z) \notin \wt B$. By Lemma~\ref{Lem17Lct1}, if $\wt U$ is a small enough neighbourhood of $\wt z$, then any point $\wt x\in \wt U$ is such that $I^{t_0}_{\wt \F}(\wt x)\in\wt B$, that $I^{t_1}_{\wt \F}(\wt x)\notin\wt B$, and that $\wt\alpha_0|_{[0,1]}$ is $\F$-equivalent to a subpath of $I^{[-1,m_1+1]}_{\wt \F}(\wt x)$.
So, by Lemma~\ref{LemRecur}, for any $k$, we have that $I^{t_0+m_k}_{\wt \F}(\wt z)\in T^{i_k}\wt B = \wt B$ and $I^{t_1+m_k}_{\wt \F}(\wt z)\notin T^{i_k}\wt B = \wt B$.
Hence, if $m_k$ is large enough, the path $I^{\Z}_{\wt \F}(\wt f^{m_k}(\wt z))$ goes out of $\wt B$ both in positive and negative times. 

There are two possibilities, given by Remark~\ref{RemPossibBand}. 

Either the path $I^{\Z}_{\wt \F}(\wt f^{m_k}(\wt z))$ draws and visits $\wt B$, which implies that $I^{\Z}_{\wt \F}(\wt f^{m_k}(\wt z))$ intersects $\F$-transversally $TI^{\Z}_{\wt \F}(\wt f^{m_k}(\wt z))$ (by Proposition~\ref{Prop2.1.15Lellouch}). By Theorem~\ref{thmMlct2}, this implies that there is an $f$-periodic point whose tracking geodesic is $\gamma$, which contradicts our initial hypothesis.

Or the path $I^{\Z}_{\wt \F}(\wt f^{m_k}(\wt z))$ draws and crosses $\wt B$. Because $I^{\Z}_{\wt \F}(\wt z)$ crosses $\wt B$ an infinite number of times, and because $\wt B$ is a topological plane of $\wt S$, the path $I^{\Z}_{\wt \F}(\wt z)$ also has to cross $\wt B$ in the other direction. By Proposition~\ref{Prop2.1.16Lellouch}, this implies that there exists $k\in\Z$ such that $I^{\Z}_{\wt \F}(\wt z)$ and $T^{k}I^{\Z}_{\wt \F}(\wt z)$ intersect $\wt \F$-transversally. 
By Lemma~\ref{LemAlphaSimple2}, one has $k\neq 0$ (recall that we are in the case where there is no periodic point having $\wt \gamma$ as a geodesic). Hence, one can suppose that $k\neq 0$, and apply again Theorem~\ref{thmMlct2} to get an $f$-periodic point whose tracking geodesic is $\gamma$, contradicting again our initial hypothesis.
\end{proof}

\section{Transverse intersections}\label{sec:transverseintersections}

We apply Proposition~\ref{LemRealizPeriod} to get a transverse path $\wt\alpha_0\subset \wt S$ that is simple and $T$-invariant, and $\wt\F$-equivalent to a transverse path $I^\Z_{\wt\F}(\wt z)$, for $\wt z \in\wt S$ a lift of a $\mu$-typical point $z\in S$.
Up to reparametrization, one can suppose that for any $t\in\R$ and $k\in\Z$, one has $\wt\alpha_0(t+k) = T^k\wt\alpha_0(t)$.
We denote $\wt B$ the set of leaves met by $\wt \alpha_0$; in this band the left and the right of a leaf of $\wt \F$ are well defined.

\subsection{Setting some constants}\label{SubSecConst}

Fix a leaf $\wt\phi\subset \wt B$. Let $m_0\in \Z$ such that $\wt f^{m_0}(\wt z)\in L(T^{-3}\wt\phi)$, and $m_1\in\N$ such that $\wt f^{m_0+m_1}(\wt z)\in R(T^{6}\wt\phi)$. Write $m'_0 = m_0-m_1$.

Note that the path $\wt f^{m_1}I^{[m'_0,m_0]}_{\wt\F}(\wt z)$ meets both $L(T^{-3}\wt\phi)$ and $R(T^{6}\wt\phi)$, and is included in $L(T^{-3}\wt f^{m_1}(\wt\phi))$ (see Figure~\ref{FigSetConst}, left).
Let $\wt\varphi_-$ be the piece of $T^{-3}\wt\phi$ linking the paths $\wt f^{m_1}I^{[m'_0,m_0]}_{\wt\F}(\wt z)$ and $\wt\alpha_0$, and $\wt\varphi_+$ be the piece of $T^{6}\wt\phi$ linking the paths $\wt f^{m_1}I^{[m'_0,m_0]}_{\wt\F}(\wt z)$ and $\wt\alpha_0$.
Denote $\wt\gamma = \wt\gamma_{\wt z}$ the tracking geodesic of $\wt z$, and
\begin{equation}\label{EqDefD}
D = \sup\left\{d(\wt a, \wt\gamma)\mid \wt a \in \wt f^{m_1}I^{[m'_0,m_0]}_{\wt\F}(\wt z)\cup \wt\varphi_-\cup \wt\varphi_+\cup\wt\alpha_0\right\}.
\end{equation}
Note that $D<+\infty$ because $\wt\alpha_0$ is at finite distance to $\wt\gamma$ (as it is $T$-invariant).

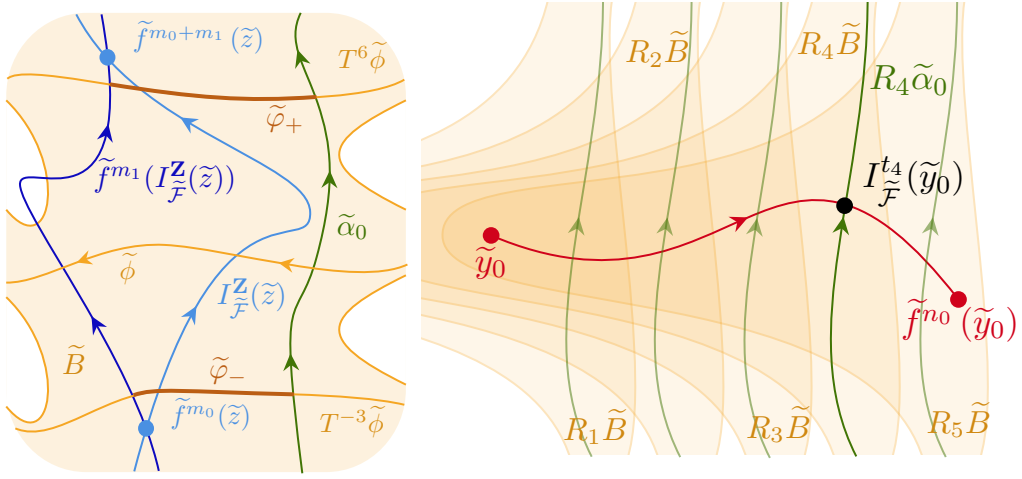
\begin{figure}
\begin{minipage}{0.38\linewidth}

\tikzset{every picture/.style={line width=0.75pt}} 

\begin{tikzpicture}[x=0.75pt,y=0.75pt,yscale=-1.1,xscale=1.05]

\draw  [draw opacity=0][fill={rgb, 255:red, 245; green, 166; blue, 35 }  ,fill opacity=0.15 ] (70,88) .. controls (70,67.01) and (87.01,50) .. (108,50) -- (222,50) .. controls (242.99,50) and (260,67.01) .. (260,88) -- (260,222) .. controls (260,242.99) and (242.99,260) .. (222,260) -- (108,260) .. controls (87.01,260) and (70,242.99) .. (70,222) -- cycle ;
\draw [color={rgb, 255:red, 245; green, 166; blue, 35 }  ,draw opacity=1 ][fill={rgb, 255:red, 255; green, 255; blue, 255 }  ,fill opacity=1 ]   (70,140) .. controls (91.83,165.46) and (100.75,118.54) .. (70,90) ;
\draw [color={rgb, 255:red, 245; green, 166; blue, 35 }  ,draw opacity=1 ][fill={rgb, 255:red, 255; green, 255; blue, 255 }  ,fill opacity=1 ]   (70,230) .. controls (91.83,255.46) and (100.75,208.54) .. (70,180) ;
\draw [color={rgb, 255:red, 245; green, 166; blue, 35 }  ,draw opacity=1 ][fill={rgb, 255:red, 255; green, 255; blue, 255 }  ,fill opacity=1 ]   (260,140) .. controls (215.75,119.04) and (219.75,109.04) .. (260,90) ;
\draw [color={rgb, 255:red, 245; green, 166; blue, 35 }  ,draw opacity=1 ][fill={rgb, 255:red, 255; green, 255; blue, 255 }  ,fill opacity=1 ]   (260,220) .. controls (215.75,199.04) and (219.75,189.04) .. (260,170) ;
\draw [color={rgb, 255:red, 65; green, 117; blue, 5 }  ,draw opacity=1 ]   (210,50) .. controls (203.12,63.5) and (211.87,73.65) .. (218.15,91.6) .. controls (224.43,109.55) and (227.09,128) .. (217.55,161.74) .. controls (208.01,195.48) and (198.75,170.04) .. (210.75,260.54) ;
\draw [shift={(208.45,66.83)}, rotate = 71.5] [fill={rgb, 255:red, 65; green, 117; blue, 5 }  ,fill opacity=1 ][line width=0.08]  [draw opacity=0] (8.04,-3.86) -- (0,0) -- (8.04,3.86) -- (5.34,0) -- cycle    ;
\draw [shift={(223.81,122.09)}, rotate = 91.27] [fill={rgb, 255:red, 65; green, 117; blue, 5 }  ,fill opacity=1 ][line width=0.08]  [draw opacity=0] (8.04,-3.86) -- (0,0) -- (8.04,3.86) -- (5.34,0) -- cycle    ;
\draw [shift={(205.12,205.67)}, rotate = 87.29] [fill={rgb, 255:red, 65; green, 117; blue, 5 }  ,fill opacity=1 ][line width=0.08]  [draw opacity=0] (8.04,-3.86) -- (0,0) -- (8.04,3.86) -- (5.34,0) -- cycle    ;
\draw [color={rgb, 255:red, 74; green, 144; blue, 226 }  ,draw opacity=1 ]   (103.29,50.69) .. controls (132.89,104.69) and (213.69,117.89) .. (214.49,141.09) .. controls (215.29,164.28) and (169.69,116.68) .. (130.85,259.65) ;
\draw [shift={(152.75,97.81)}, rotate = 30.68] [fill={rgb, 255:red, 74; green, 144; blue, 226 }  ,fill opacity=1 ][line width=0.08]  [draw opacity=0] (8.04,-3.86) -- (0,0) -- (8.04,3.86) -- (5.34,0) -- cycle    ;
\draw [shift={(160.41,184.07)}, rotate = 120.16] [fill={rgb, 255:red, 74; green, 144; blue, 226 }  ,fill opacity=1 ][line width=0.08]  [draw opacity=0] (8.04,-3.86) -- (0,0) -- (8.04,3.86) -- (5.34,0) -- cycle    ;
\draw [color={rgb, 255:red, 14; green, 10; blue, 190 }  ,draw opacity=1 ]   (115.29,50.69) .. controls (135.29,157.49) and (75.35,113.96) .. (76.55,128.75) .. controls (77.75,143.55) and (130.73,208.25) .. (142.53,259.45) ;
\draw [shift={(118.5,100.8)}, rotate = 101.57] [fill={rgb, 255:red, 14; green, 10; blue, 190 }  ,fill opacity=1 ][line width=0.08]  [draw opacity=0] (8.04,-3.86) -- (0,0) -- (8.04,3.86) -- (5.34,0) -- cycle    ;
\draw [shift={(110.51,188.66)}, rotate = 59.6] [fill={rgb, 255:red, 14; green, 10; blue, 190 }  ,fill opacity=1 ][line width=0.08]  [draw opacity=0] (8.04,-3.86) -- (0,0) -- (8.04,3.86) -- (5.34,0) -- cycle    ;
\draw [color={rgb, 255:red, 245; green, 166; blue, 35 }  ,draw opacity=1 ]   (74.32,240.01) .. controls (94.09,247.88) and (118.62,221.24) .. (149.7,222.77) .. controls (180.78,224.31) and (234.02,223.46) .. (259.25,229.2) ;
\draw [color={rgb, 255:red, 245; green, 166; blue, 35 }  ,draw opacity=1 ]   (70.5,80.13) .. controls (78.75,73.04) and (131.71,87.08) .. (168.89,88.68) .. controls (206.06,90.28) and (245.25,88.54) .. (259.25,80.54) ;
\draw  [draw opacity=0][fill={rgb, 255:red, 74; green, 144; blue, 226 }  ,fill opacity=1 ] (132.67,240) .. controls (132.67,237.91) and (134.36,236.21) .. (136.46,236.21) .. controls (138.55,236.21) and (140.25,237.91) .. (140.25,240) .. controls (140.25,242.09) and (138.55,243.79) .. (136.46,243.79) .. controls (134.36,243.79) and (132.67,242.09) .. (132.67,240) -- cycle ;
\draw [color={rgb, 255:red, 245; green, 166; blue, 35 }  ,draw opacity=1 ]   (69.98,173.34) .. controls (88.91,173.65) and (114.29,154.57) .. (145.37,156.11) .. controls (176.44,157.65) and (236.06,176.06) .. (260.72,165.06) ;
\draw [shift={(103.5,164.13)}, rotate = 340.84] [fill={rgb, 255:red, 245; green, 166; blue, 35 }  ,fill opacity=1 ][line width=0.08]  [draw opacity=0] (7.14,-3.43) -- (0,0) -- (7.14,3.43) -- (4.74,0) -- cycle    ;
\draw [shift={(199.53,164.33)}, rotate = 9.19] [fill={rgb, 255:red, 245; green, 166; blue, 35 }  ,fill opacity=1 ][line width=0.08]  [draw opacity=0] (7.14,-3.43) -- (0,0) -- (7.14,3.43) -- (4.74,0) -- cycle    ;
\draw  [draw opacity=0][fill={rgb, 255:red, 74; green, 144; blue, 226 }  ,fill opacity=1 ] (114.32,70.34) .. controls (114.32,68.25) and (116.02,66.55) .. (118.12,66.55) .. controls (120.21,66.55) and (121.91,68.25) .. (121.91,70.34) .. controls (121.91,72.43) and (120.21,74.13) .. (118.12,74.13) .. controls (116.02,74.13) and (114.32,72.43) .. (114.32,70.34) -- cycle ;
\draw [color={rgb, 255:red, 187; green, 94; blue, 20 }  ,draw opacity=1 ][line width=1.5]    (119.29,82.68) .. controls (159.61,88.79) and (183.49,90.79) .. (217.49,88.44) ;
\draw [color={rgb, 255:red, 187; green, 94; blue, 20 }  ,draw opacity=1 ][line width=1.5]    (130.54,224.83) .. controls (140.79,221.33) and (153.19,223.18) .. (206.77,224.44) ;

\draw (225.7,147.79) node [anchor=west] [inner sep=0.75pt]  [color={rgb, 255:red, 65; green, 117; blue, 5 }  ,opacity=1 ]  {$\wt{\alpha }_{0}$};
\draw (168.7,180.59) node [anchor=west] [inner sep=0.75pt]  [color={rgb, 255:red, 74; green, 144; blue, 226 }  ,opacity=1 ]  {$I_{\wt\F}^{\Z}(\wt z)$};
\draw (110,115.1) node [anchor=north west][inner sep=0.75pt]  [color={rgb, 255:red, 14; green, 10; blue, 190 }  ,opacity=1 ]  {$\wt{f}^{m_{1}}( I_{\wt\F}^{\Z}(\wt z))$};
\draw (250.92,228.29) node [anchor=north east] [inner sep=0.75pt]  [font=\small,color={rgb, 255:red, 216; green, 144; blue, 23 }  ,opacity=1 ]  {$T^{-3}\wt{\phi }$};
\draw (251.65,78.97) node [anchor=south east] [inner sep=0.75pt]  [font=\small,color={rgb, 255:red, 216; green, 144; blue, 23 }  ,opacity=1 ]  {$T^{6}\wt{\phi }$};
\draw (95.17,200.66) node [anchor=north west][inner sep=0.75pt]  [color={rgb, 255:red, 208; green, 136; blue, 16 }  ,opacity=1 ]  {$\wt{B}$};
\draw (146.6,224.56) node [anchor=north west][inner sep=0.75pt]  [font=\small,color={rgb, 255:red, 74; green, 144; blue, 226 }  ,opacity=1 ]  {$\wt{f}^{m_{0}}(\wt{z})$};
\draw (127.69,158.4) node [anchor=north] [inner sep=0.75pt]  [font=\small,color={rgb, 255:red, 216; green, 144; blue, 23 }  ,opacity=1 ]  {$\wt{\phi }$};
\draw (128.83,61.22) node [anchor=west] [inner sep=0.75pt]  [font=\small,color={rgb, 255:red, 74; green, 144; blue, 226 }  ,opacity=1 ]  {$\wt{f}^{m_{0} +m_{1}}\left(\wt{z}\right)$};
\draw (175.47,221.4) node [anchor=south] [inner sep=0.75pt]  [color={rgb, 255:red, 187; green, 94; blue, 20 }  ,opacity=1 ]  {$\wt{\varphi }_{-}$};
\draw (202.52,90.88) node [anchor=north] [inner sep=0.75pt]  [color={rgb, 255:red, 187; green, 94; blue, 20 }  ,opacity=1 ]  {$\wt{\varphi }_{+}$};

\end{tikzpicture}

\end{minipage}\hfill
\begin{minipage}{0.57\linewidth}

\tikzset{every picture/.style={line width=0.75pt}} 

\begin{tikzpicture}[x=0.75pt,y=0.75pt,yscale=-1.1,xscale=1.1]

\clip (250,30) rectangle (530,250);

\draw  [color={rgb, 255:red, 245; green, 166; blue, 35 }  ,draw opacity=0.3 ][fill={rgb, 255:red, 245; green, 166; blue, 35 }  ,fill opacity=0.1 ] (500.83,247) .. controls (523.99,214.17) and (484.21,51.73) .. (500.83,17) .. controls (488.25,17.36) and (481.58,16.69) .. (470.83,17) .. controls (454.75,142.52) and (258.78,91.87) .. (260.83,137) .. controls (262.63,176.5) and (465.5,153.52) .. (480.83,247) .. controls (495.58,246.69) and (489.25,247.36) .. (500.83,247) -- cycle ;
\draw  [color={rgb, 255:red, 245; green, 166; blue, 35 }  ,draw opacity=0.3 ][fill={rgb, 255:red, 245; green, 166; blue, 35 }  ,fill opacity=0.1 ] (460.83,247) .. controls (483.99,214.17) and (444.21,51.73) .. (460.83,17) .. controls (448.25,17.36) and (441.58,16.69) .. (430.83,17) .. controls (414.75,142.52) and (218.78,91.87) .. (220.83,137) .. controls (222.63,176.5) and (425.5,153.52) .. (440.83,247) .. controls (455.58,246.69) and (449.25,247.36) .. (460.83,247) -- cycle ;
\draw  [color={rgb, 255:red, 245; green, 166; blue, 35 }  ,draw opacity=0.3 ][fill={rgb, 255:red, 245; green, 166; blue, 35 }  ,fill opacity=0.1 ] (420.83,247) .. controls (443.99,214.17) and (404.21,51.73) .. (420.83,17) .. controls (408.25,17.36) and (401.58,16.69) .. (390.83,17) .. controls (374.75,142.52) and (178.78,91.87) .. (180.83,137) .. controls (182.63,176.5) and (385.5,153.52) .. (400.83,247) .. controls (415.58,246.69) and (409.25,247.36) .. (420.83,247) -- cycle ;
\draw  [color={rgb, 255:red, 245; green, 166; blue, 35 }  ,draw opacity=0.3 ][fill={rgb, 255:red, 245; green, 166; blue, 35 }  ,fill opacity=0.1 ] (380.83,247) .. controls (403.99,214.17) and (364.21,51.73) .. (380.83,17) .. controls (368.25,17.36) and (361.58,16.69) .. (350.83,17) .. controls (334.75,142.52) and (138.78,91.87) .. (140.83,137) .. controls (142.63,176.5) and (345.5,153.52) .. (360.83,247) .. controls (375.58,246.69) and (369.25,247.36) .. (380.83,247) -- cycle ;
\draw  [color={rgb, 255:red, 245; green, 166; blue, 35 }  ,draw opacity=0.3 ][fill={rgb, 255:red, 245; green, 166; blue, 35 }  ,fill opacity=0.1 ] (340.83,247) .. controls (363.99,214.17) and (324.21,51.73) .. (340.83,17) .. controls (328.25,17.36) and (321.58,16.69) .. (310.83,17) .. controls (294.75,142.52) and (98.78,91.87) .. (100.83,137) .. controls (102.63,176.5) and (305.5,153.52) .. (320.83,247) .. controls (335.58,246.69) and (329.25,247.36) .. (340.83,247) -- cycle ;
\draw  [draw opacity=0][fill={rgb, 255:red, 255; green, 255; blue, 255 }  ,fill opacity=1 ] (510.83,7) -- (260.83,7) -- (260.83,27) -- (510.83,27) -- cycle ;
\draw  [draw opacity=0][fill={rgb, 255:red, 255; green, 255; blue, 255 }  ,fill opacity=1 ] (510.83,237) -- (260.83,237) -- (260.83,257) -- (510.83,257) -- cycle ;
\draw  [draw opacity=0][fill={rgb, 255:red, 208; green, 2; blue, 27 }  ,fill opacity=1 ] (498.74,165.79) .. controls (498.74,163.69) and (497.04,161.99) .. (494.95,161.99) .. controls (492.85,161.99) and (491.15,163.69) .. (491.15,165.79) .. controls (491.15,167.88) and (492.85,169.58) .. (494.95,169.58) .. controls (497.04,169.58) and (498.74,167.88) .. (498.74,165.79) -- cycle ;
\draw  [draw opacity=0][fill={rgb, 255:red, 208; green, 2; blue, 27 }  ,fill opacity=1 ] (286.67,136.5) .. controls (286.67,134.41) and (284.97,132.71) .. (282.87,132.71) .. controls (280.78,132.71) and (279.08,134.41) .. (279.08,136.5) .. controls (279.08,138.59) and (280.78,140.29) .. (282.87,140.29) .. controls (284.97,140.29) and (286.67,138.59) .. (286.67,136.5) -- cycle ;
\draw [color={rgb, 255:red, 208; green, 2; blue, 27 }  ,draw opacity=1 ]   (494.95,165.79) .. controls (410.95,56.75) and (395,184.5) .. (282.87,136.5) ;
\draw [shift={(399.46,128.25)}, rotate = 156.15] [fill={rgb, 255:red, 208; green, 2; blue, 27 }  ,fill opacity=1 ][line width=0.08]  [draw opacity=0] (8.04,-3.86) -- (0,0) -- (8.04,3.86) -- (5.34,0) -- cycle    ;
\draw [color={rgb, 255:red, 65; green, 117; blue, 5 }  ,draw opacity=1 ]   (453.71,26.04) .. controls (454.21,112.54) and (419.71,180.54) .. (448.58,237.54) ;
\draw [shift={(442.76,127.96)}, rotate = 98.87] [fill={rgb, 255:red, 65; green, 117; blue, 5 }  ,fill opacity=1 ][line width=0.08]  [draw opacity=0] (8.04,-3.86) -- (0,0) -- (8.04,3.86) -- (5.34,0) -- cycle    ;
\draw  [draw opacity=0][fill={rgb, 255:red, 0; green, 0; blue, 0 }  ,fill opacity=1 ] (447.07,122.9) .. controls (447.07,120.81) and (445.37,119.11) .. (443.27,119.11) .. controls (441.18,119.11) and (439.48,120.81) .. (439.48,122.9) .. controls (439.48,124.99) and (441.18,126.69) .. (443.27,126.69) .. controls (445.37,126.69) and (447.07,124.99) .. (447.07,122.9) -- cycle ;
\draw [color={rgb, 255:red, 65; green, 117; blue, 5 }  ,draw opacity=0.5 ]   (494.33,26.33) .. controls (494.83,112.83) and (460.33,180.83) .. (489.2,237.83) ;
\draw [shift={(483.38,128.25)}, rotate = 98.87] [fill={rgb, 255:red, 65; green, 117; blue, 5 }  ,fill opacity=0.5 ][line width=0.08]  [draw opacity=0] (8.04,-3.86) -- (0,0) -- (8.04,3.86) -- (5.34,0) -- cycle    ;
\draw [color={rgb, 255:red, 65; green, 117; blue, 5 }  ,draw opacity=0.5 ]   (414.51,26.73) .. controls (415.01,113.23) and (380.51,181.23) .. (409.38,238.23) ;
\draw [shift={(403.56,128.65)}, rotate = 98.87] [fill={rgb, 255:red, 65; green, 117; blue, 5 }  ,fill opacity=0.5 ][line width=0.08]  [draw opacity=0] (8.04,-3.86) -- (0,0) -- (8.04,3.86) -- (5.34,0) -- cycle    ;
\draw [color={rgb, 255:red, 65; green, 117; blue, 5 }  ,draw opacity=0.5 ]   (373.82,26.36) .. controls (374.32,112.86) and (339.82,180.86) .. (368.69,237.86) ;
\draw [shift={(362.87,128.28)}, rotate = 98.87] [fill={rgb, 255:red, 65; green, 117; blue, 5 }  ,fill opacity=0.5 ][line width=0.08]  [draw opacity=0] (8.04,-3.86) -- (0,0) -- (8.04,3.86) -- (5.34,0) -- cycle    ;
\draw [color={rgb, 255:red, 65; green, 117; blue, 5 }  ,draw opacity=0.5 ]   (333.64,26.18) .. controls (334.14,112.68) and (299.64,180.68) .. (328.51,237.68) ;
\draw [shift={(322.69,128.1)}, rotate = 98.87] [fill={rgb, 255:red, 65; green, 117; blue, 5 }  ,fill opacity=0.5 ][line width=0.08]  [draw opacity=0] (8.04,-3.86) -- (0,0) -- (8.04,3.86) -- (5.34,0) -- cycle    ;
\draw  [draw opacity=0][fill={rgb, 255:red, 255; green, 255; blue, 255 }  ,fill opacity=1 ] (250.83,37) -- (80.83,37) -- (80.83,247) -- (250.83,247) -- cycle ;

\draw (313.93,214.1) node [anchor=north west][inner sep=0.75pt]  [font=\small,color={rgb, 255:red, 208; green, 136; blue, 16 }  ,opacity=1 ,xscale=1.2,yscale=1.2]  {$R_{1}\wt{B}$};
\draw (342.03,41.3) node [anchor=north west][inner sep=0.75pt]  [font=\small,color={rgb, 255:red, 208; green, 136; blue, 16 }  ,opacity=1 ,xscale=1.2,yscale=1.2]  {$R_{2}\wt{B}$};
\draw (397.03,213.1) node [anchor=north west][inner sep=0.75pt]  [font=\small,color={rgb, 255:red, 208; green, 136; blue, 16 }  ,opacity=1 ,xscale=1.2,yscale=1.2]  {$R_{3}\wt{B}$};
\draw (420.03,37.6) node [anchor=north west][inner sep=0.75pt]  [font=\small,color={rgb, 255:red, 208; green, 136; blue, 16 }  ,opacity=1 ,xscale=1.2,yscale=1.2]  {$R_{4}\wt{B}$};
\draw (478.83,210.9) node [anchor=north west][inner sep=0.75pt]  [font=\small,color={rgb, 255:red, 208; green, 136; blue, 16 }  ,opacity=1 ,xscale=1.2,yscale=1.2]  {$R_{5}\wt{B}$};
\draw (282.87,139.9) node [anchor=north] [inner sep=0.75pt]  [color={rgb, 255:red, 208; green, 2; blue, 27 }  ,opacity=1 ,xscale=1.2,yscale=1.2]  {$\wt{y}_{0}$};
\draw (495.77,163.87) node [anchor=north] [inner sep=0.75pt]  [color={rgb, 255:red, 208; green, 2; blue, 27 }  ,opacity=1 ,xscale=1.2,yscale=1.2]  {$\wt{f}^{n_{0}}\left(\wt{y}_{0}\right)$};
\draw (452.85,65.19) node [anchor=west] [inner sep=0.75pt]  [color={rgb, 255:red, 65; green, 117; blue, 5 }  ,opacity=1 ,xscale=1.2,yscale=1.2]  {$R_{4}\wt{\alpha }_{0}$};
\draw (449.77,124.6) node [anchor=south west] [inner sep=0.75pt]  [xscale=1.2,yscale=1.2]  {$I_{\wt\F}^{t_{4}}(\wt{y}_{0})$};

\end{tikzpicture}
\end{minipage}
\caption{Left: the objects used in Paragraph~\ref{SubSecConst}.
Right: the configuration of Proposition~\ref{PropBndedDevRatCase2}.}\label{FigSetConst}
\end{figure}

\subsection{First case: the trajectory stays in different copies of $\wt B$}

The configuration of the following proposition is depicted in Figure~\ref{FigSetConst}, right.

\begin{prop}\label{PropBndedDevRatCase2}
Suppose that there exist 5 different copies of $\wt B$, denoted by $(R_i \wt B)_{1\le i\le 5}$ (with $R_i\in\G$), such that the following is true. First, we suppose that the sets $R_i V_{D}(\wt\gamma)$ are pairwise disjoint and have the same orientation.
Second, we suppose that there exist $n_0\ge m_1$ and, for all $i$, some time $t_i\in[0,n_0]$ such that $I^{t_i}_{\wt\F}(\wt y_0)\in R_i\wt\alpha_0$, and that either for all $1\le i\le 5$ we have $I^{[0,t_i]}_{\wt\F}(\wt y_0)\subset R_i\wt B$, or for all $i$ we have $I^{[t_i, n_0]}_{\wt\F}(\wt y_0)\subset R_i\wt B$.


Then there exists an $f$-periodic orbit having a lift whose tracking geodesic crosses both $R_2\wt\gamma_{\wt z}$ and $R_3\wt\gamma_{\wt z}$.

More precisely, there exists an $f$-periodic point $p$ of period $n_0+m_1$ (where $m_1$ is the constant independent of $y_0$ and $n_0$ defined in Subsection~\ref{SubSecConst}) having a lift $\wt p$ satisfying $\wt f^{n_0+m_1}(\wt p) = R_3T^3R_2^{-1} \wt p$.

Finally, there exists a constant $d_0>0$ depending only on $z$ (and neither on $y_0$ nor on $n_0$) such that the tracking geodesic $\gamma_p$ of $p$ is freely homotopic to the concatenation $I^{[t_2,t_3]}_{\F}( y_0)\delta$, where $\diam(\wt\delta)\le d_0$ (with $\wt\delta$ a lift of $\delta$ to $\wt S$). 
\end{prop}

Similarly, one can build an $f$-periodic orbit having a lift whose tracking geodesic crosses both $R_3\wt\gamma_{\wt z}$ and $R_4\wt\gamma_{\wt z}$.

This proposition will be generalised in Proposition~\ref{PropBndedDevRatCase2b}, to adapt it in the case where the orbit segment $I^{[0,n_0]}_{\wt\F}(\wt y_0)$ crosses lifts of two different typical points $z_1$ and $z_2$.
As the proof of Proposition~\ref{PropBndedDevRatCase2} is already quite involved, we start with its proof, and then explain how to adapt it to get Proposition~\ref{PropBndedDevRatCase2b}.
\bigskip

This subsection is devoted to the proof of Proposition~\ref{PropBndedDevRatCase2}. We restrict ourselves to the case where for all $i$ we have $I^{[0,t_i]}_{\wt\F}(\wt y_0)\subset R_i\wt B$, the other case being symmetric (replacing $f$ by $f^{-1}$, and swapping the order on the $R_i\wt B$).

Start by fixing a representative of $I^{[0,n_0]}_{\wt\F}(\wt y_0)$. First, by replacing the time $t_i$ by a smaller one if necessary, one can suppose that for $t\in(0,t_i)$, we have $I^{t}_{\wt\F}(\wt y_0)\notin R_i\wt\alpha_0$.
Without loss of generality, by permuting the $R_i$ if necessary, we can suppose that the times $(t_i)$ are increasing in $i$.
Let us denote by $\wt U$ the set of leaves met by $I^{[t_{1},t_5]}_{\wt\F}(\wt y_0)$; by hypothesis it is included in $R_5\wt B$. Because $R_5\wt B$ is a trivially foliated topological plane (by Proposition~\ref{LemRealizPeriod}), on which the space of leaves is naturally identified with $\wt\alpha_0$ (and hence, identified with $\R$), the set of leaves of $R_5\wt B$ met by $I^{[t_{1},t_5]}_{\wt\F}(\wt y_0)$ is an interval. In particular, the trajectory $I^{[t_{1},t_5]}_{\wt\F}(\wt y_0)$ is simple in $\wt S$. 


\begin{lemma}\label{LemCantCrossLotLeaves}
For $1\le i\le 4$, the set of leaves of $R_i\wt B$ met by $I^{[0,n_0]}_{\wt\F}(\wt y_0)$ cannot cover a whole fundamental domain (for $R_iTR_i^{-1}$) of $R_i\wt\alpha_0$ (where, again, the set of leaves of $R_i\wt B$ is naturally identified with $R_i\wt\alpha_0$).

Similarly, for $1\le i,j\le 5$, $i\neq j$, the set of leaves of $R_i\wt B$ intersecting $R_j\wt B$ cannot cover a whole fundamental domain (for $R_iTR_i^{-1}$) of $R_i\wt\alpha_0$.
\end{lemma}

This lemma is a direct consequence of :

\begin{lemma}[{\cite[Lemma 3.4]{guiheneuf2023area}}]\label{LemEquivSubpath}
Let $ \Gamma : \Sp^1\to S$ be a simple transverse loop and $\wh\alpha : \R\to\wh\dom (\F)$ a lift of $\Gamma$.  Let $T\in\G$ be the deck transformation associated to $\wh \alpha$. Suppose that there exist a deck transformation $T'\in\G$ and $a\in\R$ such that $\wh\alpha|_{[a,a+1]}$ is equivalent to a subpath of $T'\wh\alpha$. Then $\wh\alpha|_{[a,a+1)} \cap T'\wh\alpha \neq \emptyset$. 
\end{lemma}

\begin{proof}[Proof of Lemma~\ref{LemCantCrossLotLeaves}]
We prove the first part of the lemma, the second being similar.
Suppose that the set of leaves of $R_i\wt B$ met by $I^{[0,n_0]}_{\wt\F}(\wt y_0)$ covers a whole fundamental domain of $R_i\wt\alpha_0$. This means that there exists $t\in\R$ such that $R_i\wt\alpha_0|_{[t,t+1]}$ is $\wt\F$-equivalent to a subpath of $I^{[0, n_0]}_{\wt\F}(\wt y_0)$, which itself is $\wt\F$-equivalent to a subpath of $R_5\wt\alpha_0$. 
Proposition~\ref{LemRealizPeriod} allows us to apply Lemma~\ref{LemEquivSubpath} (what happens in the lifts $\wt S$ and $\wh\dom (\F)$ is identical because everything lies inside $R_5\wt B$, which is a plane of leaves), which implies that $R_i\wt\alpha_0 \cap R_5\wt\alpha_0\neq\emptyset$, a contradiction with the hypothesis that the sets $(R_i V_D(\wt\gamma))_i$ are pairwise disjoint.
\end{proof}

Write $\wt\phi_1 = \wt\phi_{\wt y_0}$. By hypothesis, this leaf meets all the paths $R_i\wt\alpha_0$ for $1\le i\le 5$. For all $i$, let $s_i\in\R$ be such that $R_i\wt\alpha_0(s_i) \in\wt\phi_1$.

Lemma~\ref{LemCantCrossLotLeaves} implies that for all $1\le i\le 4$ and all $1\le j\le 5$ with $i\neq j$, the path $I^{[t_1,t_i]}(\wt y_0)$ is $\wt\F$-equivalent to a subpath of $R_i\wt\alpha_0|_{[s_i-1,s_i+1]}$, and that for $j\neq i$, neither $R_i\wt\alpha_0|_{[s_i-1,s_i]}$ nor $R_i\wt\alpha_0|_{[s_i,s_i+1]}$ are $\wt \F$-equivalent to a subpath of $R_j\wt\alpha_j$ or to a subpath of $I^{[t_1,t_5]}_{\wt \F}(\wt y_0)$.
For any $i$, let $k_i\in\Z$ be such that, in $R_i\wt B$, we have that (recall that $\wt\phi$ is a leaf  of $\wt B$ that was fixed in Subsection~\ref{SubSecConst})
\begin{equation}\label{EqLiesBetween}
R_i\wt\alpha_0|_{[s_i-1,s_i+1]}\textrm{ lies between }R_i T^{ k_i}\wt\phi\textrm{ and }R_i T^{ k_i+3}\wt\phi.
\end{equation}
%
%

Denote $\wt\alpha_0^-$ the piece of $\wt\alpha_0$ linking $\alpha(\wt\alpha_0)$ to the leaf $\wt\phi$, and $\wt\alpha_0^+$ the piece of $\wt\alpha_0$ linking the leaf $\wt\phi$ to $\omega(\wt\alpha_0)$. Let $\wt\delta$ the piece of $\wt f^{m_1}I^{[m'_0,m_0]}_{\wt\F}(\wt z)$ linking $T^{-3}\wt\phi$ to $T^6\wt\phi$, and denote (the objects are defined in Subsection~\ref{SubSecConst}, see also Figure~\ref{FigSetConst}, left)
\[\wt\sigma = T^{-3}\wt\alpha_0^- \cup \wt\varphi_- \cup \wt\delta \cup \wt\varphi_+ \cup T^{6}\wt\alpha_0^+.\]
This is a path included in $V_D(\wt\gamma)$ which separates $\wt S$ into two connected components denoted $L(\wt\sigma)$ and $R(\wt\sigma)$ according to the orientation of $\wt\sigma$. Moreover the paths $R_iT^{k_i}\wt\sigma$ are well ordered: for $i\le j$, one has $L(R_iT^{k_i}\wt\sigma)\subset L(R_jT^{k_j}\wt\sigma)$ and $R(R_jT^{k_j}\wt\sigma)\subset R(R_iT^{k_i}\wt\sigma)$.
Finally, because the $V_D(R_i\wt\gamma)$ are pairwise disjoint and have the same orientation,
\begin{align}\label{EqYLeftRight}
\begin{split}
I^{t_{1}}_{\wt \F}(\wt y_0) & \in L\Big(R_2T^{k_2}\wt\sigma\Big)\\
I^{t_{5}}_{\wt \F}(\wt y_0) & \in R\Big(R_4T^{k_4}\wt\sigma\Big).
\end{split}
\end{align}

Let us continue in the universal cover $\wh\dom(\F)$ of $\dom(\F)$. 
Let $\wh y_0$, $\wh f$ and $\wh \F$ be lifts of respectively $\wt y_0$, $\wt f$ and $\wt\F$ to $\wh\dom(\F)$. 
Recall that $t_i$ was chosen so that $I^{t_i}_{\wt\F}(\wt y_0)\in R_i\wt\alpha_0$ and that for $t\in(0,t_i)$, we have $I^{t}_{\wt\F}(\wt y_0)\notin R_i\wt\alpha_0$. 
Let $\wh\alpha_0$ be a lift of $\wt\alpha_0$ to $\wh\dom(\F)$, and denote $\wh B$ the set of leaves of $\wh\F$ met by $\wh\alpha_0$. Let $\wh\phi$ be a lift of $\wt\phi$ belonging to $\wh B$, and $T$ a deck transformation of $\wh\dom(\F)$ lifting the deck transformation $T$ of $S$ and stabilizing $\wh \alpha_0$. Denote $\wh z$ a lift of $\wt z$ to $\wh\dom (\F)$ such that $I^\Z_{\wh\F}(\wh z)$ is $\wh\F$-equivalent to $\wh\alpha_0$. All these lifts naturally define a lift $\wh\sigma$ of $\wt\sigma$.
For any $i$, let $R_i$ be a deck transformation of $\wh\dom(\F)$ such that $I^{t_i}_{\wh\F}(\wh y_0)\in R_i\wh\alpha_0$. 

Define
\begin{equation}\label{EqDefC0}
\wh C_0 = T^{-3}\big(L(\wh\phi) \cup \wh\phi\big) \cup \wh f^{m_1}I^{[m'_0,m_0]}_{\wh\F}(\wh z) \cup T^6\big(R(\wh\phi) \cup \wh\phi\big).
\end{equation}

The configuration of the next lemma is depicted in Figure~\ref{FigLemConsLemYLeftRight}.

\begin{lemma}\label{LemConsLemYLeftRight}
\begin{align*}
\begin{split}
I^{t_{1}}_{\wh \F}(\wh y_0) & \in L\big(R_2T^{k_2}\wh C_0\big) \cap L\big(R_4T^{k_4}\wh C_0\big),\\
I^{t_{5}}_{\wh \F}(\wh y_0) & \in R\big(R_2T^{k_2}\wh C_0\big) \cap R\big(R_4T^{k_4}\wh C_0\big).
\end{split}
\end{align*}
\end{lemma}

\begin{figure}
\begin{center}

\tikzset{every picture/.style={line width=0.75pt}} 

\begin{tikzpicture}[x=0.75pt,y=0.75pt,yscale=-1,xscale=1]

\draw [color={rgb, 255:red, 245; green, 166; blue, 35 }  ,draw opacity=1 ][fill={rgb, 255:red, 245; green, 166; blue, 35 }  ,fill opacity=0.15 ]   (110.4,50) .. controls (135.32,95.21) and (177.32,107.21) .. (190.4,50) ;
\draw [color={rgb, 255:red, 245; green, 166; blue, 35 }  ,draw opacity=1 ][fill={rgb, 255:red, 245; green, 166; blue, 35 }  ,fill opacity=0.15 ]   (110.4,250) .. controls (112.32,168.71) and (208.32,228.21) .. (190.4,250) ;
\draw [color={rgb, 255:red, 245; green, 166; blue, 35 }  ,draw opacity=1 ][fill={rgb, 255:red, 245; green, 166; blue, 35 }  ,fill opacity=0.15 ]   (249.4,50) .. controls (274.32,95.21) and (316.32,107.21) .. (329.4,50) ;
\draw [color={rgb, 255:red, 245; green, 166; blue, 35 }  ,draw opacity=1 ][fill={rgb, 255:red, 245; green, 166; blue, 35 }  ,fill opacity=0.15 ]   (249.4,250) .. controls (251.32,168.71) and (347.32,228.21) .. (329.4,250) ;
\draw [color={rgb, 255:red, 14; green, 10; blue, 190 }  ,draw opacity=1 ]   (150,60) .. controls (127.75,158.21) and (172.75,155.71) .. (150,240) ;
\draw [shift={(150,240)}, rotate = 105.1] [color={rgb, 255:red, 14; green, 10; blue, 190 }  ,draw opacity=1 ][fill={rgb, 255:red, 14; green, 10; blue, 190 }  ,fill opacity=1 ][line width=0.75]      (0, 0) circle [x radius= 2.34, y radius= 2.34]   ;
\draw [shift={(147.87,145.73)}, rotate = 76.47] [fill={rgb, 255:red, 14; green, 10; blue, 190 }  ,fill opacity=1 ][line width=0.08]  [draw opacity=0] (8.04,-3.86) -- (0,0) -- (8.04,3.86) -- (5.34,0) -- cycle    ;
\draw [shift={(150,60)}, rotate = 102.77] [color={rgb, 255:red, 14; green, 10; blue, 190 }  ,draw opacity=1 ][fill={rgb, 255:red, 14; green, 10; blue, 190 }  ,fill opacity=1 ][line width=0.75]      (0, 0) circle [x radius= 2.34, y radius= 2.34]   ;
\draw [color={rgb, 255:red, 14; green, 10; blue, 190 }  ,draw opacity=1 ]   (289,60) .. controls (266.75,158.21) and (311.75,155.71) .. (289,240) ;
\draw [shift={(289,240)}, rotate = 105.1] [color={rgb, 255:red, 14; green, 10; blue, 190 }  ,draw opacity=1 ][fill={rgb, 255:red, 14; green, 10; blue, 190 }  ,fill opacity=1 ][line width=0.75]      (0, 0) circle [x radius= 2.34, y radius= 2.34]   ;
\draw [shift={(286.87,145.73)}, rotate = 76.47] [fill={rgb, 255:red, 14; green, 10; blue, 190 }  ,fill opacity=1 ][line width=0.08]  [draw opacity=0] (8.04,-3.86) -- (0,0) -- (8.04,3.86) -- (5.34,0) -- cycle    ;
\draw [shift={(289,60)}, rotate = 102.77] [color={rgb, 255:red, 14; green, 10; blue, 190 }  ,draw opacity=1 ][fill={rgb, 255:red, 14; green, 10; blue, 190 }  ,fill opacity=1 ][line width=0.75]      (0, 0) circle [x radius= 2.34, y radius= 2.34]   ;
\draw [color={rgb, 255:red, 208; green, 2; blue, 27 }  ,draw opacity=1 ]   (416.42,144.71) .. controls (354.01,150.18) and (155.5,108) .. (63.92,159.71) ;
\draw [shift={(63.92,159.71)}, rotate = 150.55] [color={rgb, 255:red, 208; green, 2; blue, 27 }  ,draw opacity=1 ][fill={rgb, 255:red, 208; green, 2; blue, 27 }  ,fill opacity=1 ][line width=0.75]      (0, 0) circle [x radius= 2.34, y radius= 2.34]   ;
\draw [shift={(242.83,134.49)}, rotate = 182.24] [fill={rgb, 255:red, 208; green, 2; blue, 27 }  ,fill opacity=1 ][line width=0.08]  [draw opacity=0] (8.04,-3.86) -- (0,0) -- (8.04,3.86) -- (5.34,0) -- cycle    ;
\draw [shift={(416.42,144.71)}, rotate = 174.99] [color={rgb, 255:red, 208; green, 2; blue, 27 }  ,draw opacity=1 ][fill={rgb, 255:red, 208; green, 2; blue, 27 }  ,fill opacity=1 ][line width=0.75]      (0, 0) circle [x radius= 2.34, y radius= 2.34]   ;

\draw (145,109.7) node [anchor=east] [inner sep=0.75pt]  [font=\small,color={rgb, 255:red, 14; green, 10; blue, 190 }  ,opacity=1 ,xscale=1.2,yscale=1.2]  {$R_{2} T^{k_{2}}\wh{f}^{m_{1}}\big( I_{\wh\F}^{[ m'_{0} ,m_{0}]}(\wh{z})\big)$};
\draw (116.4,59.33) node [anchor=north east] [inner sep=0.75pt]  [font=\small,color={rgb, 255:red, 245; green, 166; blue, 35 }  ,opacity=1 ,xscale=1.2,yscale=1.2]  {$R_{2} T^{k_{2} +6}\wh{\phi }$};
\draw (111.4,232.27) node [anchor=south east] [inner sep=0.75pt]  [font=\small,color={rgb, 255:red, 245; green, 166; blue, 35 }  ,opacity=1 ,xscale=1.2,yscale=1.2]  {$R_{2} T^{k_{2} -3}\wh{\phi }$};
\draw (283.48,109.27) node [anchor=west] [inner sep=0.75pt]  [font=\small,color={rgb, 255:red, 14; green, 10; blue, 190 }  ,opacity=1 ,xscale=1.2,yscale=1.2]  {$R_{4} T^{k_{4}}\wh{f}^{m_{1}}\big( I_{\wh\F}^{[ m'_{0} ,m_{0}]}(\wh{z})\big)$};
\draw (324.52,56.33) node [anchor=north west][inner sep=0.75pt]  [font=\small,color={rgb, 255:red, 245; green, 166; blue, 35 }  ,opacity=1 ,xscale=1.2,yscale=1.2]  {$R_{4} T^{k_{4} +6}\wh{\phi }$};
\draw (326.2,227.67) node [anchor=south west] [inner sep=0.75pt]  [font=\small,color={rgb, 255:red, 245; green, 166; blue, 35 }  ,opacity=1 ,xscale=1.2,yscale=1.2]  {$R_{4} T^{k_{4} -3}\wh{\phi }$};
\draw (203.09,142.77) node [anchor=north] [inner sep=0.75pt]  [font=\small,color={rgb, 255:red, 208; green, 2; blue, 27 }  ,opacity=1 ,xscale=1.2,yscale=1.2]  {$I_{\wh\F}^{[ t_{1} ,t_{5}]}(\wh{y_0})$};
\draw (63.92,163.11) node [anchor=north] [inner sep=0.75pt]  [font=\small,color={rgb, 255:red, 208; green, 2; blue, 27 }  ,opacity=1 ,xscale=1.2,yscale=1.2]  {$I_{\wh\F}^{t_{1}}(\wh{y_0})$};
\draw (420,150) node [anchor=north] [inner sep=0.75pt]  [font=\small,color={rgb, 255:red, 208; green, 2; blue, 27 }  ,opacity=1 ,xscale=1.2,yscale=1.2]  {$I_{\wh\F}^{t_{5}}(\wh{y_0})$};
\end{tikzpicture}
\caption{The configuration of Lemma~\ref{LemConsLemYLeftRight}.}\label{FigLemConsLemYLeftRight}
\end{center}
\end{figure}
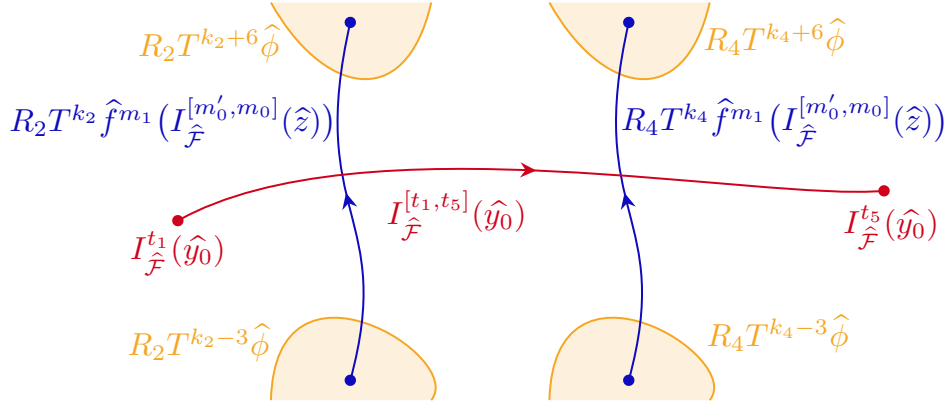

\begin{proof}
We make the proof of the lemma for the translates by $R_2T^{k_2}$, the one for the translates by $R_4T^{k_4}$ is identical.

By \eqref{EqYLeftRight} and the definition of $t_2$, we have $I^{[t_{1}, t_2)}_{\wt \F}(\wt y_0) \subset L\big(R_2T^{k_2}\wt\sigma\big)$. By the very definition of the lift of $R_2$ to $\wh\dom(\F)$, this implies that $I^{[t_{1}, t_2)}_{\wh \F}(\wh y_0) \subset L(R_2T^{k_2}\wh\sigma)$. By definition of $k_2$, that follows from the definition of $s_2$, allowed by Lemma~\ref{LemCantCrossLotLeaves}, and by the fact that all objects involved are included in the plane $R_2\wt B$, this implies that
\[I^{[t_{1}, t_2)}_{\wh \F}(\wh y_0) \cap R_2\Big( T^{k_2-3}\big(L(\wh\phi) \cup \wh\phi\big) \cup T^{k_2+6}\big(R(\wh\phi)\cup \wh\phi\big)\Big) = \emptyset.\]
As $ T^{-3}\wt\alpha_0^- \cup \wt\varphi_- \subset T^{-3}\big(L(\wh\phi) \cup \wh\phi\big)$ and $T^{6}\wt\alpha_0^+ \cup \wt\varphi_+  \subset  T^6\big(R(\wh\phi)\cup \wh\phi\big)$, this proves that $I^{t_{1}}_{\wh \F}(\wh y_0) \in L\big(R_2T^{k_2}\wh C_0\big)$.
\bigskip

By \eqref{EqYLeftRight}, we know that there exists $t'_2\in [t_{1},t_5)$ such that and $I^{t'_2}_{\wt \F}(\wt y_0) \in R_2T^{k_2}\wt\sigma$ and $I^{(t'_2, t_5]}_{\wt \F}(\wt y_0) \subset R(R_2T^{k_2}\wt\sigma)$. We have three cases.

Either $I^{[t_{1}, t_5)}_{\wt \F}(\wt y_0)\subset R_2\wt B$. Notice that $R_2T^{k_2}\wt\sigma\subset R_2\wt B$. But as $R_2\wt B$ is a topological plane of leaves, the projection $\wh B \to \wt B$ is injective, in particular the notion of left/right passes to the lift. 
This shows that $I^{(t'_2, t_5]}_{\wh \F}(\wh y_0)\subset R(R_2T^{k_2}\wh\sigma)$.
But the same argument of injectivity of the projection, combined with the definition of $k_2$, also proves that 
\[I^{(t_{1}, t_5]}_{\wh \F}(\wh y_0) \cap R_2\Big( T^{k_2-3}\big(L(\wh\phi) \cup \wh\phi\big) \cup T^{k_2+6}\big(R(\wh\phi)\cup \wh\phi\big)\Big) = \emptyset.\]
As $ T^{-3}\wh\alpha_0^- \cup \wh\varphi_- \subset T^{-3}\big(L(\wh\phi) \cup \wh\phi\big)$ and $\wh\varphi_+ \cup T^{6}\wh\alpha_0^+ \subset  T^6\big(R(\wh\phi)\cup \wh\phi\big)$, this proves that $I^{t_{5}}_{\wh \F}(\wh y_0) \in R\big(R_2T^{k_2}\wh C_0\big)$.

Or $I^{[t_{1}, t_5)}_{\wt \F}(\wt y_0)\cap L( R_2\wt B)\neq\emptyset$. Note that the sets $R_i\wt B$ are topological planes that satisfy: for $i\le j$, we have $L(R_i\wt B)\subset L(R_j\wt B)$ and $R(R_j\wt B)\subset R(R_i\wt B)$. Because $I^{[t_1,t_5]}_{\wt\F}(\wt y_0)\subset R_5\wt B$, this implies that for all $1\le i\le 5$ we have $I^{[t_1,t_5]}_{\wt\F}(\wt y_0)\cap L(R_i\wt B) = \emptyset$. Hence this second case is impossible.

Or $I^{[t_{1}, t_5)}_{\wt \F}(\wt y_0)\cap R( R_2\wt B)\neq\emptyset$. Let $t''_2$ be the smallest real bigger than $t_0$ such that $I^{t''_2}_{\wt \F}(\wt y_0)\in R( R_2\wt B)$. By having chosen $t_2$ as the first intersection time of the trajectory $I^{[0,n_0]}_{\wt\F}(\wt y_0)$ with $R_2\wt\alpha_0$, one can suppose that $t_2\le t''_2$. This implies that $I^{t''_2}_{\wh \F}(\wh y_0)\in R( R_2\wh B)$. Denoting $\wh U$ the set of leaves of $\wh \F$ crossing $I^{[0,n_0]}_{\wh\F}(\wh y_0)$, we have $I^{t''_2}_{\wh \F}(\wh y_0)\in \wh U$ that is disjoint from $R_2\big(T^{k_2-3}\wh\phi \cup T^{k_2+6}\wh\phi\big)$ (by definition of $k_2$). This implies that $I^{t''_2}_{\wh \F}(\wh y_0)\in R(R_2T^{k_2}\wh C_0)$, proving the lemma.
\end{proof}

\begin{lemma}\label{LemExistTrans26}
There exists an $\wh f$-admissible transverse path $\wh\beta$ of order $n_0+m_1$ linking $R_2T^{ k_2-3}\wh \phi$ to $R_4T^{ k_4+6}\wh \phi$.
\end{lemma}

\begin{proof}
The idea of the proof is depicted in Figure~\ref{FigLemExistTrans26}.
Let us define 
\[\wh C_{k}^- = T^k L(\wh\phi) \qquad\textrm{and}\qquad
\wh C_{k}^+ = T^k R(\wh\phi).\]
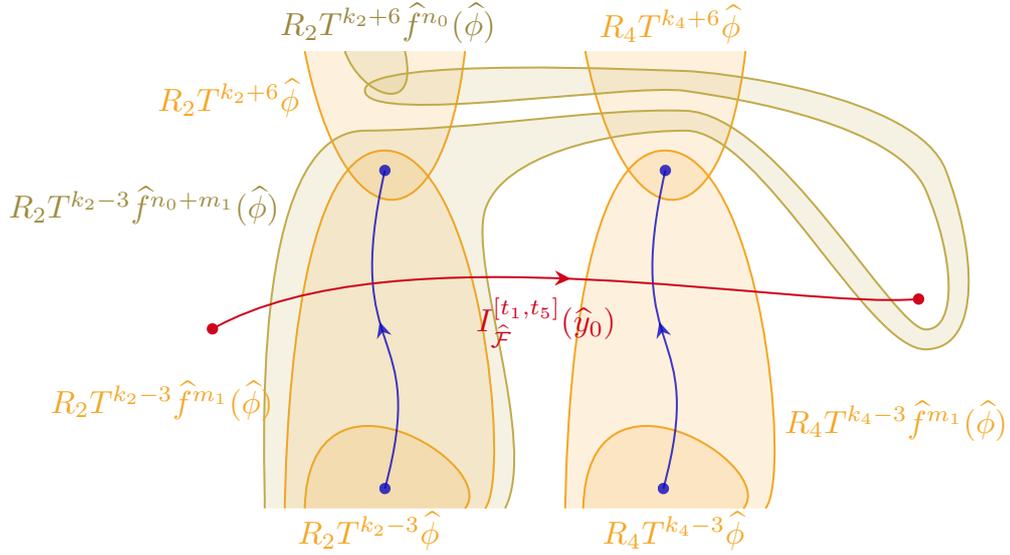
\begin{figure}
\begin{center}

\tikzset{every picture/.style={line width=0.75pt}} 

\begin{tikzpicture}[x=0.75pt,y=0.75pt,yscale=-1,xscale=1]

\draw [color={rgb, 255:red, 245; green, 166; blue, 35 }  ,draw opacity=1 ][fill={rgb, 255:red, 245; green, 166; blue, 35 }  ,fill opacity=0.15 ]   (130,270) .. controls (131.92,188.71) and (227.92,248.21) .. (210,270) ;
\draw [color={rgb, 255:red, 245; green, 166; blue, 35 }  ,draw opacity=1 ][fill={rgb, 255:red, 245; green, 166; blue, 35 }  ,fill opacity=0.15 ]   (269,270) .. controls (270.92,188.71) and (366.92,248.21) .. (349,270) ;
\draw [color={rgb, 255:red, 245; green, 166; blue, 35 }  ,draw opacity=1 ][fill={rgb, 255:red, 245; green, 166; blue, 35 }  ,fill opacity=0.15 ]   (270,40) .. controls (289.21,138.35) and (338.81,141.15) .. (350,40) ;
\draw [color={rgb, 255:red, 245; green, 166; blue, 35 }  ,draw opacity=1 ][fill={rgb, 255:red, 245; green, 166; blue, 35 }  ,fill opacity=0.15 ]   (130,40) .. controls (149.21,138.35) and (198.81,141.15) .. (210,40) ;
\draw [color={rgb, 255:red, 245; green, 166; blue, 35 }  ,draw opacity=1 ][fill={rgb, 255:red, 245; green, 166; blue, 35 }  ,fill opacity=0.15 ]   (260,270) .. controls (260.48,249.51) and (262.65,89.65) .. (310,90) .. controls (357.35,90.35) and (373.4,253.7) .. (360,270) ;
\draw [color={rgb, 255:red, 245; green, 166; blue, 35 }  ,draw opacity=1 ][fill={rgb, 255:red, 245; green, 166; blue, 35 }  ,fill opacity=0.15 ]   (120,270) .. controls (120.48,249.51) and (122.65,89.65) .. (170,90) .. controls (217.35,90.35) and (233.4,253.7) .. (220,270) ;
\draw [color={rgb, 255:red, 193; green, 171; blue, 75 }  ,draw opacity=1 ][fill={rgb, 255:red, 193; green, 171; blue, 75 }  ,fill opacity=0.15 ]   (110,270) .. controls (110.48,249.51) and (100.79,80.45) .. (160,80) .. controls (219.21,79.55) and (273.99,68.45) .. (320,70) .. controls (366.01,71.55) and (423.72,180.05) .. (440,180) .. controls (456.28,179.95) and (454.12,142.85) .. (440,110) .. controls (425.88,77.15) and (348.12,64.05) .. (320,60) .. controls (291.88,55.95) and (161.05,77.65) .. (160,60) .. controls (158.95,42.35) and (306.41,49.76) .. (320,50) .. controls (333.59,50.24) and (435.48,61.55) .. (450,100) .. controls (464.52,138.45) and (469.08,189.55) .. (440,190) .. controls (410.92,190.45) and (358.28,79.95) .. (320,80) .. controls (281.72,80.05) and (228.28,92.35) .. (220,120) .. controls (211.72,147.65) and (246.68,237.15) .. (230,270) ;
\draw [color={rgb, 255:red, 193; green, 171; blue, 75 }  ,draw opacity=1 ][fill={rgb, 255:red, 193; green, 171; blue, 75 }  ,fill opacity=0.15 ]   (150,40) .. controls (157.75,62.35) and (187.75,74.35) .. (180,40) ;
\draw [color={rgb, 255:red, 14; green, 10; blue, 190 }  ,draw opacity=0.8 ]   (170,100) .. controls (147.75,198.21) and (192.75,175.71) .. (170,260) ;
\draw [shift={(170,260)}, rotate = 105.1] [color={rgb, 255:red, 14; green, 10; blue, 190 }  ,draw opacity=0.8 ][fill={rgb, 255:red, 14; green, 10; blue, 190 }  ,fill opacity=0.8 ][line width=0.75]      (0, 0) circle [x radius= 2.34, y radius= 2.34]   ;
\draw [shift={(167.4,176.16)}, rotate = 73.38] [fill={rgb, 255:red, 14; green, 10; blue, 190 }  ,fill opacity=0.8 ][line width=0.08]  [draw opacity=0] (8.04,-3.86) -- (0,0) -- (8.04,3.86) -- (5.34,0) -- cycle    ;
\draw [shift={(170,100)}, rotate = 102.77] [color={rgb, 255:red, 14; green, 10; blue, 190 }  ,draw opacity=0.8 ][fill={rgb, 255:red, 14; green, 10; blue, 190 }  ,fill opacity=0.8 ][line width=0.75]      (0, 0) circle [x radius= 2.34, y radius= 2.34]   ;
\draw [color={rgb, 255:red, 14; green, 10; blue, 190 }  ,draw opacity=0.8 ]   (310,100) .. controls (287.75,198.21) and (331.75,175.71) .. (309,260) ;
\draw [shift={(309,260)}, rotate = 105.1] [color={rgb, 255:red, 14; green, 10; blue, 190 }  ,draw opacity=0.8 ][fill={rgb, 255:red, 14; green, 10; blue, 190 }  ,fill opacity=0.8 ][line width=0.75]      (0, 0) circle [x radius= 2.34, y radius= 2.34]   ;
\draw [shift={(307.02,176.16)}, rotate = 74.1] [fill={rgb, 255:red, 14; green, 10; blue, 190 }  ,fill opacity=0.8 ][line width=0.08]  [draw opacity=0] (8.04,-3.86) -- (0,0) -- (8.04,3.86) -- (5.34,0) -- cycle    ;
\draw [shift={(310,100)}, rotate = 102.77] [color={rgb, 255:red, 14; green, 10; blue, 190 }  ,draw opacity=0.8 ][fill={rgb, 255:red, 14; green, 10; blue, 190 }  ,fill opacity=0.8 ][line width=0.75]      (0, 0) circle [x radius= 2.34, y radius= 2.34]   ;
\draw [color={rgb, 255:red, 208; green, 2; blue, 27 }  ,draw opacity=1 ]   (436.42,164.71) .. controls (374.01,170.18) and (175.5,128) .. (83.92,179.71) ;
\draw [shift={(83.92,179.71)}, rotate = 150.55] [color={rgb, 255:red, 208; green, 2; blue, 27 }  ,draw opacity=1 ][fill={rgb, 255:red, 208; green, 2; blue, 27 }  ,fill opacity=1 ][line width=0.75]      (0, 0) circle [x radius= 2.34, y radius= 2.34]   ;
\draw [shift={(262.83,154.49)}, rotate = 182.24] [fill={rgb, 255:red, 208; green, 2; blue, 27 }  ,fill opacity=1 ][line width=0.08]  [draw opacity=0] (8.04,-3.86) -- (0,0) -- (8.04,3.86) -- (5.34,0) -- cycle    ;
\draw [shift={(436.42,164.71)}, rotate = 174.99] [color={rgb, 255:red, 208; green, 2; blue, 27 }  ,draw opacity=1 ][fill={rgb, 255:red, 208; green, 2; blue, 27 }  ,fill opacity=1 ][line width=0.75]      (0, 0) circle [x radius= 2.34, y radius= 2.34]   ;

\draw (129,52.4) node [anchor=north east] [inner sep=0.75pt]  [font=\small,color={rgb, 255:red, 245; green, 166; blue, 35 }  ,opacity=1 ,xscale=1.2,yscale=1.2]  {$R_{2} T^{k_{2} +6}\wh{\phi }$};
\draw (199,292.6) node [anchor=south east] [inner sep=0.75pt]  [font=\small,color={rgb, 255:red, 245; green, 166; blue, 35 }  ,opacity=1 ,xscale=1.2,yscale=1.2]  {$R_{2} T^{k_{2} -3}\wh{\phi }$};
\draw (312.24,37.73) node [anchor=south] [inner sep=0.75pt]  [font=\small,color={rgb, 255:red, 245; green, 166; blue, 35 }  ,opacity=1 ,xscale=1.2,yscale=1.2]  {$R_{4} T^{k_{4} +6}\wh{\phi }$};
\draw (277,292.6) node [anchor=south west] [inner sep=0.75pt]  [font=\small,color={rgb, 255:red, 245; green, 166; blue, 35 }  ,opacity=1 ,xscale=1.2,yscale=1.2]  {$R_{4} T^{k_{4} -3}\wh{\phi }$};
\draw (250.5,162.4) node [anchor=north] [inner sep=0.75pt]  [font=\small,color={rgb, 255:red, 208; green, 2; blue, 27 }  ,opacity=1 ,xscale=1.2,yscale=1.2]  {$I_{\wh\F}^{[ t_{1} ,t_{5}]}(\wh{y}_{0})$};
\draw (116,227.6) node [anchor=south east] [inner sep=0.75pt]  [font=\small,color={rgb, 255:red, 245; green, 166; blue, 35 }  ,opacity=1 ,xscale=1.2,yscale=1.2]  {$R_{2} T^{k_{2} -3}\wh{f}^{m_{1}}(\wh{\phi })$};
\draw (368,226) node [anchor=west] [inner sep=0.75pt]  [font=\small,color={rgb, 255:red, 245; green, 166; blue, 35 }  ,opacity=1 ,xscale=1.2,yscale=1.2]  {$R_{4} T^{k_{4} -3}\wh{f}^{m_{1}}(\wh{\phi })$};
\draw (119,129.6) node [anchor=south east] [inner sep=0.75pt]  [font=\small,color={rgb, 255:red, 157; green, 139; blue, 60 }  ,opacity=1 ,xscale=1.2,yscale=1.2]  {$R_{2} T^{k_{2} -3}\wh{f}^{n_{0} +m_{1}}(\wh{\phi })$};
\draw (171.5,37.6) node [anchor=south] [inner sep=0.75pt]  [font=\small,color={rgb, 255:red, 152; green, 134; blue, 59 }  ,opacity=1 ,xscale=1.2,yscale=1.2]  {$R_{2} T^{k_{2} +6}\wh{f}^{n_{0}}(\wh{\phi })$};

\end{tikzpicture}

\caption{Proof of Lemma~\ref{LemExistTrans26} (this is the continuation of Figure~\ref{FigLemConsLemYLeftRight}).}\label{FigLemExistTrans26}
\end{center}
\end{figure}

Note that for $i=2,4$,
\begin{equation}\label{EqInclusionAnyN}
\wh C_{k_i-3}^- \cup T^{ k_i}\wh f^{m_1}I^{[m'_0,m_0]}_{\wh\F}(\wh z)
\subset  \wh f^{m_1}(\wh C_{k_i-3}^-).
\end{equation}
By Lemma~\ref{LemConsLemYLeftRight}
\begin{multline*}
L\Big(\wh f^{n_0}\Big(R_2\big(\wh C_{k_2-3}^- \cup T^{ k_2-3}\wh f^{m_1}I^{[m'_0,m_0]}_{\wh\F}(\wh z)\cup \wh C_{k_2+6}^+\big)\Big)\Big) \\
\cap R\Big(R_4\big(\wh C_{k_4-3}^- \cup T^{ k_4-3}\wh f^{m_1}I^{[m'_0,m_0]}_{\wh\F}(\wh z)\cup \wh C_{k_4+6}^+\big)\Big)\neq\emptyset,
\end{multline*}
So by \eqref{EqInclusionAnyN},
\begin{multline}\label{Mult1}
L\Big(\wh f^{n_0}\Big(R_2\big( \wh f^{m_1}(\wh C_{k_i-3}^-) \cup \wh C_{k_2+6}^+\big)\Big)\Big) \\
\cap R\Big(R_4\big(\wh C_{k_4-3}^- \cup T^{ k_4-3}\wh f^{m_1}I^{[m'_0,m_0]}_{\wh\F}(\wh z)\cup \wh C_{k_4+6}^+\big)\Big)\neq\emptyset.
\end{multline}

But 
\begin{multline*}
R_2T^{ k_2-3}\wh f^{m_0}(\wh z) \in R_0T^{ k_2-3}\wh f^{m_1}I^{[m'_0,m_0]}_{\wh\F}(\wh z)\\
\cap L\Big(R_4\big(\wh C_{k_4-3}^- \cup T^{ k_4-3}\wh f^{m_1}I^{[m'_0,m_0]}_{\wh\F}(\wh z)\cup \wh C_{k_4+6}^+\big)\Big)
\end{multline*}
so by \eqref{EqInclusionAnyN} and the fact that $\wh C_{k_2+3}^+$ is positively $\wh f$-invariant (it is the closure of the right of a leaf),
\[\wh f^{n_0+m_1}(R_2\wh C_{k_2-3}^-)
\cap L\Big(R_4\big(\wh C_{k_4-3}^- \cup T^{ k_4-3}\wh f^{m_1}I^{[m'_0,m_0]}_{\wh\F}(\wh z)\cup \wh C_{k_4+6}^+\big)\Big)
\neq\emptyset.
\]

Hence, \eqref{Mult1} becomes
\[
R_2\big(\wh f^{n_0+m_1}(\wh C_{k_2-3}^-)\cup \wh f^{n_0}(\wh C_{k_2+6}^+)\big)
\cap R_4\big(\wh C_{k_4-3}^- \cup T^{ k_4-3}\wh f^{m_1}I^{[m'_0,m_0]}_{\wh\F}(\wh z)\cup \wh C_{k_4+6}^+\big)\neq\emptyset.
\]
and as a consequence, using again \eqref{EqInclusionAnyN},
\[
R_2\big(\wh f^{n_0+m_1}(\wh C_{k_2-3}^-)\cup \wh f^{n_0}(\wh C_{k_2+6}^+)\big)
\cap R_4\big(\wh f^{m_1}(\wh C_{k_4-3}^-)\cup \wh C_{k_4+6}^+\big)\neq\emptyset.
\]
Using the fact that the sets $\big(R_i\wh f^{m_1}(\wh C_{k_i}^-)\big)_i$ are pairwise disjoint and that the sets $\big(R_i\wh C_{k_i+6}^+\big)_i$ are pairwise disjoint (by \eqref{EqLiesBetween} and Lemma~\ref{LemCantCrossLotLeaves}), this implies that
\[
\big(R_2\wh f^{n_0+m_1}(\wh C_{k_2-3}^-)\cap R_4\wh C_{k_4+6}^+\big)
\cup \big(R_2 \wh f^{n_0}(\wh C_{k_2+6}^+) \cap R_4\wh f^{m_1}(\wh C_{k_4-3}^-)\big)\neq\emptyset.
\]
However (because we have supposed $n_0\ge m_1$), $\wh f^{n_0-m_1}(\wh C_{k_2+6}^+)\subset \wh C_{k_2+6}^+$, and $R_2 \wh C_{k_2+6}^+ \cap R_4 \wh C_{k_4}^- = \emptyset$, so the second intersection $\wh f^{m_1}\big(R_2 \wh f^{n_0-m_1}(\wh C_{k_2+6}^+) \cap R_4\wh C_{k_4-3}^-\big)$ of the last equation is empty, and subsequently
\[R_2\wh f^{n_0+m_1}(\wh C_{k_2+6}^-)\cap R_4\wh C_{k_4+6}^+\neq\emptyset,\]
which implies that\footnote{Using the fact that the transverse path linking $R_2T^{ k_2-3}\wh \phi$ to $R_4T^{ k_4+6}\wh \phi$ has the leaf $R_2 T^{k_2+6}\wh\phi$ on its left, and \cite[Proposition 19]{lct1}.}
\[\wh f^{n_0+m_1}\big(R_2T^{ k_2-3}\wh \phi\big)\cap R_4T^{ k_4+6}\wh \phi\neq\emptyset.\]
Hence, there exists an $f$-admissible transverse path $\wh\beta$ of order $n_0+m_1$ linking $R_2T^{ k_2-3}\wh \phi$ to $R_4T^{ k_4+6}\wh \phi$.
\end{proof}

\begin{figure}
\begin{center}

\tikzset{every picture/.style={line width=0.75pt}} 

\begin{tikzpicture}[x=0.75pt,y=0.75pt,yscale=-1.1,xscale=1.1]

\draw [color={rgb, 255:red, 245; green, 166; blue, 35 }  ,draw opacity=1 ]   (230,240) .. controls (250.71,196.42) and (288.21,196.92) .. (290,240) ;
\draw [color={rgb, 255:red, 245; green, 166; blue, 35 }  ,draw opacity=1 ]   (320,240) .. controls (340.71,196.42) and (378.21,196.92) .. (380,240) ;
\draw [color={rgb, 255:red, 245; green, 166; blue, 35 }  ,draw opacity=1 ]   (410,240) .. controls (430.71,196.42) and (468.21,196.92) .. (470,240) ;
\draw [color={rgb, 255:red, 245; green, 166; blue, 35 }  ,draw opacity=1 ]   (240,100) .. controls (249.29,129.42) and (281.29,141.42) .. (300,100) ;
\draw [color={rgb, 255:red, 245; green, 166; blue, 35 }  ,draw opacity=1 ]   (330,100) .. controls (339.29,129.42) and (371.29,141.42) .. (390,100) ;
\draw [color={rgb, 255:red, 245; green, 166; blue, 35 }  ,draw opacity=1 ]   (420,100) .. controls (429.29,129.42) and (461.29,141.42) .. (480,100) ;
\draw [color={rgb, 255:red, 208; green, 2; blue, 114 }  ,draw opacity=1 ]   (260,240) .. controls (255.29,196.92) and (279.79,142.42) .. (270,100) ;
\draw [shift={(267.42,166.83)}, rotate = 100.05] [fill={rgb, 255:red, 208; green, 2; blue, 114 }  ,fill opacity=1 ][line width=0.08]  [draw opacity=0] (8.04,-3.86) -- (0,0) -- (8.04,3.86) -- (5.34,0) -- cycle    ;
\draw [color={rgb, 255:red, 208; green, 2; blue, 114 }  ,draw opacity=1 ][line width=0.75]    (350,240) .. controls (345.29,196.92) and (369.79,142.42) .. (360,100) ;
\draw [shift={(357.42,166.83)}, rotate = 100.05] [fill={rgb, 255:red, 208; green, 2; blue, 114 }  ,fill opacity=1 ][line width=0.08]  [draw opacity=0] (8.04,-3.86) -- (0,0) -- (8.04,3.86) -- (5.34,0) -- cycle    ;
\draw [color={rgb, 255:red, 208; green, 2; blue, 114 }  ,draw opacity=1 ]   (440,240) .. controls (435.29,196.92) and (459.79,142.42) .. (450,100) ;
\draw [shift={(447.42,166.83)}, rotate = 100.05] [fill={rgb, 255:red, 208; green, 2; blue, 114 }  ,fill opacity=1 ][line width=0.08]  [draw opacity=0] (8.04,-3.86) -- (0,0) -- (8.04,3.86) -- (5.34,0) -- cycle    ;
\draw [color={rgb, 255:red, 23; green, 101; blue, 193 }  ,draw opacity=1 ]   (240,180) .. controls (288.79,180.42) and (432.29,203.42) .. (480,180) ;
\draw [shift={(363.28,189.11)}, rotate = 183.43] [fill={rgb, 255:red, 23; green, 101; blue, 193 }  ,fill opacity=1 ][line width=0.08]  [draw opacity=0] (7.14,-3.43) -- (0,0) -- (7.14,3.43) -- (4.74,0) -- cycle    ;
\draw [color={rgb, 255:red, 65; green, 117; blue, 5 }  ,draw opacity=1 ][line width=1.5]    (258.13,239.17) .. controls (257.06,226.23) and (257.5,209.35) .. (263,183) .. controls (290.78,184.83) and (396.51,197.07) .. (440.88,190.98) .. controls (448.09,154.37) and (453.2,123.91) .. (447.88,100.42) ;
\draw [shift={(259.04,206.5)}, rotate = 96.13] [fill={rgb, 255:red, 65; green, 117; blue, 5 }  ,fill opacity=1 ][line width=0.08]  [draw opacity=0] (8.75,-4.2) -- (0,0) -- (8.75,4.2) -- (5.81,0) -- cycle    ;
\draw [shift={(356.4,190.68)}, rotate = 183.78] [fill={rgb, 255:red, 65; green, 117; blue, 5 }  ,fill opacity=1 ][line width=0.08]  [draw opacity=0] (8.75,-4.2) -- (0,0) -- (8.75,4.2) -- (5.81,0) -- cycle    ;
\draw [shift={(449.03,141.23)}, rotate = 96.46] [fill={rgb, 255:red, 65; green, 117; blue, 5 }  ,fill opacity=1 ][line width=0.08]  [draw opacity=0] (8.75,-4.2) -- (0,0) -- (8.75,4.2) -- (5.81,0) -- cycle    ;
\draw [color={rgb, 255:red, 65; green, 117; blue, 5 }  ,draw opacity=1 ][line width=1.5]    (347.33,240) .. controls (342.62,196.92) and (367.12,142.42) .. (357.33,100) ;
\draw [shift={(355.13,164.65)}, rotate = 99.95] [fill={rgb, 255:red, 65; green, 117; blue, 5 }  ,fill opacity=1 ][line width=0.08]  [draw opacity=0] (9.91,-4.76) -- (0,0) -- (9.91,4.76) -- (6.58,0) -- cycle    ;

\draw (267.63,146.01) node [anchor=east] [inner sep=0.75pt]  [color={rgb, 255:red, 182; green, 0; blue, 99 }  ,opacity=1 ,xscale=1.2,yscale=1.2]  {$R_{2}\hat{\alpha }_{0}$};
\draw (362.69,149.35) node [anchor=west] [inner sep=0.75pt]  [color={rgb, 255:red, 182; green, 0; blue, 99 }  ,opacity=1 ,xscale=1.2,yscale=1.2]  {$R_{3}\hat{\alpha }_{0}$};
\draw (453.69,148.01) node [anchor=west] [inner sep=0.75pt]  [color={rgb, 255:red, 182; green, 0; blue, 99 }  ,opacity=1 ,xscale=1.2,yscale=1.2]  {$R_{4}\hat{\alpha }_{0}$};
\draw (475.03,181.96) node [anchor=north west][inner sep=0.75pt]  [color={rgb, 255:red, 23; green, 101; blue, 193 }  ,opacity=1 ,xscale=1.2,yscale=1.2]  {$I^{[ 0,n_{0}]}_{\wh\F}( \wh y_{0})$};
\draw (290.75,187.68) node [anchor=north west][inner sep=0.75pt]  [color={rgb, 255:red, 65; green, 117; blue, 5 }  ,opacity=1 ,xscale=1.2,yscale=1.2]  {$\hat{\beta }$};
\draw (356.02,143.8) node [anchor=east] [inner sep=0.75pt]  [color={rgb, 255:red, 65; green, 117; blue, 5 }  ,opacity=1 ,xscale=1.2,yscale=1.2]  {$S_{0}\hat{\beta }$};
\draw (250.3,98.15) node [anchor=south] [inner sep=0.75pt]  [font=\small,color={rgb, 255:red, 213; green, 135; blue, 4 }  ,opacity=1 ,xscale=1.2,yscale=1.2]  {$R_{2} T^{k_{2} +3}\hat{\phi }$};
\draw (359.3,96.65) node [anchor=south] [inner sep=0.75pt]  [font=\small,color={rgb, 255:red, 213; green, 135; blue, 4 }  ,opacity=1 ,xscale=1.2,yscale=1.2]  {$R_{3} T^{k_{3} +3}\hat{\phi }$};
\draw (469.3,97.15) node [anchor=south] [inner sep=0.75pt]  [font=\small,color={rgb, 255:red, 213; green, 135; blue, 4 }  ,opacity=1 ,xscale=1.2,yscale=1.2]  {$R_{4} T^{k_{4} +3}\hat{\phi }$};
\draw (241.35,237.4) node [anchor=north] [inner sep=0.75pt]  [font=\small,color={rgb, 255:red, 213; green, 135; blue, 4 }  ,opacity=1 ,xscale=1.2,yscale=1.2]  {$R_{2} T^{k_{2}}\hat{\phi }$};
\draw (348.8,238.15) node [anchor=north] [inner sep=0.75pt]  [font=\small,color={rgb, 255:red, 213; green, 135; blue, 4 }  ,opacity=1 ,xscale=1.2,yscale=1.2]  {$R_{3} T^{k_{3}}\hat{\phi }$};
\draw (466.8,236.65) node [anchor=north] [inner sep=0.75pt]  [font=\small,color={rgb, 255:red, 213; green, 135; blue, 4 }  ,opacity=1 ,xscale=1.2,yscale=1.2]  {$R_{4} T^{k_{4}}\hat{\phi }$};

\end{tikzpicture}

\caption{Proof of Lemma~\ref{LemInterBetaa}. \label{FigBananaKi}}
\end{center}
\end{figure}

Let us write $S_0 = R_3T^3R_2^{-1}$, and consider $\wh\beta$ the path given by Lemma~\ref{LemExistTrans26}
(see Figure~\ref{FigBananaKi}).

\begin{lemma}\label{LemInterBetaa}
The paths $\wh\beta$ and $S_0\wh\beta$ intersect $\wh\F$-transversally.
Similarly, the paths $\wh\beta$ and $R_2T^{-3}R_3^{-1}\wh\beta$ intersect $\wh\F$-transversally.
\end{lemma}

\begin{proof}
We prove the first part of the lemma, the second one being identical. The idea of the proof is depicted in Figure~\ref{FigBananaKi}.

Recall that the trajectory $I^{[0,n_0]}_{\wt\F}(\wt y_0)$ is simple in $\wt S$. Moreover, by hypothesis, the paths $R_2\wt\alpha_0$, $R_3\wt\alpha_0$ and $R_4\wt\alpha_0$ are pairwise disjoint and cross $I^{[0,n_0]}_{\wt\F}(\wt y_0)$ in an increasing order: for any $i=2,3,4$, there exist $t_i, s_i$ such that $I^{t_i}_{\wt\F}(\wt y_0) = R_i\wt\alpha_0(s_i)$, with $t_2<t_3<t_4$. We also denote $s'\in\R$ such that $\wt\alpha_0(s')\in \wt\phi$.

These facts imply that the two paths $R_2\wt\alpha_0|_{[s'+k_2-3, s_2]}I^{[t_2, t_4]}_{\wt\F}(\wt y_0)R_4\wt\alpha_0|_{[s_4, s'+k_4+6]}$ and $R_3\wt\alpha_0|_{[s'+k_3-3, s'+k_3+6]}$ intersect at the point $I^{t_3}_{\wt\F}(\wt y_0) = R_3\wt\alpha_0(s_3)$. 
Moreover, recall that for $i\neq j$, the leaves $R_iT^{k_i}(\wh\phi)$ and $R_iT^{k_i+3}(\wh\phi)$ do not meet $R_j\wh\alpha_0$. Finally, we have that $R_2\wt\alpha_0(s'+k_2)\in L(R_3\wt\alpha_0)$ and $R_4\wt\alpha_0(s'+k_4+3)\in R(R_3\wt\alpha_0)$. 

This allows us to apply \cite[Lemma~10.7]{pa} to the paths $R_2\wt\alpha_0|_{[-\infty, s_2]}I^{[t_2, t_4]}_{\wt\F}(\wt y_0)R_4\wt\alpha_0|_{[s_4, +\infty]}$ and $R_3\wh\alpha_0$ (point (2) of this lemma implies that there exists an ``essential intersection point'' as defined by \cite[Definition~10.6]{pa} and point (3) implies our conclusion\footnote{\cite[Lemma~10.7]{pa} is stated in terms of transverse trajectories of points but the proof is in fact valid for general transverse paths.}): one gets that the paths $R_2\wh\alpha_0|_{[s'+k_2, s_2]}I^{[t_2, t_4]}_{\wh\F}(\wh y_0)R_2\wh\alpha_0|_{[s_4, s'+k_4+3]}$ and $R_3\wh\alpha_0|_{[s'+k_3, s'+k_3+3]}$ intersect $\wh\F$-transversally at $I^{t_3}_{\wh\F}(\wh y_0) = R_3\wh\alpha_0(s_3)$. The first of these paths is $\wh\F$-equivalent to a subpath of $\wh\beta$, and the second one is $\wh\F$-equivalent to a subpath of  $S_0\wh\beta$.
\end{proof}

\begin{proof}[Proof of Proposition~\ref{PropBndedDevRatCase2}]
By Lemma~\ref{LemInterBetaa}, the paths $\wh\beta$ and $S_0\wh\beta$ intersect $\F$-transversally.
It allows us to apply Theorem~\ref{thmMlct2} which asserts that there exists $\wt z_p\in \wt S$ such that $\wt f^{n_0+m_1}(\wt z_p) = S_0\wt z_p$: the point $z_p$ is $n_0+m_1$-periodic and turns around $S$ by the deck transformation $S_0$. 

Note that $S_0(R_2\wt\gamma_{\wt z}) = R_3T^3R_2^{-1}(R_2\wt\gamma_{\wt z}) = R_3\wt\gamma_{\wt z}$. Recall that the sets $R_i\wt\gamma_{\wt z}$ are pairwise disjoint. Moreover, the orientations of the $R_i\wt\gamma_{\wt z}$ are supposed to be identical. 
This forces the geodesic axis of $S_0$ to cross both $R_2\wt\gamma_{\wt z}$ and $R_3\wt\gamma_{\wt z}$. Note that $\wt\gamma_{R_2^{-1}\wt z_p} = R_2^{-1} \wt\gamma_{\wt z_p}$, hence by what we just said the tracking geodesic of $R_2^{-1}\wt z_p$ crosses the one of $\wt z$.
\end{proof}

With the same ideas, one can prove the following proposition which improves Proposition~\ref{PropBndedDevRatCase2}. It is included here not just for the sake of generalization, but also as it will be useful in future works, as for instance \cite{PABig1, PABig2}). Let us first define some objects.

Let $\gamma_1, \gamma_2$ be two closed geodesics that are tracking geodesics for some $f$-ergodic measures. Let $T_1, T_2\in \G$ be primitive deck transformations associated to these closed geodesics.

We apply Proposition~\ref{LemRealizPeriod} twice to get transverse paths $\wt\alpha_i\subset \wt S$, $i=1,2$, that are simple and $T_i$-invariant, and $\wt\F$-equivalent to a transverse path $I^\Z_{\wt\F}(\wt z_i)$, for some $\mu_i\in\Merg$ and $\wt z_i \in\wt S$ a lift of a $\mu_i$-typical point $z_i\in S$ having $\wt \gamma_i$ as a tracking geodesic and staying at finite distance to it. 
Up to reparametrization, one can suppose that for any $t\in\R$ and $k\in\Z$, one has $\wt\alpha_i(t+k) = T_i^k\wt\alpha_i(t)$.
We denote $\wt B_i$ the set of leaves met by $\wt \alpha_i$; in this band the left and the right of a leaf of $\wt \F$ are well defined.

\begin{prop}\label{PropBndedDevRatCase2b}
There exist $D'>0$ and $m_1'\ge 0$ such that the following is true.
For $i=1,2$, suppose that there exist 4 different copies of $\wt B_i$, denoted by $(R_j^i \wt B_i)_{1\le j\le 4}$, such that the following properties hold:
\begin{itemize}
\item the sets $(R_j^i V_{D'}(\wt\gamma_i))_{1\le j\le 4,\, i=1,2}$ are pairwise disjoint and for $i=1,2$, the sets $(R_j^i V_{D'}(\wt\gamma_i))_{1\le j\le 4}$ have the same orientation;
\item there exists $n_0\ge m_1$ and, for all $1\le j\le 4$, some times $t_j^i\in[0,n_0]$ such that $I^{t_j^i}_{\wt\F}(\wt y_0)\in R_i^j\wt\alpha_0$, and that either for all $j$ we have $I^{[t_1^i,t_j^i]}_{\wt\F}(\wt y_0)\subset R_j^i\wt B_i$, or for all $j$ we have $I^{[t_j^i, t_4^i]}_{\wt\F}(\wt y_0)\subset R_j^i\wt B_i$;
\item we have $t_4^1\le t_1^2$.
\end{itemize}


Then there exists an $\wt f$-admissible transverse path $\wt\beta$ of order $n_0+2m_1$ and parametrised by $[t_0, t_2]$, and some $t_1\in (t_0, t_2)$ such that $\wt\beta|_{[t_0, t_1]}$ and $R_3^1 T_1^3(R_2^1)^{-1}\wt\beta|_{[t_0, t_1]}$ intersect $\F$-transversally, and that $\wt\beta|_{[t_1, t_2]}$ and $R_2^2 T_2^{-3}(R_3^2)^{-1}\wt\beta|_{[t_1, t_2]}$ intersect $\F$-transversally.
\end{prop}

Note that the geodesic axis of $R_3^1 T_1^3(R_2^1)^{-1}$ crosses both $R_2^1\wt\gamma_1$ and $R_3^1\wt\gamma_1$ and that the geodesic axis of $R_2^2 T_2^{-3}(R_3^2)^{-1}$ crosses both $R_2^2\wt\gamma_1$ and $R_3^2\wt\gamma_1$.

In particular, the conclusion of the proposition implies that the projection $\beta$ of $\wt\beta$ on $S$ has two $\F$-transverse self-intersections.

\begin{rem}\label{RemBigSubpaths}
The path $\wt\beta$ is made of the concatenation of some paths $I^{[s_1,t_1]}_\F(z_1)$, $I^{[u_1, u_2]}_\F(y_0)$ and  $I^{[s_2,t_2]}_\F(z_2)$; by modifying the proof one can suppose that $t_1-s_1$ and $t_2-s_2$ are large, while $u_1$ and $u_2$ remain bounded.
\end{rem}

\begin{figure}
\begin{center}
\tikzset{every picture/.style={line width=0.75pt}} 

\begin{tikzpicture}[x=0.75pt,y=0.75pt,yscale=-1.1,xscale=1.1]

\draw [color={rgb, 255:red, 245; green, 166; blue, 35 }  ,draw opacity=1 ]   (250,211) .. controls (270.71,167.42) and (308.21,167.92) .. (310,211) ;
\draw [color={rgb, 255:red, 245; green, 166; blue, 35 }  ,draw opacity=1 ]   (340,211) .. controls (360.71,167.42) and (398.21,167.92) .. (400,211) ;
\draw [color={rgb, 255:red, 245; green, 166; blue, 35 }  ,draw opacity=1 ]   (430,211) .. controls (450.71,167.42) and (488.21,167.92) .. (490,211) ;
\draw [color={rgb, 255:red, 245; green, 166; blue, 35 }  ,draw opacity=1 ]   (260,71) .. controls (269.29,100.42) and (301.29,112.42) .. (320,71) ;
\draw [color={rgb, 255:red, 245; green, 166; blue, 35 }  ,draw opacity=1 ]   (350,71) .. controls (359.29,100.42) and (391.29,112.42) .. (410,71) ;
\draw [color={rgb, 255:red, 245; green, 166; blue, 35 }  ,draw opacity=1 ]   (440,71) .. controls (449.29,100.42) and (481.29,112.42) .. (500,71) ;
\draw [color={rgb, 255:red, 23; green, 101; blue, 193 }  ,draw opacity=1 ]   (260,150.67) .. controls (308.79,151.08) and (549.03,119.33) .. (608.69,149.67) ;
\draw [shift={(439.11,138.48)}, rotate = 176.97] [fill={rgb, 255:red, 23; green, 101; blue, 193 }  ,fill opacity=1 ][line width=0.08]  [draw opacity=0] (7.14,-3.43) -- (0,0) -- (7.14,3.43) -- (4.74,0) -- cycle    ;
\draw [color={rgb, 255:red, 65; green, 117; blue, 5 }  ,draw opacity=1 ][line width=1.5]    (278.55,230.15) .. controls (277.48,217.21) and (278.06,179.35) .. (283.56,153) .. controls (332.56,148.67) and (501.03,134.67) .. (564.36,141.33) .. controls (571.58,104.72) and (563.08,71.25) .. (557.75,47.75) ;
\draw [shift={(279.17,186.92)}, rotate = 93.73] [fill={rgb, 255:red, 65; green, 117; blue, 5 }  ,fill opacity=1 ][line width=0.08]  [draw opacity=0] (8.75,-4.2) -- (0,0) -- (8.75,4.2) -- (5.81,0) -- cycle    ;
\draw [shift={(428.71,142.48)}, rotate = 176.87] [fill={rgb, 255:red, 65; green, 117; blue, 5 }  ,fill opacity=1 ][line width=0.08]  [draw opacity=0] (8.75,-4.2) -- (0,0) -- (8.75,4.2) -- (5.81,0) -- cycle    ;
\draw [shift={(565.76,89.63)}, rotate = 84] [fill={rgb, 255:red, 65; green, 117; blue, 5 }  ,fill opacity=1 ][line width=0.08]  [draw opacity=0] (8.75,-4.2) -- (0,0) -- (8.75,4.2) -- (5.81,0) -- cycle    ;
\draw [color={rgb, 255:red, 65; green, 117; blue, 5 }  ,draw opacity=1 ][line width=1.5]    (441.82,234.32) .. controls (442.53,224.09) and (440.96,216.29) .. (442.65,206.49) .. controls (447.3,206.42) and (457.82,206.99) .. (469.49,208.15) .. controls (471.65,196.65) and (480.46,109.52) .. (471.33,70) ;
\draw [shift={(441.85,215.29)}, rotate = 89.62] [fill={rgb, 255:red, 65; green, 117; blue, 5 }  ,fill opacity=1 ][line width=0.08]  [draw opacity=0] (9.91,-4.76) -- (0,0) -- (9.91,4.76) -- (6.58,0) -- cycle    ;
\draw [shift={(461.19,207.4)}, rotate = 184.16] [fill={rgb, 255:red, 65; green, 117; blue, 5 }  ,fill opacity=1 ][line width=0.08]  [draw opacity=0] (9.91,-4.76) -- (0,0) -- (9.91,4.76) -- (6.58,0) -- cycle    ;
\draw [shift={(475.23,133.26)}, rotate = 91.62] [fill={rgb, 255:red, 65; green, 117; blue, 5 }  ,fill opacity=1 ][line width=0.08]  [draw opacity=0] (9.91,-4.76) -- (0,0) -- (9.91,4.76) -- (6.58,0) -- cycle    ;
\draw [color={rgb, 255:red, 245; green, 166; blue, 35 }  ,draw opacity=1 ]   (520.4,209.8) .. controls (541.11,166.22) and (578.61,166.72) .. (580.4,209.8) ;
\draw [color={rgb, 255:red, 245; green, 166; blue, 35 }  ,draw opacity=1 ]   (530.4,69.8) .. controls (539.69,99.22) and (571.69,111.22) .. (590.4,69.8) ;
\draw [color={rgb, 255:red, 65; green, 117; blue, 5 }  ,draw opacity=1 ][line width=1.5]    (370.49,210.99) .. controls (371.19,200.75) and (377.83,124.67) .. (368.17,74.67) .. controls (373.83,75) and (388.5,75.67) .. (398.5,76.33) .. controls (399.17,68.33) and (399.17,55.67) .. (400,46.67) ;
\draw [shift={(373.33,137.07)}, rotate = 89.7] [fill={rgb, 255:red, 65; green, 117; blue, 5 }  ,fill opacity=1 ][line width=0.08]  [draw opacity=0] (9.91,-4.76) -- (0,0) -- (9.91,4.76) -- (6.58,0) -- cycle    ;
\draw [shift={(388.44,75.73)}, rotate = 183.05] [fill={rgb, 255:red, 65; green, 117; blue, 5 }  ,fill opacity=1 ][line width=0.08]  [draw opacity=0] (9.91,-4.76) -- (0,0) -- (9.91,4.76) -- (6.58,0) -- cycle    ;
\draw [shift={(399.39,56.48)}, rotate = 92.17] [fill={rgb, 255:red, 65; green, 117; blue, 5 }  ,fill opacity=1 ][line width=0.08]  [draw opacity=0] (9.91,-4.76) -- (0,0) -- (9.91,4.76) -- (6.58,0) -- cycle    ;

\draw (581.63,144) node [anchor=south west] [inner sep=0.75pt]  [color={rgb, 255:red, 23; green, 101; blue, 193 }  ,opacity=1 ,xscale=1.2,yscale=1.2]  {$I^{[ 0,n_{0}]}_{\wh\F}( \wh y_{0})$};
\draw (314.08,153) node [anchor=north west][inner sep=0.75pt]  [color={rgb, 255:red, 65; green, 117; blue, 5 }  ,opacity=1 ,xscale=1.2,yscale=1.2]  {$\wh{\beta }$};
\draw (370.02,109) node [anchor=east] [inner sep=0.75pt]  [font=\small,color={rgb, 255:red, 65; green, 117; blue, 5 }  ,opacity=1 ,xscale=1.2,yscale=1.2]  {$R_{3}^{1} T_{1}^{3}( R_{2}^{1})^{-1}\wh{\beta }$};
\draw (270.63,69.15) node [anchor=south] [inner sep=0.75pt]  [font=\small,color={rgb, 255:red, 213; green, 135; blue, 4 }  ,opacity=1 ,xscale=1.2,yscale=1.2]  {$R_{2}^{1} T_{1}^{k_{2}^{1} +3}\wh{\phi }_{1}$};
\draw (358.3,66.65) node [anchor=south] [inner sep=0.75pt]  [font=\small,color={rgb, 255:red, 213; green, 135; blue, 4 }  ,opacity=1 ,xscale=1.2,yscale=1.2]  {$R_{3}^{1} T_{1}^{k_{3}^{1} +3}\wh{\phi }_{1}$};
\draw (471.3,68.15) node [anchor=south] [inner sep=0.75pt]  [font=\small,color={rgb, 255:red, 213; green, 135; blue, 4 }  ,opacity=1 ,xscale=1.2,yscale=1.2]  {$R_{2}^{2} T^{k_{2}^{2} +3}\wh{\phi }_2$};
\draw (238.89,210.73) node [anchor=north] [inner sep=0.75pt]  [font=\small,color={rgb, 255:red, 213; green, 135; blue, 4 }  ,opacity=1 ,xscale=1.2,yscale=1.2]  {$R_{2}^{1} T_{1}^{k_{2}^{1}}\wh{\phi }_{1}$};
\draw (394.06,209.45) node [anchor=north] [inner sep=0.75pt]  [font=\small,color={rgb, 255:red, 213; green, 135; blue, 4 }  ,opacity=1 ,xscale=1.2,yscale=1.2]  {$R_{3}^{1} T_{1}^{k_{3}^{1}}\wh{\phi }_{1}$};
\draw (481.73,209.52) node [anchor=north] [inner sep=0.75pt]  [font=\small,color={rgb, 255:red, 213; green, 135; blue, 4 }  ,opacity=1 ,xscale=1.2,yscale=1.2]  {$R_{2}^{2} T_{2}^{k_{2}^{2}}\wh{\phi }_2$};
\draw (474,168) node [anchor=west] [inner sep=0.75pt]  [font=\small,color={rgb, 255:red, 65; green, 117; blue, 5 }  ,opacity=1 ,xscale=1.2,yscale=1.2]  {$R_{2}^{2} T_{2}^{-3}( R_{3}^{2})^{-1}\wh{\beta }$};
\draw (580,204) node [anchor=north] [inner sep=0.75pt]  [font=\small,color={rgb, 255:red, 213; green, 135; blue, 4 }  ,opacity=1 ,xscale=1.2,yscale=1.2]  {$R_{3}^{2} T_{2}^{k_{3}^{2}}\wh{\phi }_2$};
\draw (598.63,65.48) node [anchor=south] [inner sep=0.75pt]  [font=\small,color={rgb, 255:red, 213; green, 135; blue, 4 }  ,opacity=1 ,xscale=1.2,yscale=1.2]  {$R_{3}^{2} T^{k_{3}^{2} +3}\wh{\phi }_2$};

\end{tikzpicture}

\caption{Proof of Proposition~\ref{PropBndedDevRatCase2b}. \label{FigPropBndedDevRatCase2b}}
\end{center}
\end{figure}

\begin{proof}
We adapt the proof of Proposition~\ref{PropBndedDevRatCase2} (see also Figure~\ref{FigPropBndedDevRatCase2b}). First, the constant $D$ defined before Proposition~\ref{PropBndedDevRatCase2} does depend on the trajectory of $z$. 
For Proposition~\ref{PropBndedDevRatCase2b} we have two different points $z_1$, $z_2$, each of them associated with a band, $B_1$ and $B_2$. We choose $\phi_1\subset B_1$ and $\phi_2 \subset B_2$ two leaves. This allows us to get two constants $D_1$ and $D_2$, the first one associated to $z_1$ and the second one associated to $z_2$ (as in \eqref{EqDefD}). We then replace the $D$ of Proposition~\ref{PropBndedDevRatCase2} with $D' = \max(D_1,D_2)$. Similarly, one can define $m'_1$ as the maximum of the constants $m_1$ (defined in Subsection~\ref{SubSecConst}) associated to resp. $z_1$ and $z_2$.

The adaptation of Lemma~\ref{LemCantCrossLotLeaves} is straightforward, fixing $i=1,2$ and replacing the $R_i\wt B$ by $R_j^i \wt B_i$.

This lemma allows us to define integers $k_j^i$ as in \eqref{EqLiesBetween}.
One then has to replace the set $\wh C_0$ of \eqref{EqDefC0} by two sets $\wh C_1$ and $\wh C_2$, the first one adapted to the trajectory of $z_1$ and the second one adapted to the trajectory of $z_2$.

Lemma~\ref{LemConsLemYLeftRight} can then be adapted in the following way, with the same proof:
\begin{align*}
\begin{split}
\wh y_0 & \in L\big(R_2^1T^{k_2^1}\wh C_1\big) \cap L\big(R_2^2T^{k_3^2}\wh C_2\big),\\
\wh f^{n_0}(\wh y_0) & \in R\big(R_3^1T^{k_2^1}\wh C_1\big) \cap R\big(R_3^2T^{k_3^2}\wh C_2\big).
\end{split}
\end{align*}

Lemma~\ref{LemExistTrans26} then can be adapted as follows, with the same proof: 
\emph{There exists an $\wh f$-admissible transverse path $\wh\beta$ of order $n_0+m'_1$ linking $R_2^1T_1^{ k^1_2-3}\wh \phi_1$ to $R_3^2T_2^{ k_3^2+6}\wh \phi_2$.}

To get Remark~\ref{RemBigSubpaths}, it suffices to consider $M$ large and change this property by:
\emph{There exists an $\wh f$-admissible transverse path $\wh\beta$ of order $n_0+m'_1$ linking $R_2^1T_1^{ k^1_2-M}\wh \phi_1$ to $R_3^2T_2^{ k_3^2+M}\wh \phi_2$.} The rest of the proof is identical, adapting the constants to this change. 

Lemma~\ref{LemInterBetaa} then becomes: \emph{The path $\wh\beta$ intersects transversally both $R_3^1 T_1^3(R_2^1)^{-1}\wt\beta$ and $R_2^2 T_2^{-3}(R_3^2)^{-1}\wt\beta$.} This proves the proposition.
\end{proof}

\subsection{Second case: the trajectory crosses different copies of $\wt B$}

Recall that $\wt\alpha_0\subset \wt S$ is a transverse path  that is given by Proposition~\ref{LemRealizPeriod}; it is simple, $T$-invariant and $\wt\F$-equivalent to a transverse path $I^\Z_{\wt\F}(\wt z)$, for $\wt z \in\wt S$ a lift of a $\mu$-typical point $z\in S$.
Moreover, for any $t\in\R$ and $k\in\Z$, one has $\wt\alpha_0(t+k) = T^k\wt\alpha_0(t)$.
The set of leaves met by $\wt \alpha_0$ is $\wt B$.

We divide the proof depending on whether some $R_i\wt\alpha_0$ accumulates in $I^{[0,n_0]}_{\wt \F}(\wt y_0)$ or not. Let us begin with the non-accumulating case, which is the easiest one.

\begin{prop}\label{PropBndedDevRatCase1a}
Suppose that there exist $\wt y_0\in \wt S$, $ R_0,R_1\in\G$ and $t_0<t'_0\le t_1<t'_1$ such that for $i=0,1$, the trajectory $I^{[t_i,t'_i]}_{\wt \F}(\wt y_0)$ crosses the band $R_i\wt B$. 
Suppose that $R_0\wt\gamma_{\wt z}$ and $R_1\wt\gamma_{\wt z}$ do not cross and have the same orientation.
Suppose also that none of the $R_i\wt\alpha_0$ (recall that this path is given by Proposition~\ref{LemRealizPeriod}) accumulates in $I^{[t_i,t'_i]}_{\wt \F}(\wt y_0)$. 

Then there exists an $f$-periodic orbit having a lift whose tracking geodesic crosses both $R_0\wt\gamma_{\wt z}$ and $R_1\wt\gamma_{\wt z}$.
\end{prop}

In fact under the hypotheses of Proposition~\ref{PropBndedDevRatCase1a} one can get the conclusions of Propositions~\ref{PropBndedDevRatCase2} and \ref{PropBndedDevRatCase2b}. This shows that under the hypothesis that the tracking geodesic of $z$ is not simple, the conclusions of Proposition~\ref{PropBndedDevRatCase2} holds (as by Proposition~\ref{Prop3.3GLCP} this prevents from having an accumulation phenomenon). Similarly, under the hypothesis that the tracking geodesics of $z_1$ and $z_2$ are not simple, the conclusion of Proposition~\ref{PropBndedDevRatCase2b} holds. This corollary will be used in further works, it gives a simple criterion of creation of periodic orbits.

\begin{coro}\label{CoroKiSert}
Let $f\in\Homeo_0(S)$ and $\gamma_1, \gamma_2$ two closed geodesics that are tracking geodesics for some $f$-ergodic measures and that are not simple geodesics. 
Let $T_1, T_2\in \G$ be primitive deck transformations associated to these closed geodesics.

Then there exist periodic points $z_1$ and $z_2$ such that $\gamma_{z_1} = \gamma_1$ and $\gamma_{z_2} = \gamma_2$.

Moreover, for any $M>0$ there exists $D'>0$ and $m_1\ge 0$ such that the following is true.
For $i=1,2$, suppose that there exist $4$ deck transformations $(R_i^j)_{1\le j\le 4}\in\G$ such that the following properties hold:
\begin{itemize}
\item the sets $R_i^j V_{D'}(\wt\gamma_i)$ are pairwise disjoint and have the same orientation;
\item there exists $0\le n'_0\le n_0$, with $n'_0\ge m_1$ and $n_0-n'_0\ge m_1$ such that for any $1\le j\le 4$, the points $\wt y_0$ and $\wt f^{n'_0}(\wt y_0)$ lie in different sides of the complement of $R_1^j V_{D'}(\wt\gamma_1)$, and the points $\wt f^{n'_0}(\wt y_0)$ and $\wt f^{n_0}(\wt y_0)$ lie in different sides of the complement of $R_2^j V_{D'}(\wt\gamma_2)$.
\end{itemize}

Then there exists an $\wt f$-admissible transverse path $\wt\beta$ of order $n_0+2m_1$ and parametrised by $[t_0, t_2]$, and some $t_1\in (t_0, t_2)$ such that $\wt\beta|_{[t_0, t_1]}$ and $R_3^1 T_1^3(R_2^1)^{-1}\wt\beta|_{[t_0, t_1]}$ intersect $\F$-transversally, and that $\wt\beta|_{[t_1, t_2]}$ and $R_2^2 T_2^{-3}(R_3^2)^{-1}\wt\beta|_{[t_1, t_2]}$ intersect $\F$-transversally.

The path $\wt\beta$ is made of the concatenation of some paths $I^{[s_1,t_1]}_\F(z_1)$, $I^{[u_1, u_2]}_\F(y_0)$ and $I^{[s_2,t_2]}_\F(z_2)$, with $t_1-s_1\ge M$ and $t_2-s_2\ge M$.

Finally, if $\gamma_1 = \gamma_2$, then there exists a constant $d_0>0$ depending only on $z$ (and neither on $y_0$ nor on $n_0$) such that the tracking geodesic $\gamma_p$ of $p$ is freely homotopic to the concatenation $I^{[t_2,t_3]}_{\F}( y_0)\delta$, where $\diam(\wt\delta)\le d_0$ (with $\wt\delta$ a lift of $\delta$ to $\wt S$). 
\end{coro}

\begin{proof}
%
Thanks to Remark~\ref{RemPossibBand}, the corollary is obtained as a combination of Propositions~\ref{PropBndedDevRatCase2b} and \ref{PropBndedDevRatCase1a}, apart from the existence of the points $z_1$ and $z_2$ that is a consequence of \cite[Proposition 4.1.(iii)]{pa}. 
\end{proof}

\begin{proof}[Proof of Proposition~\ref{PropBndedDevRatCase1a}]
The hypotheses of the proposition allow to apply Lemma~\ref{LemPasAccImplTrans} which implies that for $i=0,1$ the transverse paths $I^{[t_i,t'_i]}_{\wt \F}(\wt y_0)$ and $R_i\wt\alpha_0$ intersect $\wt\F$-transversally. 
As $\wt\alpha_0$ is $\F$-equivalent to $I^\Z_{\wt\F}(\wt z)$, we deduce that there exist $s_i<s''_i< s'_i$ and $t''_i\in(t_i, t'_i)$ such that the transverse trajectories $I^{[t_i,t'_i]}_{\wt \F}(\wt y_0)$ and $R_i I^{[s_i,s'_i]}_{\wt \F}(\wt z)$ intersect $\wt\F$-transversally at $I^{t''_i}_{\wt \F}(\wt y_0) = R_i I^{s''_i}_{\wt \F}(\wt z)$. 

By Proposition~\ref{propFondalct1}, for any $s\ge s'_1$, there exists an admissible path $\wt\beta_s$ that is $\F$-equivalent to $I^{[t_0, t''_1]}(\wt y_0) \, I^{[s''_1,s]}_{\wt \F}(R_1\wt z)$. 

Consider a neighbourhood $\wt U$ of $\wt z$ given by Lemma~\ref{Lem17Lct1} (applied in $\wt S$), such that for any $\wt z'\in \wt U$ the path $I^{[s_0,s'_0]}_{\wt\F}(\wt z)$ is $\wt \F$-equivalent to a subpath of $I^{[s_0-1, s'_0+1]}_{\wt \F}(\wt z')$.
By Lemma~\ref{LemRecur}, there exist a neighbourhood $\wt W\subset \wt U$, $\ell\ge s''_1-s_1+1$ and $q>0$ such that $\wt f^{\ell}(\wt z) \in T^{q}(\wt W) \subset T^{q}(\wt U)$. As a result, the path $I^{[s_0,s'_0]}_{\wt\F}(\wt z)$ is $\wt \F$-equivalent to a subpath of $I^{[s_0-1, s'_0+1]}_{\wt \F}\big(T^{-q}\wt f^\ell(\wt z)\big) = T^{-q}I^{[s_0-1+\ell, s'_0+1+\ell]}_{\wt \F}(\wt z)$. Hence, the paths $I^{[t_0, t'_0]}(\wt y_0)$ and $R_0 T^{-q}I^{[s_0-1+\ell, s'_0+1+\ell]}_{\wt \F}(\wt z)$ intersect $\wt\F$-transversally. The first one is a subpath of $\wt\beta_{s'_0+1+\ell}$ and the second one is a subpath of $R_0 T^{-q}R_1^{-1}\wt\beta_{s'_0+1+\ell}$, therefore the paths $\wt\beta_{s'_0+1+\ell}$ and $R_0 T^{-q}R_1^{-1}\wt\beta_{s'_0+1+\ell}$ intersect $\wt\F$-transversally.

This allows us to apply Theorem~\ref{thmMlct2}, which implies that there exists $r>0$ and $\wt z'\in\wt S$ such that $\wt f^r(\wt z') = R_0 T^{-q}R_1^{-1} \wt z'$.

Because, by hypothesis, $R_0\wt\gamma_{\wt z}$ and $R_1\wt\gamma_{\wt z}$ do not cross and have the same orientation, the axis of the deck transformation $R_0 T^{-q}R_1^{-1}$ --- that sends the second one on the first one --- has to cross both $R_0\wt\gamma_{\wt z}$ and $R_1\wt\gamma_{\wt z}$. Hence, the tracking geodesic of $\wt z'$ has to cross both $R_0\wt\gamma_{\wt z}$ and $R_1\wt\gamma_{\wt z}$; this proves the proposition. 
\end{proof}

The difficult case is handeled in the following proposition.

\begin{prop}\label{PropBndedDevRatCase1}
Let $\wt y_0\in \wt S$, $n_0\in\N$ and 4 deck transformations $(R_i)_{}$ in $\G$ such that $(R_{i} \wt B)_{-1\le i\le 2}$ are different copies of $\wt B$ such that  $I^{[a_i,b_i]}(\wt y_0)$ crosses $R_i(\wt B)$ for $-1\le i\le 2$.
Suppose also that for some $i\in\{-1,0,1,2\}$, the path $R_i\wt\alpha_0$ (recall that this path is given by Proposition~\ref{LemRealizPeriod}) accumulates in $I^{[0,n_0]}_{\wt \F}(\wt y_0)$. 

Then there exists an $f$-periodic orbit having a lift whose tracking geodesic crosses both $R_0\wt\gamma_{\wt z}$ and $R_1\wt\gamma_{\wt z}$.
\end{prop}

The end of this subsection is devoted to the proof of this proposition.

Let us give an idea of the proof (see Figure~\ref{FigIdeaPropBndedDevRatCase1}). The two transverse trajectories $R_0^{-1}I^{[0,n_0]}_{\wt\F}(\wt y_0)$ and $R_1^{-1} I^{[0,n_0]}_{\wt\F}(\wt y_0)$ cross the band $\wt B$. If these two trajectories have an $\wt\F$-transverse intersection inside the band $\wt B$ (Figure~\ref{FigIdeaPropBndedDevRatCase1}, left), then Theorem~\ref{thmMlct2} allows us to create a new periodic point whose tracking geodesic crosses $\gamma$. If not (Figure~\ref{FigIdeaPropBndedDevRatCase1}, right), then there exists $k\in \Z$ such that $R_1^{-1} I^{[0,n_0]}_{\wt\F}(\wt y_0)$ crosses $\wt B$ between $T^k R_0^{-1}I^{[0,n_0]}_{\wt\F}(\wt y_0)$ and $T^{k+1} R_0^{-1}I^{[0,n_0]}_{\wt\F}(\wt y_0)$. The idea is to use the trajectory of $\wt z$ to apply Lellouch's forcing argument \cite{lellouch} that creates an admissible path $\wt\beta$ such that $R_0^{-1}\wt\beta$ and $R_1^{-1}\wt \beta$ intersect $\wt \F$-transversally; again Theorem~\ref{thmMlct2} allows us to create a new periodic point whose tracking geodesic crosses $\gamma$.

\begin{figure}
\begin{center}

\tikzset{every picture/.style={line width=0.75pt}} 

\begin{tikzpicture}[x=0.75pt,y=0.75pt,yscale=-.9,xscale=.92]

\draw  [draw opacity=0][fill={rgb, 255:red, 245; green, 166; blue, 35 }  ,fill opacity=0.15 ] (100,102) .. controls (100,89.85) and (109.85,80) .. (122,80) -- (298,80) .. controls (310.15,80) and (320,89.85) .. (320,102) -- (320,168) .. controls (320,180.15) and (310.15,190) .. (298,190) -- (122,190) .. controls (109.85,190) and (100,180.15) .. (100,168) -- cycle ;
\draw [color={rgb, 255:red, 245; green, 166; blue, 35 }  ,draw opacity=1 ][fill={rgb, 255:red, 255; green, 255; blue, 255 }  ,fill opacity=1 ]   (210,80) .. controls (231.83,105.46) and (256.33,113.96) .. (260,80) ;
\draw [color={rgb, 255:red, 245; green, 166; blue, 35 }  ,draw opacity=1 ][fill={rgb, 255:red, 255; green, 255; blue, 255 }  ,fill opacity=1 ]   (270,80) .. controls (268.83,103.46) and (287.83,113.96) .. (300,80) ;
\draw [color={rgb, 255:red, 245; green, 166; blue, 35 }  ,draw opacity=1 ][fill={rgb, 255:red, 255; green, 255; blue, 255 }  ,fill opacity=1 ]   (170,80) .. controls (168.83,103.46) and (187.83,113.96) .. (200,80) ;
\draw [color={rgb, 255:red, 245; green, 166; blue, 35 }  ,draw opacity=1 ][fill={rgb, 255:red, 255; green, 255; blue, 255 }  ,fill opacity=1 ]   (110,80) .. controls (131.83,105.46) and (156.33,113.96) .. (160,80) ;
\draw  [draw opacity=0][fill={rgb, 255:red, 255; green, 255; blue, 255 }  ,fill opacity=1 ] (97.14,75.57) .. controls (97.14,70.6) and (101.17,66.57) .. (106.14,66.57) .. controls (111.11,66.57) and (115.14,70.6) .. (115.14,75.57) .. controls (115.14,80.54) and (111.11,84.57) .. (106.14,84.57) .. controls (101.17,84.57) and (97.14,80.54) .. (97.14,75.57) -- cycle ;
\draw [color={rgb, 255:red, 245; green, 166; blue, 35 }  ,draw opacity=1 ][fill={rgb, 255:red, 255; green, 255; blue, 255 }  ,fill opacity=1 ]   (250,190) .. controls (240.14,152.67) and (280.81,164.33) .. (300,190) ;
\draw [color={rgb, 255:red, 245; green, 166; blue, 35 }  ,draw opacity=1 ][fill={rgb, 255:red, 255; green, 255; blue, 255 }  ,fill opacity=1 ]   (120,190) .. controls (110.14,152.67) and (150.81,164.33) .. (170,190) ;
\draw [color={rgb, 255:red, 0; green, 80; blue, 235 }  ,draw opacity=1 ]   (131.42,188.88) .. controls (156.66,139.72) and (275.65,144.45) .. (285.65,85.65) ;
\draw [shift={(216.3,140.14)}, rotate = 160.22] [fill={rgb, 255:red, 0; green, 80; blue, 235 }  ,fill opacity=1 ][line width=0.08]  [draw opacity=0] (8.04,-3.86) -- (0,0) -- (8.04,3.86) -- (5.34,0) -- cycle    ;
\draw [color={rgb, 255:red, 245; green, 166; blue, 35 }  ,draw opacity=1 ][fill={rgb, 255:red, 255; green, 255; blue, 255 }  ,fill opacity=1 ]   (190.57,190.57) .. controls (193.37,164.37) and (208.51,151.8) .. (220.57,190.57) ;
\draw [color={rgb, 255:red, 74; green, 144; blue, 226 }  ,draw opacity=1 ]   (257.16,186.51) .. controls (282.41,137.36) and (220.52,142.2) .. (230.52,83.4) ;
\draw [shift={(244.52,132.91)}, rotate = 48.39] [fill={rgb, 255:red, 74; green, 144; blue, 226 }  ,fill opacity=1 ][line width=0.08]  [draw opacity=0] (8.04,-3.86) -- (0,0) -- (8.04,3.86) -- (5.34,0) -- cycle    ;
\draw  [draw opacity=0][fill={rgb, 255:red, 245; green, 166; blue, 35 }  ,fill opacity=0.15 ] (357.5,103.17) .. controls (357.5,91.02) and (367.35,81.17) .. (379.5,81.17) -- (555.5,81.17) .. controls (567.65,81.17) and (577.5,91.02) .. (577.5,103.17) -- (577.5,169.17) .. controls (577.5,181.32) and (567.65,191.17) .. (555.5,191.17) -- (379.5,191.17) .. controls (367.35,191.17) and (357.5,181.32) .. (357.5,169.17) -- cycle ;
\draw [color={rgb, 255:red, 245; green, 166; blue, 35 }  ,draw opacity=1 ][fill={rgb, 255:red, 255; green, 255; blue, 255 }  ,fill opacity=1 ]   (467.5,81.17) .. controls (489.33,106.63) and (513.83,115.13) .. (517.5,81.17) ;
\draw [color={rgb, 255:red, 245; green, 166; blue, 35 }  ,draw opacity=1 ][fill={rgb, 255:red, 255; green, 255; blue, 255 }  ,fill opacity=1 ]   (527.5,81.17) .. controls (526.33,104.63) and (545.33,115.13) .. (557.5,81.17) ;
\draw [color={rgb, 255:red, 245; green, 166; blue, 35 }  ,draw opacity=1 ][fill={rgb, 255:red, 255; green, 255; blue, 255 }  ,fill opacity=1 ]   (427.5,81.17) .. controls (426.33,104.63) and (445.33,115.13) .. (457.5,81.17) ;
\draw [color={rgb, 255:red, 245; green, 166; blue, 35 }  ,draw opacity=1 ][fill={rgb, 255:red, 255; green, 255; blue, 255 }  ,fill opacity=1 ]   (367.5,81.17) .. controls (389.33,106.63) and (413.83,115.13) .. (417.5,81.17) ;
\draw  [draw opacity=0][fill={rgb, 255:red, 255; green, 255; blue, 255 }  ,fill opacity=1 ] (354.64,76.74) .. controls (354.64,71.77) and (358.67,67.74) .. (363.64,67.74) .. controls (368.61,67.74) and (372.64,71.77) .. (372.64,76.74) .. controls (372.64,81.71) and (368.61,85.74) .. (363.64,85.74) .. controls (358.67,85.74) and (354.64,81.71) .. (354.64,76.74) -- cycle ;
\draw [color={rgb, 255:red, 245; green, 166; blue, 35 }  ,draw opacity=1 ][fill={rgb, 255:red, 255; green, 255; blue, 255 }  ,fill opacity=1 ]   (488.5,191.5) .. controls (478.64,154.17) and (519.31,165.83) .. (538.5,191.5) ;
\draw [color={rgb, 255:red, 245; green, 166; blue, 35 }  ,draw opacity=1 ][fill={rgb, 255:red, 255; green, 255; blue, 255 }  ,fill opacity=1 ]   (388.17,191.5) .. controls (378.31,154.17) and (418.97,165.83) .. (438.17,191.5) ;
\draw [color={rgb, 255:red, 245; green, 166; blue, 35 }  ,draw opacity=1 ][fill={rgb, 255:red, 255; green, 255; blue, 255 }  ,fill opacity=1 ]   (448.07,191.74) .. controls (450.87,165.54) and (466.01,152.96) .. (478.07,191.74) ;
\draw [color={rgb, 255:red, 74; green, 144; blue, 226 }  ,draw opacity=1 ]   (514.29,189.88) .. controls (539.54,140.72) and (478.02,143.36) .. (488.02,84.56) ;
\draw [shift={(502.21,135.34)}, rotate = 49.76] [fill={rgb, 255:red, 74; green, 144; blue, 226 }  ,fill opacity=1 ][line width=0.08]  [draw opacity=0] (8.04,-3.86) -- (0,0) -- (8.04,3.86) -- (5.34,0) -- cycle    ;
\draw [color={rgb, 255:red, 0; green, 80; blue, 235 }  ,draw opacity=1 ]   (459.69,187.14) .. controls (484.94,137.99) and (433.03,146.27) .. (443.03,87.47) ;
\draw [shift={(452.89,135.28)}, rotate = 53.46] [fill={rgb, 255:red, 0; green, 80; blue, 235 }  ,fill opacity=1 ][line width=0.08]  [draw opacity=0] (8.04,-3.86) -- (0,0) -- (8.04,3.86) -- (5.34,0) -- cycle    ;
\draw [color={rgb, 255:red, 74; green, 144; blue, 226 }  ,draw opacity=1 ]   (412.79,189.88) .. controls (438.04,140.72) and (380.02,145.36) .. (390.02,86.56) ;
\draw [shift={(402.51,136.16)}, rotate = 51.02] [fill={rgb, 255:red, 74; green, 144; blue, 226 }  ,fill opacity=1 ][line width=0.08]  [draw opacity=0] (8.04,-3.86) -- (0,0) -- (8.04,3.86) -- (5.34,0) -- cycle    ;
\draw [color={rgb, 255:red, 144; green, 19; blue, 254 }  ,draw opacity=1 ]   (357.36,139.14) .. controls (494.67,117.88) and (555.61,146.81) .. (578.17,130.47) ;
\draw [shift={(472.42,131.11)}, rotate = 180.86] [fill={rgb, 255:red, 144; green, 19; blue, 254 }  ,fill opacity=1 ][line width=0.08]  [draw opacity=0] (8.04,-3.86) -- (0,0) -- (8.04,3.86) -- (5.34,0) -- cycle    ;
\draw [color={rgb, 255:red, 65; green, 117; blue, 5 }  ,draw opacity=1 ]   (465.17,186.88) .. controls (490.41,137.72) and (540.67,155.88) .. (546.67,87.47) ;
\draw [shift={(517.52,141.58)}, rotate = 147.4] [fill={rgb, 255:red, 65; green, 117; blue, 5 }  ,fill opacity=1 ][line width=0.08]  [draw opacity=0] (8.04,-3.86) -- (0,0) -- (8.04,3.86) -- (5.34,0) -- cycle    ;

\draw (114.33,119.15) node [anchor=north west][inner sep=0.75pt]  [color={rgb, 255:red, 199; green, 135; blue, 28 }  ,opacity=1 ,xscale=1.2,yscale=1.2]  {$\wt{B}$};
\draw (136.55,191.71) node [anchor=north] [inner sep=0.75pt]  [font=\small,color={rgb, 255:red, 0; green, 80; blue, 235 }  ,opacity=1 ,xscale=1.2,yscale=1.2]  {$R_{1}^{-1} I_{\wt{\F}}^{[ 0,n_{0}]}(\wt{y}_{0})$};
\draw (220,82.06) node [anchor=south] [inner sep=0.75pt]  [font=\small,color={rgb, 255:red, 74; green, 144; blue, 226 }  ,opacity=1 ,xscale=1.2,yscale=1.2]  {$R_{0}^{-1} I_{\wt{\F}}^{[ 0,n_{0}]}(\wt{y}_{0})$};
\draw (365.17,99.65) node [anchor=north west][inner sep=0.75pt]  [color={rgb, 255:red, 199; green, 135; blue, 28 }  ,opacity=1 ,xscale=1.2,yscale=1.2]  {$\wt{B}$};
\draw (446.21,83.65) node [anchor=south] [inner sep=0.75pt]  [font=\small,color={rgb, 255:red, 0; green, 80; blue, 235 }  ,opacity=1 ,xscale=1.2,yscale=1.2]  {$R_{1}^{-1} I_{\wt{\F}}^{[ 0,n_{0}]}(\wt{y}_{0})$};
\draw (555,191.96) node [anchor=north] [inner sep=0.75pt]  [font=\small,color={rgb, 255:red, 74; green, 144; blue, 226 }  ,opacity=1 ,xscale=1.2,yscale=1.2]  {$TR_{0}^{-1} I_{\wt{\F}}^{[ 0,n_{0}]}(\wt{y}_{0})$};
\draw (390,193.96) node [anchor=north] [inner sep=0.75pt]  [font=\small,color={rgb, 255:red, 74; green, 144; blue, 226 }  ,opacity=1 ,xscale=1.2,yscale=1.2]  {$R_{0}^{-1} I_{\wt{\F}}^{[ 0,n_{0}]}(\wt{y}_{0})$};
\draw (579.33,125.26) node [anchor=west] [inner sep=0.75pt]  [font=\small,color={rgb, 255:red, 144; green, 19; blue, 254 }  ,opacity=1 ,xscale=1.2,yscale=1.2]  {$I_{\wt{F}}^{\Z}(\wt{z})$};
\draw (545.06,79.83) node [anchor=south] [inner sep=0.75pt]  [color={rgb, 255:red, 65; green, 117; blue, 5 }  ,opacity=1 ,xscale=1.2,yscale=1.2]  {$\wt{\beta }$};

\end{tikzpicture}
\caption{Idea of the proof of Proposition~\ref{PropBndedDevRatCase1}.}\label{FigIdeaPropBndedDevRatCase1}

\end{center}
\end{figure}
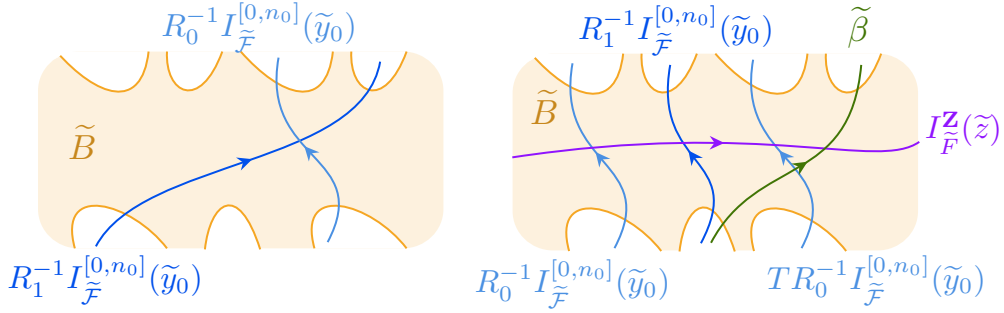

\bigskip

By replacing $\wt B$ by another lift of it, one can suppose $R_0 = \Id_{\wt S}$

We treat the case where $\wt\alpha_0$ accumulates positively in $I^{[0,n_0]}_{\wt \F}(\wt y_0)$, the case where $\wt\alpha_0$ accumulates negatively in $I^{[0,n_0]}_{\wt \F}(\wt y_0)$ being identical.

By Proposition~\ref{Prop3.3GLCP}.1, the band $\wt B$ projects to a simple annulus of $S$; in particular the geodesic $\gamma$ has to be simple.

Without loss of generality, we can suppose that $I^{[0,n_0]}_{\wt \F}(\wt y_0)$ crosses $\wt B$ from right to left (the other case being identical). By hypothesis, for $j={-1,1,2}$, the transverse path $I^{[0,n_0]}_{\wt \F}(\wt y_0)$ crosses $R_{j}\wt B$; by Proposition~\ref{Prop3.3GLCP}.3 it crosses it from right to left, and it can cross it at most once.
For $j={-1,0,1,2}$, denote $a_j<b_j\in [0,n_0]$ some numbers such that $I^{a_j}_{\wt \F}(\wt y_0) \in R_j\partial \wt B^L$ and $I^{b_j}_{\wt \F}(\wt y_0) \in R_j\partial \wt B^R$. Note that, because $\wt b$ projects to an annulus of $S$ (Proposition~\ref{Prop3.3GLCP}), for $j=-1,0,1$ one has $b_j<a_{j+1}$.

The objects of the next lemma are depicted in Figure~\ref{FigLemLeavesDisjointTraj}.

\begin{lemma}\label{LemLeavesDisjointTraj}
There exists a leaf $\wt\phi_0\subset \wt B$ such that:
\begin{itemize}
\item $R_{-1}\wt\phi_0\subset R(\wt\alpha_0)$ and $R_{1}\wt\phi_0\subset L(\wt\alpha_0)$;
\item $I^{[0,n_0]}_{\wt \F}(\wt y_0)$ meets $R_{-1}\wt\phi_0$, $R_{1}\wt\phi_0$ and $R_{2}\wt\phi_0$;
\item For any $i\in\{-3,-2,-1,1,2,3\}$, the leaves $R_{-1}\wt\phi_0$ and $R_{1}\wt\phi_0$ are disjoint from the trajectory $T^iI^{[0,n_0]}_{\wt \F}(\wt y_0)$, and the  leaves $R_{1}^{-1}\wt\phi_0$ and $R_{1}^{-1} R_2\wt\phi_0$ are disjoint from the trajectory $T^iR_1^{-1}I^{[0,n_0]}_{\wt \F}(\wt y_0)$.
\end{itemize}
\end{lemma}

\begin{proof}
The idea is to choose $\wt\phi_0 = \wt\phi_{\wt\alpha_0(t)}$ for $t$ large enough.

\begin{figure}
\begin{center}
\tikzset{every picture/.style={line width=0.75pt}} 

\begin{tikzpicture}[x=0.75pt,y=0.75pt,yscale=-1.4,xscale=1.45]

\draw  [draw opacity=0][fill={rgb, 255:red, 245; green, 166; blue, 35 }  ,fill opacity=0.2 ] (137,104) .. controls (147.37,103.95) and (147.06,104.06) .. (157,104) .. controls (166.48,133.73) and (183.98,141.98) .. (197,104) .. controls (207.29,104.18) and (207.73,103.95) .. (217,104) .. controls (209.23,168.23) and (147.48,177.73) .. (137,104) -- cycle ;
\draw  [draw opacity=0][fill={rgb, 255:red, 245; green, 166; blue, 35 }  ,fill opacity=0.2 ] (237,104) .. controls (247.37,103.95) and (247.06,104.06) .. (257,104) .. controls (266.48,133.73) and (283.98,141.98) .. (297,104) .. controls (307.29,104.18) and (307.73,103.95) .. (317,104) .. controls (309.23,168.23) and (247.48,177.73) .. (237,104) -- cycle ;
\draw  [draw opacity=0][fill={rgb, 255:red, 245; green, 166; blue, 35 }  ,fill opacity=0.2 ] (337,104) .. controls (347.37,103.95) and (347.06,104.06) .. (357,104) .. controls (366.48,133.73) and (383.98,141.98) .. (397,104) .. controls (407.29,104.18) and (407.73,103.95) .. (417,104) .. controls (409.23,168.23) and (347.48,177.73) .. (337,104) -- cycle ;
\draw  [draw opacity=0][fill={rgb, 255:red, 245; green, 166; blue, 35 }  ,fill opacity=0.2 ] (137,284) .. controls (147.37,283.95) and (147.06,284.06) .. (157,284) .. controls (167.51,255.65) and (184.08,239.94) .. (197,284) .. controls (207.29,284.18) and (207.73,283.95) .. (217,284) .. controls (197.51,209.65) and (147.51,210.8) .. (137,284) -- cycle ;
\draw  [draw opacity=0][fill={rgb, 255:red, 245; green, 166; blue, 35 }  ,fill opacity=0.2 ] (127,174) .. controls (126.94,183.94) and (127.51,203.37) .. (127,214) .. controls (157.03,205.37) and (197.6,224.51) .. (227,214) .. controls (257.6,204.8) and (297.03,223.66) .. (327,214) .. controls (356.75,205.09) and (396.75,223.66) .. (427,214) .. controls (427.13,204.23) and (427.13,183.94) .. (427,174) .. controls (407.89,164.8) and (347.89,184.8) .. (327,174) .. controls (307.6,165.09) and (247.87,183.44) .. (227,174) .. controls (206.18,163.66) and (148.18,182.8) .. (127,174) -- cycle ;
\draw  [draw opacity=0][fill={rgb, 255:red, 245; green, 166; blue, 35 }  ,fill opacity=0.2 ] (237,284) .. controls (247.37,283.95) and (247.06,284.06) .. (257,284) .. controls (267.51,255.65) and (284.08,239.94) .. (297,284) .. controls (307.29,284.18) and (307.73,283.95) .. (317,284) .. controls (297.51,209.65) and (247.51,210.8) .. (237,284) -- cycle ;
\draw  [draw opacity=0][fill={rgb, 255:red, 245; green, 166; blue, 35 }  ,fill opacity=0.2 ] (337,284) .. controls (347.37,283.95) and (347.06,284.06) .. (357,284) .. controls (367.51,255.65) and (384.08,239.94) .. (397,284) .. controls (407.29,284.18) and (407.73,283.95) .. (417,284) .. controls (397.51,209.65) and (347.51,210.8) .. (337,284) -- cycle ;
\draw [color={rgb, 255:red, 65; green, 117; blue, 5 }  ,draw opacity=1 ]   (127,194) .. controls (159.08,184.55) and (199.48,203.75) .. (227,194) ;
\draw [color={rgb, 255:red, 65; green, 117; blue, 5 }  ,draw opacity=1 ]   (227,194) .. controls (259.08,184.55) and (299.48,203.75) .. (327,194) ;
\draw [shift={(279.76,194.26)}, rotate = 187.73] [fill={rgb, 255:red, 65; green, 117; blue, 5 }  ,fill opacity=1 ][line width=0.08]  [draw opacity=0] (7.14,-3.43) -- (0,0) -- (7.14,3.43) -- (4.74,0) -- cycle    ;
\draw [color={rgb, 255:red, 65; green, 117; blue, 5 }  ,draw opacity=1 ]   (327,194) .. controls (359.08,184.55) and (399.48,203.75) .. (427,194) ;
\draw [color={rgb, 255:red, 0; green, 120; blue, 255 }  ,draw opacity=1 ]   (177,104) .. controls (184.28,225.75) and (229.88,103.35) .. (250.68,134.15) .. controls (271.48,164.95) and (135.88,201.75) .. (177,284) ;
\draw [shift={(194.46,159.05)}, rotate = 352.82] [fill={rgb, 255:red, 0; green, 120; blue, 255 }  ,fill opacity=1 ][line width=0.08]  [draw opacity=0] (7.14,-3.43) -- (0,0) -- (7.14,3.43) -- (4.74,0) -- cycle    ;
\draw [shift={(196.98,198.47)}, rotate = 134.12] [fill={rgb, 255:red, 0; green, 120; blue, 255 }  ,fill opacity=1 ][line width=0.08]  [draw opacity=0] (7.14,-3.43) -- (0,0) -- (7.14,3.43) -- (4.74,0) -- cycle    ;
\draw [color={rgb, 255:red, 0; green, 120; blue, 255 }  ,draw opacity=1 ]   (277.4,104) .. controls (284.68,225.75) and (330.28,103.35) .. (351.08,134.15) .. controls (371.88,164.95) and (236.28,201.75) .. (277.4,284) ;
\draw [shift={(294.86,159.05)}, rotate = 352.82] [fill={rgb, 255:red, 0; green, 120; blue, 255 }  ,fill opacity=1 ][line width=0.08]  [draw opacity=0] (7.14,-3.43) -- (0,0) -- (7.14,3.43) -- (4.74,0) -- cycle    ;
\draw [shift={(297.38,198.47)}, rotate = 134.12] [fill={rgb, 255:red, 0; green, 120; blue, 255 }  ,fill opacity=1 ][line width=0.08]  [draw opacity=0] (7.14,-3.43) -- (0,0) -- (7.14,3.43) -- (4.74,0) -- cycle    ;
\draw [color={rgb, 255:red, 0; green, 120; blue, 255 }  ,draw opacity=1 ]   (377,104) .. controls (384.28,225.75) and (429.88,103.35) .. (450.68,134.15) .. controls (471.48,164.95) and (335.88,201.75) .. (377,284) ;
\draw [shift={(394.46,159.05)}, rotate = 352.82] [fill={rgb, 255:red, 0; green, 120; blue, 255 }  ,fill opacity=1 ][line width=0.08]  [draw opacity=0] (7.14,-3.43) -- (0,0) -- (7.14,3.43) -- (4.74,0) -- cycle    ;
\draw [shift={(396.98,198.47)}, rotate = 134.12] [fill={rgb, 255:red, 0; green, 120; blue, 255 }  ,fill opacity=1 ][line width=0.08]  [draw opacity=0] (7.14,-3.43) -- (0,0) -- (7.14,3.43) -- (4.74,0) -- cycle    ;
\draw [color={rgb, 255:red, 255; green, 158; blue, 0 }  ,draw opacity=1 ]   (142.17,104.5) .. controls (164.97,151.3) and (187.33,141.17) .. (199.33,104.5) ;
\draw [color={rgb, 255:red, 255; green, 158; blue, 0 }  ,draw opacity=1 ]   (241.5,104.33) .. controls (264.3,151.13) and (286.67,141) .. (298.67,104.33) ;
\draw [color={rgb, 255:red, 255; green, 158; blue, 0 }  ,draw opacity=1 ]   (341.5,104.33) .. controls (364.3,151.13) and (386.67,141) .. (398.67,104.33) ;
\draw [color={rgb, 255:red, 255; green, 158; blue, 0 }  ,draw opacity=1 ]   (151.04,283.82) .. controls (162.15,231.59) and (186.9,205.79) .. (214.17,284) ;
\draw [color={rgb, 255:red, 65; green, 117; blue, 5 }  ,draw opacity=1 ]   (147,104) .. controls (174.3,164.23) and (196.3,148.01) .. (207,104) ;
\draw [shift={(181.67,143.3)}, rotate = 189.23] [fill={rgb, 255:red, 65; green, 117; blue, 5 }  ,fill opacity=1 ][line width=0.08]  [draw opacity=0] (7.14,-3.43) -- (0,0) -- (7.14,3.43) -- (4.74,0) -- cycle    ;
\draw [color={rgb, 255:red, 65; green, 117; blue, 5 }  ,draw opacity=1 ]   (247,104) .. controls (274.3,164.23) and (296.3,148.01) .. (307,104) ;
\draw [shift={(281.67,143.3)}, rotate = 189.23] [fill={rgb, 255:red, 65; green, 117; blue, 5 }  ,fill opacity=1 ][line width=0.08]  [draw opacity=0] (7.14,-3.43) -- (0,0) -- (7.14,3.43) -- (4.74,0) -- cycle    ;
\draw [color={rgb, 255:red, 65; green, 117; blue, 5 }  ,draw opacity=1 ]   (347,104) .. controls (374.3,164.23) and (396.3,148.01) .. (407,104) ;
\draw [shift={(381.67,143.3)}, rotate = 189.23] [fill={rgb, 255:red, 65; green, 117; blue, 5 }  ,fill opacity=1 ][line width=0.08]  [draw opacity=0] (7.14,-3.43) -- (0,0) -- (7.14,3.43) -- (4.74,0) -- cycle    ;
\draw [color={rgb, 255:red, 65; green, 117; blue, 5 }  ,draw opacity=1 ]   (147,284) .. controls (159.79,217.34) and (194.9,236.01) .. (207,284) ;
\draw [shift={(178.33,241.25)}, rotate = 184.18] [fill={rgb, 255:red, 65; green, 117; blue, 5 }  ,fill opacity=1 ][line width=0.08]  [draw opacity=0] (7.14,-3.43) -- (0,0) -- (7.14,3.43) -- (4.74,0) -- cycle    ;
\draw [color={rgb, 255:red, 65; green, 117; blue, 5 }  ,draw opacity=1 ]   (247,284) .. controls (259.79,217.34) and (294.9,236.01) .. (307,284) ;
\draw [shift={(278.33,241.25)}, rotate = 184.18] [fill={rgb, 255:red, 65; green, 117; blue, 5 }  ,fill opacity=1 ][line width=0.08]  [draw opacity=0] (7.14,-3.43) -- (0,0) -- (7.14,3.43) -- (4.74,0) -- cycle    ;
\draw [color={rgb, 255:red, 65; green, 117; blue, 5 }  ,draw opacity=1 ]   (347,284) .. controls (359.79,217.34) and (394.9,236.01) .. (407,284) ;
\draw [shift={(378.33,241.25)}, rotate = 184.18] [fill={rgb, 255:red, 65; green, 117; blue, 5 }  ,fill opacity=1 ][line width=0.08]  [draw opacity=0] (7.14,-3.43) -- (0,0) -- (7.14,3.43) -- (4.74,0) -- cycle    ;
\draw [color={rgb, 255:red, 255; green, 158; blue, 0 }  ,draw opacity=1 ]   (250.77,283.72) .. controls (261.88,231.49) and (286.63,205.69) .. (313.9,283.9) ;
\draw [color={rgb, 255:red, 255; green, 158; blue, 0 }  ,draw opacity=1 ]   (351.44,284.25) .. controls (362.55,232.03) and (387.3,206.22) .. (414.57,284.43) ;
\draw [color={rgb, 255:red, 245; green, 166; blue, 35 }  ,draw opacity=1 ]   (127.54,181) .. controls (197.92,177.92) and (227.03,183.66) .. (255.25,199.25) .. controls (283.47,214.84) and (386.58,204.92) .. (427.25,208.92) ;
\draw [shift={(188.7,181.1)}, rotate = 4.73] [fill={rgb, 255:red, 245; green, 166; blue, 35 }  ,fill opacity=1 ][line width=0.08]  [draw opacity=0] (8.04,-3.86) -- (0,0) -- (8.04,3.86) -- (5.34,0) -- cycle    ;
\draw [shift={(336.58,208.43)}, rotate = 359.81] [fill={rgb, 255:red, 245; green, 166; blue, 35 }  ,fill opacity=1 ][line width=0.08]  [draw opacity=0] (8.04,-3.86) -- (0,0) -- (8.04,3.86) -- (5.34,0) -- cycle    ;
\draw [color={rgb, 255:red, 0; green, 120; blue, 255 }  ,draw opacity=1 ]   (276.33,50.06) .. controls (279.83,83.06) and (274.86,94.45) .. (279.09,123.62) ;
\draw [shift={(277.58,82.47)}, rotate = 91.43] [fill={rgb, 255:red, 0; green, 120; blue, 255 }  ,fill opacity=1 ][line width=0.08]  [draw opacity=0] (8.04,-3.86) -- (0,0) -- (8.04,3.86) -- (5.34,0) -- cycle    ;
\draw  [draw opacity=0][fill={rgb, 255:red, 245; green, 166; blue, 35 }  ,fill opacity=0.2 ] (259.13,50.2) .. controls (263.58,50.25) and (263.51,50.25) .. (267.66,50.2) .. controls (270.18,76.38) and (282.59,84.33) .. (285.74,50.34) .. controls (289.66,50.46) and (290.11,50.31) .. (293.93,50.46) .. controls (292.19,95.13) and (264.59,93.53) .. (259.13,50.2) -- cycle ;
\draw [color={rgb, 255:red, 255; green, 158; blue, 0 }  ,draw opacity=1 ]   (260.86,50.46) .. controls (267.29,74.65) and (281.66,93) .. (287.66,50.2) ;
\draw [color={rgb, 255:red, 65; green, 117; blue, 5 }  ,draw opacity=1 ]   (263.22,49.67) .. controls (277.09,86.8) and (281.57,94.46) .. (292.27,50.45) ;
\draw [shift={(281.05,80.13)}, rotate = 199.79] [fill={rgb, 255:red, 65; green, 117; blue, 5 }  ,fill opacity=1 ][line width=0.08]  [draw opacity=0] (7.14,-3.43) -- (0,0) -- (7.14,3.43) -- (4.74,0) -- cycle    ;

\draw (122.33,212.4) node [anchor=north west][inner sep=0.75pt]  [color={rgb, 255:red, 176; green, 109; blue, 0 }  ,opacity=1 ,xscale=1.2,yscale=1.2]  {$\wt{B}$};
\draw (312.77,100.87) node [anchor=south] [inner sep=0.75pt]  [font=\small,color={rgb, 255:red, 176; green, 109; blue, 0 }  ,opacity=1 ,xscale=1.2,yscale=1.2]  {$R_{1}\wt{B}$};
\draw (417,100.6) node [anchor=south] [inner sep=0.75pt]  [font=\small,color={rgb, 255:red, 176; green, 109; blue, 0 }  ,opacity=1 ,xscale=1.2,yscale=1.2]  {$TR_{1}\wt{B}$};
\draw (166,102.6) node [anchor=south east] [inner sep=0.75pt]  [font=\small,color={rgb, 255:red, 176; green, 109; blue, 0 }  ,opacity=1 ,xscale=1.2,yscale=1.2]  {$T^{-1} R_{1}\wt{B}$};
\draw (159.19,283.6) node [anchor=north east] [inner sep=0.75pt]  [font=\small,color={rgb, 255:red, 176; green, 109; blue, 0 }  ,opacity=1 ,xscale=1.2,yscale=1.2]  {$T^{-1} R_{-1}\wt{B}$};
\draw (309.77,281.8) node [anchor=north] [inner sep=0.75pt]  [font=\small,color={rgb, 255:red, 176; green, 109; blue, 0 }  ,opacity=1 ,xscale=1.2,yscale=1.2]  {$R_{-1}\wt{B}$};
\draw (405.07,281.13) node [anchor=north west][inner sep=0.75pt]  [font=\small,color={rgb, 255:red, 176; green, 109; blue, 0 }  ,opacity=1 ,xscale=1.2,yscale=1.2]  {$TR_{-1}\wt{B}$};
\draw (283.6,192.6) node [anchor=south] [inner sep=0.75pt]  [color={rgb, 255:red, 56; green, 104; blue, 0 }  ,opacity=1 ,xscale=1.2,yscale=1.2]  {$\wt{\alpha }_{0}$};
\draw (266.69,284.45) node [anchor=north] [inner sep=0.75pt]  [font=\small,color={rgb, 255:red, 4; green, 90; blue, 191 }  ,opacity=1 ,xscale=1.2,yscale=1.2]  {$I_{\wt \F}^{\Z}(\wt y_{0})$};
\draw (368.81,283.2) node [anchor=north] [inner sep=0.75pt]  [font=\small,color={rgb, 255:red, 4; green, 90; blue, 191 }  ,opacity=1 ,xscale=1.2,yscale=1.2]  {$TI_{\wt \F}^{\Z}(\wt y_{0})$};
\draw (166,283.2) node [anchor=north west][inner sep=0.75pt]  [font=\small,color={rgb, 255:red, 4; green, 90; blue, 191 }  ,opacity=1 ,xscale=1.2,yscale=1.2]  {$T^{-1} I_{\wt \F}^{\Z}(\wt y_{0})$};
\draw (241.5,100.93) node [anchor=south] [inner sep=0.75pt]  [font=\small,color={rgb, 255:red, 223; green, 138; blue, 0 }  ,opacity=1 ,xscale=1.2,yscale=1.2]  {$\wt{\phi } '=R_{1}\wt{\phi }_{0}$};
\draw (292.42,241) node [anchor=south west] [inner sep=0.75pt]  [font=\small,color={rgb, 255:red, 223; green, 138; blue, 0 }  ,opacity=1 ,xscale=1.2,yscale=1.2]  {$\wt{\phi } =R_{-1}\wt{\phi }_{0}$};
\draw (138.2,174.5) node [anchor=south] [inner sep=0.75pt]  [color={rgb, 255:red, 223; green, 138; blue, 0 }  ,opacity=1 ,xscale=1.2,yscale=1.2]  {$\wt{\phi }_{0}$};
\draw (308.23,72.61) node [anchor=south] [inner sep=0.75pt]  [font=\small,color={rgb, 255:red, 176; green, 109; blue, 0 }  ,opacity=1 ,xscale=1.2,yscale=1.2]  {$R_{2}\wt{B}$};

\end{tikzpicture}\\
\vspace{10pt}
\hrule
\vspace{10pt}
\begin{tikzpicture}[x=0.75pt,y=0.75pt,yscale=-1.3,xscale=1.35]

\draw  [draw opacity=0][fill={rgb, 255:red, 245; green, 166; blue, 35 }  ,fill opacity=0.2 ] (157,67.99) .. controls (167.37,67.94) and (167.06,68.05) .. (177,67.99) .. controls (186.48,97.72) and (203.98,105.97) .. (217,67.99) .. controls (227.29,68.17) and (227.73,67.94) .. (237,67.99) .. controls (229.23,132.22) and (167.48,141.72) .. (157,67.99) -- cycle ;
\draw  [draw opacity=0][fill={rgb, 255:red, 245; green, 166; blue, 35 }  ,fill opacity=0.2 ] (257,67.99) .. controls (267.37,67.94) and (267.06,68.05) .. (277,67.99) .. controls (286.48,97.72) and (303.98,105.97) .. (317,67.99) .. controls (327.29,68.17) and (327.73,67.94) .. (337,67.99) .. controls (329.23,132.22) and (267.48,141.72) .. (257,67.99) -- cycle ;
\draw  [draw opacity=0][fill={rgb, 255:red, 245; green, 166; blue, 35 }  ,fill opacity=0.2 ] (357,67.99) .. controls (367.37,67.94) and (367.06,68.05) .. (377,67.99) .. controls (386.48,97.72) and (403.98,105.97) .. (417,67.99) .. controls (427.29,68.17) and (427.73,67.94) .. (437,67.99) .. controls (429.23,132.22) and (367.48,141.72) .. (357,67.99) -- cycle ;
\draw  [draw opacity=0][fill={rgb, 255:red, 245; green, 166; blue, 35 }  ,fill opacity=0.2 ] (157,247.99) .. controls (167.37,247.94) and (167.06,248.05) .. (177,247.99) .. controls (187.51,219.64) and (204.08,203.93) .. (217,247.99) .. controls (227.29,248.17) and (227.73,247.94) .. (237,247.99) .. controls (217.51,173.64) and (167.51,174.78) .. (157,247.99) -- cycle ;
\draw  [draw opacity=0][fill={rgb, 255:red, 245; green, 166; blue, 35 }  ,fill opacity=0.2 ] (147,137.99) .. controls (146.94,147.93) and (147.51,167.36) .. (147,177.99) .. controls (177.03,169.36) and (217.6,188.5) .. (247,177.99) .. controls (277.6,168.79) and (317.03,187.65) .. (347,177.99) .. controls (376.75,169.07) and (416.75,187.65) .. (447,177.99) .. controls (447.13,168.22) and (447.13,147.93) .. (447,137.99) .. controls (427.89,128.79) and (367.89,148.79) .. (347,137.99) .. controls (327.6,129.07) and (267.87,147.43) .. (247,137.99) .. controls (226.18,127.65) and (168.18,146.79) .. (147,137.99) -- cycle ;
\draw  [draw opacity=0][fill={rgb, 255:red, 245; green, 166; blue, 35 }  ,fill opacity=0.2 ] (257,247.99) .. controls (267.37,247.94) and (267.06,248.05) .. (277,247.99) .. controls (287.51,219.64) and (304.08,203.93) .. (317,247.99) .. controls (327.29,248.17) and (327.73,247.94) .. (337,247.99) .. controls (317.51,173.64) and (267.51,174.78) .. (257,247.99) -- cycle ;
\draw  [draw opacity=0][fill={rgb, 255:red, 245; green, 166; blue, 35 }  ,fill opacity=0.2 ] (357,247.99) .. controls (367.37,247.94) and (367.06,248.05) .. (377,247.99) .. controls (387.51,219.64) and (404.08,203.93) .. (417,247.99) .. controls (427.29,248.17) and (427.73,247.94) .. (437,247.99) .. controls (417.51,173.64) and (367.51,174.78) .. (357,247.99) -- cycle ;
\draw [color={rgb, 255:red, 65; green, 117; blue, 5 }  ,draw opacity=1 ]   (147,157.99) .. controls (179.08,148.54) and (219.48,167.74) .. (247,157.99) ;
\draw [color={rgb, 255:red, 65; green, 117; blue, 5 }  ,draw opacity=1 ]   (247,157.99) .. controls (279.08,148.54) and (319.48,167.74) .. (347,157.99) ;
\draw [shift={(299.76,158.25)}, rotate = 187.73] [fill={rgb, 255:red, 65; green, 117; blue, 5 }  ,fill opacity=1 ][line width=0.08]  [draw opacity=0] (7.14,-3.43) -- (0,0) -- (7.14,3.43) -- (4.74,0) -- cycle    ;
\draw [color={rgb, 255:red, 65; green, 117; blue, 5 }  ,draw opacity=1 ]   (347,157.99) .. controls (379.08,148.54) and (419.48,167.74) .. (447,157.99) ;
\draw [color={rgb, 255:red, 0; green, 120; blue, 255 }  ,draw opacity=1 ]   (197,67.99) .. controls (204.28,189.74) and (249.88,67.34) .. (270.68,98.14) .. controls (291.48,128.94) and (155.88,165.74) .. (197,247.99) ;
\draw [shift={(214.46,123.04)}, rotate = 352.82] [fill={rgb, 255:red, 0; green, 120; blue, 255 }  ,fill opacity=1 ][line width=0.08]  [draw opacity=0] (7.14,-3.43) -- (0,0) -- (7.14,3.43) -- (4.74,0) -- cycle    ;
\draw [shift={(216.98,162.46)}, rotate = 134.12] [fill={rgb, 255:red, 0; green, 120; blue, 255 }  ,fill opacity=1 ][line width=0.08]  [draw opacity=0] (7.14,-3.43) -- (0,0) -- (7.14,3.43) -- (4.74,0) -- cycle    ;
\draw [color={rgb, 255:red, 0; green, 120; blue, 255 }  ,draw opacity=1 ]   (297.4,67.99) .. controls (304.68,189.74) and (350.28,67.34) .. (371.08,98.14) .. controls (391.88,128.94) and (256.28,165.74) .. (297.4,247.99) ;
\draw [shift={(314.86,123.04)}, rotate = 352.82] [fill={rgb, 255:red, 0; green, 120; blue, 255 }  ,fill opacity=1 ][line width=0.08]  [draw opacity=0] (7.14,-3.43) -- (0,0) -- (7.14,3.43) -- (4.74,0) -- cycle    ;
\draw [shift={(317.38,162.46)}, rotate = 134.12] [fill={rgb, 255:red, 0; green, 120; blue, 255 }  ,fill opacity=1 ][line width=0.08]  [draw opacity=0] (7.14,-3.43) -- (0,0) -- (7.14,3.43) -- (4.74,0) -- cycle    ;
\draw [color={rgb, 255:red, 0; green, 120; blue, 255 }  ,draw opacity=1 ]   (397,67.99) .. controls (404.28,189.74) and (449.88,67.34) .. (470.68,98.14) .. controls (491.48,128.94) and (355.88,165.74) .. (397,247.99) ;
\draw [shift={(414.46,123.04)}, rotate = 352.82] [fill={rgb, 255:red, 0; green, 120; blue, 255 }  ,fill opacity=1 ][line width=0.08]  [draw opacity=0] (7.14,-3.43) -- (0,0) -- (7.14,3.43) -- (4.74,0) -- cycle    ;
\draw [shift={(416.98,162.46)}, rotate = 134.12] [fill={rgb, 255:red, 0; green, 120; blue, 255 }  ,fill opacity=1 ][line width=0.08]  [draw opacity=0] (7.14,-3.43) -- (0,0) -- (7.14,3.43) -- (4.74,0) -- cycle    ;
\draw [color={rgb, 255:red, 255; green, 158; blue, 0 }  ,draw opacity=1 ]   (162.17,68.49) .. controls (184.97,115.29) and (207.33,105.16) .. (219.33,68.49) ;
\draw [color={rgb, 255:red, 255; green, 158; blue, 0 }  ,draw opacity=1 ]   (261.5,68.32) .. controls (284.3,115.12) and (306.67,104.99) .. (318.67,68.32) ;
\draw [color={rgb, 255:red, 255; green, 158; blue, 0 }  ,draw opacity=1 ]   (361.5,68.32) .. controls (384.3,115.12) and (406.67,104.99) .. (418.67,68.32) ;
\draw [color={rgb, 255:red, 255; green, 158; blue, 0 }  ,draw opacity=1 ]   (171.04,247.8) .. controls (182.15,195.58) and (206.9,169.77) .. (234.17,247.99) ;
\draw [color={rgb, 255:red, 65; green, 117; blue, 5 }  ,draw opacity=1 ]   (167,67.99) .. controls (194.3,128.22) and (216.3,112) .. (227,67.99) ;
\draw [shift={(201.67,107.29)}, rotate = 189.23] [fill={rgb, 255:red, 65; green, 117; blue, 5 }  ,fill opacity=1 ][line width=0.08]  [draw opacity=0] (7.14,-3.43) -- (0,0) -- (7.14,3.43) -- (4.74,0) -- cycle    ;
\draw [color={rgb, 255:red, 65; green, 117; blue, 5 }  ,draw opacity=1 ]   (267,67.99) .. controls (294.3,128.22) and (316.3,112) .. (327,67.99) ;
\draw [shift={(301.67,107.29)}, rotate = 189.23] [fill={rgb, 255:red, 65; green, 117; blue, 5 }  ,fill opacity=1 ][line width=0.08]  [draw opacity=0] (7.14,-3.43) -- (0,0) -- (7.14,3.43) -- (4.74,0) -- cycle    ;
\draw [color={rgb, 255:red, 65; green, 117; blue, 5 }  ,draw opacity=1 ]   (367,67.99) .. controls (394.3,128.22) and (416.3,112) .. (427,67.99) ;
\draw [shift={(401.67,107.29)}, rotate = 189.23] [fill={rgb, 255:red, 65; green, 117; blue, 5 }  ,fill opacity=1 ][line width=0.08]  [draw opacity=0] (7.14,-3.43) -- (0,0) -- (7.14,3.43) -- (4.74,0) -- cycle    ;
\draw [color={rgb, 255:red, 65; green, 117; blue, 5 }  ,draw opacity=1 ]   (167,247.99) .. controls (179.79,181.33) and (214.9,200) .. (227,247.99) ;
\draw [shift={(198.33,205.24)}, rotate = 184.18] [fill={rgb, 255:red, 65; green, 117; blue, 5 }  ,fill opacity=1 ][line width=0.08]  [draw opacity=0] (7.14,-3.43) -- (0,0) -- (7.14,3.43) -- (4.74,0) -- cycle    ;
\draw [color={rgb, 255:red, 65; green, 117; blue, 5 }  ,draw opacity=1 ]   (267,247.99) .. controls (279.79,181.33) and (314.9,200) .. (327,247.99) ;
\draw [shift={(298.33,205.24)}, rotate = 184.18] [fill={rgb, 255:red, 65; green, 117; blue, 5 }  ,fill opacity=1 ][line width=0.08]  [draw opacity=0] (7.14,-3.43) -- (0,0) -- (7.14,3.43) -- (4.74,0) -- cycle    ;
\draw [color={rgb, 255:red, 65; green, 117; blue, 5 }  ,draw opacity=1 ]   (367,247.99) .. controls (379.79,181.33) and (414.9,200) .. (427,247.99) ;
\draw [shift={(398.33,205.24)}, rotate = 184.18] [fill={rgb, 255:red, 65; green, 117; blue, 5 }  ,fill opacity=1 ][line width=0.08]  [draw opacity=0] (7.14,-3.43) -- (0,0) -- (7.14,3.43) -- (4.74,0) -- cycle    ;
\draw [color={rgb, 255:red, 255; green, 158; blue, 0 }  ,draw opacity=1 ]   (270.77,247.7) .. controls (281.88,195.48) and (306.63,169.67) .. (333.9,247.89) ;
\draw [color={rgb, 255:red, 255; green, 158; blue, 0 }  ,draw opacity=1 ]   (371.44,248.24) .. controls (382.55,196.02) and (407.3,170.21) .. (434.57,248.42) ;
\draw [color={rgb, 255:red, 245; green, 166; blue, 35 }  ,draw opacity=1 ]   (147.54,144.99) .. controls (217.92,141.91) and (247.03,147.65) .. (275.25,163.24) .. controls (303.47,178.83) and (406.58,168.91) .. (447.25,172.91) ;
\draw [shift={(208.7,145.09)}, rotate = 4.73] [fill={rgb, 255:red, 245; green, 166; blue, 35 }  ,fill opacity=1 ][line width=0.08]  [draw opacity=0] (8.04,-3.86) -- (0,0) -- (8.04,3.86) -- (5.34,0) -- cycle    ;
\draw [shift={(356.58,172.42)}, rotate = 359.81] [fill={rgb, 255:red, 245; green, 166; blue, 35 }  ,fill opacity=1 ][line width=0.08]  [draw opacity=0] (8.04,-3.86) -- (0,0) -- (8.04,3.86) -- (5.34,0) -- cycle    ;

\draw (172.17,174.39) node [anchor=north east] [inner sep=0.75pt]  [color={rgb, 255:red, 176; green, 109; blue, 0 }  ,opacity=1 ,xscale=1.2,yscale=1.2]  {$\wt{B}$};
\draw (332.77,64.86) node [anchor=south] [inner sep=0.75pt]  [font=\small,color={rgb, 255:red, 176; green, 109; blue, 0 }  ,opacity=1 ,xscale=1.2,yscale=1.2]  {$R_{1}^{-1} R_{2}\wt{B}$};
\draw (161.19,247.09) node [anchor=north east] [inner sep=0.75pt]  [font=\small,color={rgb, 255:red, 176; green, 109; blue, 0 }  ,opacity=1 ,xscale=1.2,yscale=1.2]  {$T^{-1} R_{1}^{-1}\wt{B}$};
\draw (337.28,243.79) node [anchor=north] [inner sep=0.75pt]  [font=\small,color={rgb, 255:red, 176; green, 109; blue, 0 }  ,opacity=1 ,xscale=1.2,yscale=1.2]  {$R_{1}^{-1}\wt{B}$};
\draw (436.57,242.62) node [anchor=north west][inner sep=0.75pt]  [font=\small,color={rgb, 255:red, 176; green, 109; blue, 0 }  ,opacity=1 ,xscale=1.2,yscale=1.2]  {$TR_{1}^{-1}\wt{B}$};
\draw (303.1,156.59) node [anchor=south] [inner sep=0.75pt]  [color={rgb, 255:red, 56; green, 104; blue, 0 }  ,opacity=1 ,xscale=1.2,yscale=1.2]  {$\wt{\alpha }_{0}$};
\draw (286.69,248.44) node [anchor=north] [inner sep=0.75pt]  [font=\small,color={rgb, 255:red, 4; green, 90; blue, 191 }  ,opacity=1 ,xscale=1.2,yscale=1.2]  {$R_{1}^{-1} I_{\wt \F}^{\Z}(\wt y_{0})$};
\draw (261.5,64.92) node [anchor=south] [inner sep=0.75pt]  [font=\small,color={rgb, 255:red, 223; green, 138; blue, 0 }  ,opacity=1 ,xscale=1.2,yscale=1.2]  {$R_{1}^{-1} R_{2}\wt{\phi }_{0}$};
\draw (319.92,207.99) node [anchor=south west] [inner sep=0.75pt]  [font=\small,color={rgb, 255:red, 223; green, 138; blue, 0 }  ,opacity=1 ,xscale=1.2,yscale=1.2]  {$R_{1}^{-1}\wt{\phi }_{0}$};
\draw (158.2,138.49) node [anchor=south] [inner sep=0.75pt]  [color={rgb, 255:red, 223; green, 138; blue, 0 }  ,opacity=1 ,xscale=1.2,yscale=1.2]  {$\wt{\phi }_{0}$};
\draw (397.5,248.39) node [anchor=north] [inner sep=0.75pt]  [font=\small,color={rgb, 255:red, 4; green, 90; blue, 191 }  ,opacity=1 ,xscale=1.2,yscale=1.2]  {$TR_{1}^{-1} I_{\wt \F}^{\Z}(\wt y_{0})$};
\draw (204.5,250.89) node [anchor=north] [inner sep=0.75pt]  [font=\small,color={rgb, 255:red, 4; green, 90; blue, 191 }  ,opacity=1 ,xscale=1.2,yscale=1.2]  {$T^{-1} R_{1}^{-1} I_{\wt \F}^{\Z}(\wt y_{0})$};
\draw (439.77,64.36) node [anchor=south] [inner sep=0.75pt]  [font=\small,color={rgb, 255:red, 176; green, 109; blue, 0 }  ,opacity=1 ,xscale=1.2,yscale=1.2]  {$TR_{1}^{-1} R_{2}\wt{B}$};
\draw (153.77,66.36) node [anchor=south] [inner sep=0.75pt]  [font=\small,color={rgb, 255:red, 176; green, 109; blue, 0 }  ,opacity=1 ,xscale=1.2,yscale=1.2]  {$T^{-1} R_{1}^{-1} R_{2}\wt{B}$};

\end{tikzpicture}

\caption{The objects of Lemma~\ref{LemLeavesDisjointTraj}. The bottom picture is obtained by taking the image of the first one by the deck transformation $R_1^{-1}$. More precisely, the top one is suited for the translates $T^i I^\Z_{\wt \F}(\wt y_0)$ that cross $\wt B$, and the bottom one is suited (up to taking the image by $R_1^{-1}$) for the translates $R_1 T^i R_1^{-1} I^\Z_{\wt \F}(\wt y_0)$ that cross $R_1 \wt B$.}\label{FigLemLeavesDisjointTraj}
\end{center}
\end{figure}
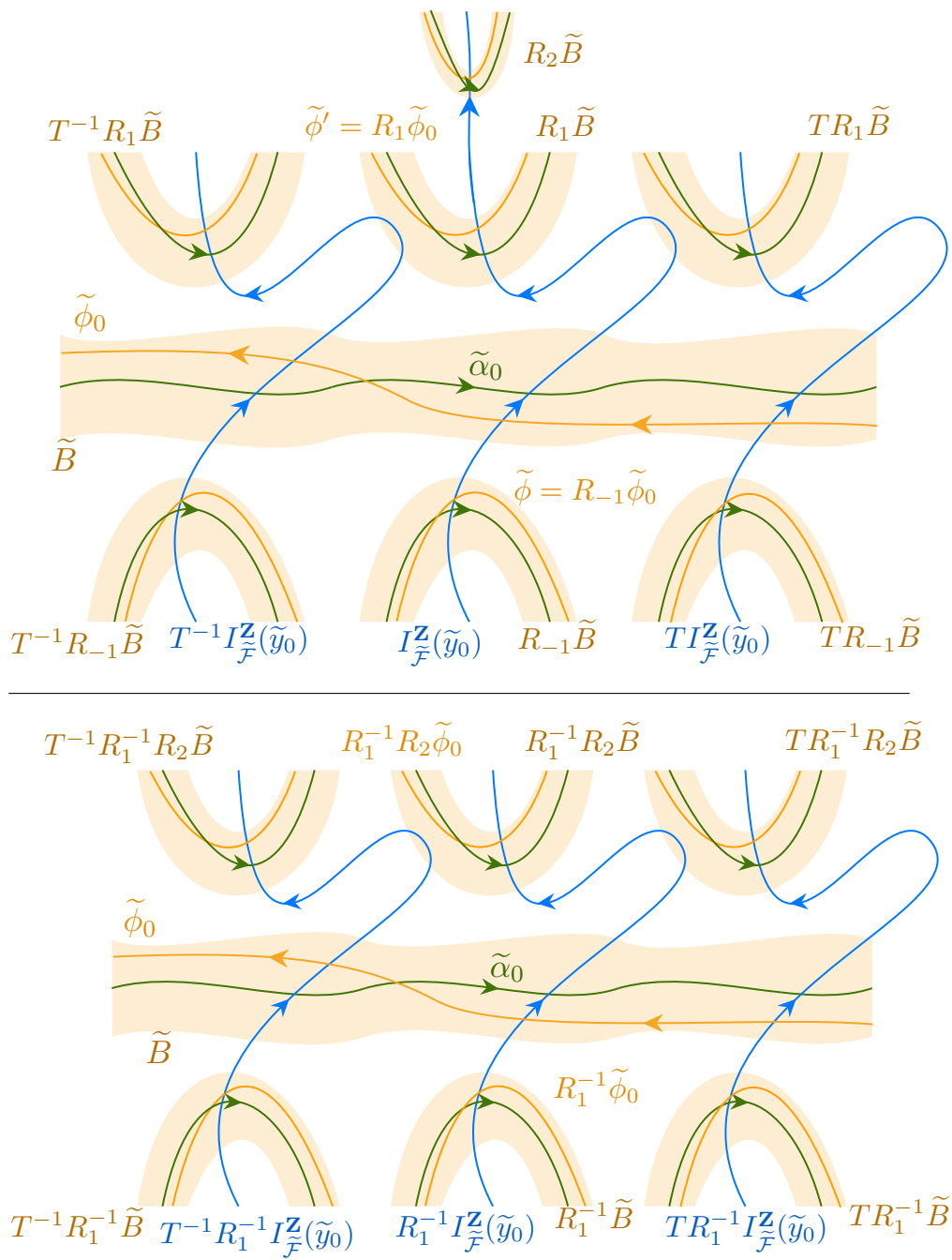

By Proposition~\ref{Prop3.3GLCP}.1, the bands $\wt B$, $R_{-1}\wt B$ and $T R_{-1}\wt B$ are pairwise disjoint. None of them separates the two other ones: this comes from the fact that $\wt\alpha_0$ is disjoint from its images by elements of $\G/\langle T\rangle$, because $\wt\alpha_0$ and $R_1\wt\alpha_0$ have the same orientation, and because of the north-south action of $T$ on $\partial\wt S$.

Recall that (still by Proposition~\ref{Prop3.3GLCP}) a transverse path cannot cross $\wt B$ from left to right.
This implies that the trajectory $I^{[0,n_0]}_{\wt \F}(\wt y_0)$ does not cross $TR_{-1}\wt B$: indeed, this trajectory crosses both $\wt B$ and $R_{-1}\wt B$, and we have seen that none of the bands $\wt B$, $R_{-1}\wt B$ and $T R_{-1}\wt B$ separates the two other ones.
Similarly, the trajectory $I^{[0,n_0]}_{\wt \F}(\wt y_0)$ does not cross $T^{-1}R_{-1}\wt B$, $TR_{1}\wt B$ nor $T^{-1}R_{1}\wt B$.
This allows us to apply Lemma~\ref{LemAccumul2} four times, to the paths $\wt \alpha_2^i = (R_{-1})^{-1} T^iI^{[0,n_0]}_{\wt \F}(\wt y_0)$ for $i=-3,-2,-1,1,2,3$. It implies that for $t$ large enough, and $i=-3,-2,-1,1,2,3$,
\begin{equation}\label{EqInterImagesYEmpty}
R_{-1} \wt\phi_{\wt\alpha_0(t)}\cap T^iI^{[0,n_0]}_{\wt \F}(\wt y_0) = \emptyset.
\end{equation}
Similarly, for $t$ large enough, and $i$ in $\{ -3, -2,-1,1,2, 3\}$,
\begin{align*}
R_{1} \wt\phi_{\wt\alpha_0(t)}\cap T^iI^{[0,n_0]}_{\wt \F}(\wt y_0)  & = \emptyset\\
R_{1}^{-1} \wt\phi_{\wt\alpha_0(t)}\cap T^iR_1^{-1}I^{[0,n_0]}_{\wt \F}(\wt y_0) & = \emptyset\\
R_{1}^{-1}R_2 \wt\phi_{\wt\alpha_0(t)}\cap T^iR_1^{-1}I^{[0,n_0]}_{\wt \F}(\wt y_0) & = \emptyset.
\end{align*}

By Lemma~\ref{LemAccumul}, for $t$ large enough, the leaf $\wt\phi_{\wt\alpha_0(t)}$ meets all the trajectories $R_j^{-1} I^{[0,n_0]}_{\wt \F}(\wt y_0)$ for $j={-1,0,1,2}$, as well as $T I^{[0,n_0]}_{\wt \F}(\wt y_0)$.
Fix $t\gg 1$ such a large enough time and set $\wt\phi_0 = \wt\phi_{\wt\alpha_0(t)}$. 

Note that by Lemma~\ref{LemAccumul}, the trajectory $I^{[0,n_0]}_{\wt \F}(\wt y)$ meets the leaves $R_{-1}\wt\phi_0$, $R_{1}\wt\phi_0$ and $R_2\wt\phi_0$. 
Equation \eqref{EqInterImagesYEmpty} then implies the third point of the lemma.
\end{proof}

Let us write
\[\wt\phi = R_{-1}\wt\phi_0,
\qquad \wt\phi' = R_{1}\wt\phi_0
\qquad \textrm{and}\qquad S_1 = R_2R_1^{-1}.\]

We now repeat the arguments of \cite[Section 3.4]{lellouch} (with the same notations).

Let us consider a lift $\wh \alpha_0$ of $\wt\alpha_0$ to the universal cover $\wh\dom(\F)$ of the domain of the isotopy. Denote $\wh f$ and $\wh \F$ the lift of the homeomorphism $\wt f$ and of the foliation $\wt \F$ to $\wh\dom(\F)$. Denote $\wh B$ the set of leaves of $\wh \F$ crossing $\wh\alpha_0$; it is $T$-invariant (where, by abuse of notation, $T$ is a deck transformation of $\wh\dom(\F)$ that projects down to the deck transformation $T$ of $\wt S$).
Consider the lift $\wh y_0$ of $\wt y_0$ to $\wh\dom(\F)$ such that $I^\Z_{\wh \F}(\wh y_0)$ crosses $\wh B$, it also crosses the bands $R_i \wh B$ for $i=-1, 1, 2$ (where the $R_i$ are here some appropriate lifts of the deck transformations $R_i$ of $\wt S$). Let $\wh \phi$ be the lift of $\wt\phi$ contained in $R_{-1}\wh B$, and $\wh \phi'$ the lift of $\wt\phi'$ contained in $R_{1}\wh B$.

The set\footnote{Recall that $S_1\wh\phi' = R_2\wh\phi_0$.} $\wh O = \wh f^{n_0}\big(L(\wh\phi)\big) \cup R(S_1\wh \phi')$ is a connected open set (see Figure~\ref{FigLemInterO}), hence there exists an oriented simple path $\wh c$, linking $\wh\phi$ to $S_1\wh\phi'$, included in $\wh O$ and whose interior is included in $R(\wh\phi) \cap L(S_1\wh\phi')$. Note that this path is not necessarily transverse and can meet various translates of $\wh B$. It separates the set $R(\wh\phi) \cap L(S_1\wh\phi')$ into two connected components, one on the left of $\wh c$ (relative to its orientation) denoted $l(\wh c)$, and one of the right of $\wh c$ denoted $r(\wh c)$. 

The following is \cite[Lemme 3.4.5]{lellouch}.

\begin{lemma}\label{Lem345Lellouch}
Let $i,n\in\N^*$. If $\wh f^n(\wh O)\cap T^i \wh O  =\emptyset$, then $\wh f^n(\wh O) \subset T^i l(\wh c)$ and $\wh f^{-n}(\wh O) \subset T^{-i} r(\wh c)$.
\end{lemma}

The proof of this lemma is rather direct, we refer to \cite[Lemme 3.4.5]{lellouch} for a complete demonstration.

Note that the sets $(T^i \wh c)_{i\in \Z}$ are compact and separate the band $\wh B$ in fundamental domains for the action of $T$.
By Proposition~\ref{Prop2.2.13Lellouch}, this implies that for any $i\in\Z$ there exists $m_i, m'_i\in \Z$ such that $\wh f^{m_i}(\wh z)\in T^i l(\wh c)$ and $\wh f^{m'_i}(\wh z)\in T^i r(\wh c)$. We can moreover suppose that $m_i<m'_i$ and that $(m_i)$ and $(m'_i)$ are increasing in $i$. Let us prove it implies the following:

\begin{lemma}\label{LemInterO}
There exists $r\in\N$, with $r\ge n_0$, such that $\wh f^r(\wh O) \cap T^3 \wh O \neq\emptyset$.
\end{lemma}

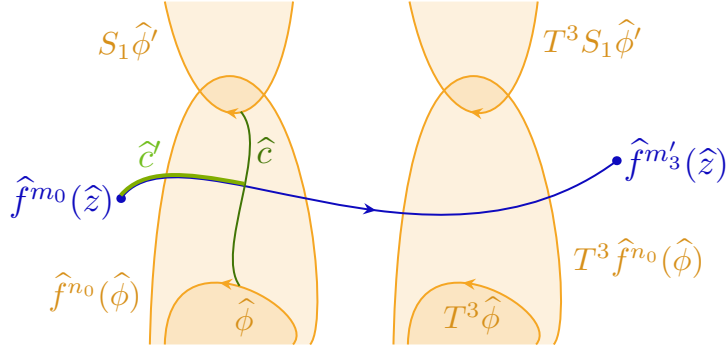
\begin{figure}
\begin{center}
\tikzset{every picture/.style={line width=0.75pt}} 

\begin{tikzpicture}[x=0.75pt,y=0.75pt,yscale=-.75,xscale=.8]
\draw [color={rgb, 255:red, 245; green, 166; blue, 35 }  ,draw opacity=1 ][fill={rgb, 255:red, 245; green, 166; blue, 35 }  ,fill opacity=0.15 ]   (301.67,35) .. controls (320.88,133.35) and (370.48,136.15) .. (381.67,35) ;
\draw [shift={(339.31,108.8)}, rotate = 9.9] [fill={rgb, 255:red, 245; green, 166; blue, 35 }  ,fill opacity=1 ][line width=0.08]  [draw opacity=0] (8.04,-3.86) -- (0,0) -- (8.04,3.86) -- (5.34,0) -- cycle    ;
\draw [color={rgb, 255:red, 245; green, 166; blue, 35 }  ,draw opacity=1 ][fill={rgb, 255:red, 245; green, 166; blue, 35 }  ,fill opacity=0.15 ]   (149,35) .. controls (168.21,133.35) and (217.81,136.15) .. (229,35) ;
\draw [shift={(186.64,108.8)}, rotate = 9.9] [fill={rgb, 255:red, 245; green, 166; blue, 35 }  ,fill opacity=1 ][line width=0.08]  [draw opacity=0] (8.04,-3.86) -- (0,0) -- (8.04,3.86) -- (5.34,0) -- cycle    ;
\draw [color={rgb, 255:red, 245; green, 166; blue, 35 }  ,draw opacity=1 ][fill={rgb, 255:red, 245; green, 166; blue, 35 }  ,fill opacity=0.15 ]   (291.67,265) .. controls (292.15,244.51) and (294.32,84.65) .. (341.67,85) .. controls (389.01,85.35) and (405.07,248.7) .. (391.67,265) ;
\draw [color={rgb, 255:red, 245; green, 166; blue, 35 }  ,draw opacity=1 ][fill={rgb, 255:red, 245; green, 166; blue, 35 }  ,fill opacity=0.15 ]   (139.67,265) .. controls (140.15,244.51) and (142.32,84.65) .. (189.67,85) .. controls (237.01,85.35) and (253.07,248.7) .. (239.67,265) ;
\draw [color={rgb, 255:red, 245; green, 166; blue, 35 }  ,draw opacity=1 ][fill={rgb, 255:red, 245; green, 166; blue, 35 }  ,fill opacity=0.15 ]   (149,265) .. controls (150.92,183.71) and (246.92,243.21) .. (229,265) ;
\draw [shift={(183.15,223.61)}, rotate = 7.45] [fill={rgb, 255:red, 245; green, 166; blue, 35 }  ,fill opacity=1 ][line width=0.08]  [draw opacity=0] (8.04,-3.86) -- (0,0) -- (8.04,3.86) -- (5.34,0) -- cycle    ;
\draw [color={rgb, 255:red, 14; green, 10; blue, 190 }  ,draw opacity=1 ]   (431.08,141.74) .. controls (315.74,234.41) and (157.74,113.74) .. (121.74,167.08) ;
\draw [shift={(121.74,167.08)}, rotate = 124.02] [color={rgb, 255:red, 14; green, 10; blue, 190 }  ,draw opacity=1 ][fill={rgb, 255:red, 14; green, 10; blue, 190 }  ,fill opacity=1 ][line width=0.75]      (0, 0) circle [x radius= 2.34, y radius= 2.34]   ;
\draw [shift={(280.89,175.09)}, rotate = 188.09] [fill={rgb, 255:red, 14; green, 10; blue, 190 }  ,fill opacity=1 ][line width=0.08]  [draw opacity=0] (8.04,-3.86) -- (0,0) -- (8.04,3.86) -- (5.34,0) -- cycle    ;
\draw [shift={(431.08,141.74)}, rotate = 141.22] [color={rgb, 255:red, 14; green, 10; blue, 190 }  ,draw opacity=1 ][fill={rgb, 255:red, 14; green, 10; blue, 190 }  ,fill opacity=1 ][line width=0.75]      (0, 0) circle [x radius= 2.34, y radius= 2.34]   ;
\draw [color={rgb, 255:red, 245; green, 166; blue, 35 }  ,draw opacity=1 ][fill={rgb, 255:red, 245; green, 166; blue, 35 }  ,fill opacity=0.15 ]   (300.67,265) .. controls (302.58,183.71) and (398.58,243.21) .. (380.67,265) ;
\draw [shift={(334.81,223.61)}, rotate = 7.45] [fill={rgb, 255:red, 245; green, 166; blue, 35 }  ,fill opacity=1 ][line width=0.08]  [draw opacity=0] (8.04,-3.86) -- (0,0) -- (8.04,3.86) -- (5.34,0) -- cycle    ;
\draw [color={rgb, 255:red, 65; green, 117; blue, 5 }  ,draw opacity=1 ]   (196.42,108.83) .. controls (215.42,132.83) and (179.69,201.94) .. (195.69,225.69) ;
\draw [color={rgb, 255:red, 129; green, 173; blue, 0 }  ,draw opacity=1 ][line width=1.5]    (121.47,164.99) .. controls (135.47,147.39) and (159.77,149.1) .. (198.67,157.26) ;

\draw (148,46.73) node [anchor=north east] [inner sep=0.75pt]  [font=\small,color={rgb, 255:red, 215; green, 144; blue, 26 }  ,opacity=1 ,xscale=1.2,yscale=1.2]  {$S_{1}\wh{\phi }'$};
\draw (136.7,230.46) node [anchor=east] [inner sep=0.75pt]  [font=\small,color={rgb, 255:red, 215; green, 144; blue, 26 }  ,opacity=1 ,xscale=1.2,yscale=1.2]  {$\wh{f}^{n_{0}}(\wh{\phi })$};
\draw (401.75,207.13) node [anchor=west] [inner sep=0.75pt]  [font=\small,color={rgb, 255:red, 215; green, 144; blue, 26 }  ,opacity=1 ,xscale=1.2,yscale=1.2]  {$T^{3}\wh{f}^{n_{0}}(\wh{\phi })$};
\draw (198.48,231.84) node [anchor=north] [inner sep=0.75pt]  [font=\small,color={rgb, 255:red, 215; green, 144; blue, 26 }  ,opacity=1 ,xscale=1.2,yscale=1.2]  {$\wh{\phi }$};
\draw (340.48,230.73) node [anchor=north] [inner sep=0.75pt]  [font=\small,color={rgb, 255:red, 215; green, 144; blue, 26 }  ,opacity=1 ,xscale=1.2,yscale=1.2]  {$T^{3}\wh{\phi }$};
\draw (382.78,60.23) node [anchor=west] [inner sep=0.75pt]  [font=\small,color={rgb, 255:red, 215; green, 144; blue, 26 }  ,opacity=1 ,xscale=1.2,yscale=1.2]  {$T^{3} S_{1}\wh{\phi }'$};
\draw (119.74,167.08) node [anchor=east] [inner sep=0.75pt]  [color={rgb, 255:red, 14; green, 10; blue, 190 }  ,opacity=1 ,xscale=1.2,yscale=1.2]  {$\wh{f}^{m_{0}}(\wh{z})$};
\draw (433.08,141.74) node [anchor=west] [inner sep=0.75pt]  [color={rgb, 255:red, 14; green, 10; blue, 190 }  ,opacity=1 ,xscale=1.2,yscale=1.2]  {$\wh{f}^{m'_{3}}(\wh{z})$};
\draw (204.04,135.05) node [anchor=west] [inner sep=0.75pt]  [color={rgb, 255:red, 65; green, 117; blue, 5 }  ,opacity=1 ,xscale=1.2,yscale=1.2]  {$\wh{c}$};
\draw (149.63,148.55) node [anchor=south east] [inner sep=0.75pt]  [color={rgb, 255:red, 115; green, 191; blue, 32 }  ,opacity=1 ,xscale=1.2,yscale=1.2]  {$\wh{c} '$};

\end{tikzpicture}

\caption{Proof of Lemma~\ref{LemInterO}}\label{FigLemInterO}
\end{center}
\end{figure}

\begin{proof}
The proof is illustrated in Figure~\ref{FigLemInterO}.
By the above fact, we have that $\wh f^{m_{0}}(\wh z)\in l(\wh c)$ and $\wh f^{m'_{3}}(\wh z)\in T^3r(\wh c)$. Set $r = m'_3-m_0$ (because $m'_3$ can be chosen arbitrarily large, one can suppose that $r\ge n_0$ the length of the orbit of $y_0$) and suppose that the conclusion of the lemma is false for this $r$. By Lemma~\ref{Lem345Lellouch}, this implies that 
\begin{equation}\label{EqInclusionslr}
\wh f^r(\wh O)\subset T^3 l(\wh c)
\qquad \textrm{and}\qquad 
\wh f^{-r}(T^3\wh O)\subset  r(\wh c).
\end{equation}
Let $\wh c'$ be a simple path linking $\wh f^{m_{0}}(\wh z)$ to $\wh c$ and included in $l(\wh c)$. By \eqref{EqInclusionslr}, it is disjoint from $\wh f^{-r}(T^3\wh O)$. So $\wh f^r(\wh c')$ is disjoint from $T^3\wh O$, and contains $\wh f^{m'_{3}}(\wh z)\in T^3r(\wh c)$. Hence, it is included in $T^3 r(\wh c)$, and in particular one of its extremities belongs to $T^3 r(\wh c) \cap \wh f^r(\wh c) \subset T^3 r(\wh c) \cap \wh f^r(\wh O)$. This contradicts \eqref{EqInclusionslr}. 
\end{proof}

\begin{lemma}\label{LemExistPathGamma}
There exists a transverse path $\wh\beta : [0,1]\to \wh\dom(\F)$, admissible of order $r+n_0$, and such that $\wh\beta(0) \in \wh\phi$ and $\wh\beta(1)\in T^3S_1\wh \phi'$. 
\end{lemma}

Note that $\wh\beta$ meets the leaves $\wh \phi = R_{-1}\wh\phi_0$, $\wh\phi_0$, $T^3\wh\phi' = T^3 R_1\wh\phi_0$ and $T^3S_1\wh\phi' = T^3 R_2\wh\phi_0$ in this order.

\begin{proof}
The conclusion of Lemma~\ref{LemInterO} is equivalent to:
\begin{equation*}
 \wh f^r\Big(\wh f^{n_0} (L\wh \phi) \cup S_1 R\wh \phi'\Big) \cap T^3 \Big(\wh f^{n_0} (L\wh \phi) \cup S_1 R\wh \phi'\Big) \neq\emptyset,
\end{equation*}
in other words
\begin{multline}\label{EqFourPossibInter}
\Big(\wh f^{r+n_0} \big(L\wh \phi\big) \cap T^3 \wh f^{n_0} \big(L\wh \phi)\Big)
\cup
\Big(\wh f^{r+n_0} \big(L\wh \phi\big) \cap T^3 S_1 R\wh \phi'\Big)
\cup \\
\Big(f^r\big(S_1 R\wh \phi'\big) \cap T^3 \wh f^{n_0} \big(L\wh \phi\big)\Big)
\cup
\Big(f^r\big(S_1 R\wh \phi'\big) \cap T^3 S_1 R\wh \phi'\Big) \neq\emptyset.
\end{multline}
We will prove that all these intersections are empty but the second one.

By Lemma~\ref{LemLeavesDisjointTraj}, we have 
\[L(T^3\wh\phi) \subset R(\wh\phi)
\qquad\textrm{and}\qquad
L(\wh\phi)\subset R(T^3\wh\phi),\]
as well as 
\[R(T^3S_1\wh\phi') \subset L(S_1\wh\phi')
\qquad\textrm{and}\qquad
R(S_1\wh\phi')\subset L(T^3S_1\wh\phi')\]
(for the first inclusion, note that $R(T^3S_1\wh\phi') \subset R(T^3\wh\phi')\subset L(\wh\phi')\subset L(S_1\wh\phi')$, the second is obtained in a similar way) .
Combined with the fact that $\wh\phi$ and $\wh\phi'$ are Brouwer lines, this implies that
\[\wh f^{r+n_0} (L\wh \phi) \cap T^3 \wh f^{n_0} (L\wh \phi) = 
\wh f^r\big( S_1 R\wh \phi'\big) \cap T^3S_1 R\wh \phi' =
\emptyset.\]
Moreover, using the fact that $r\ge n_0$ and that $\wh\phi'$ is a Brouwer line,
\begin{align*}
\wh f^r\big( S_1 R\wh \phi'\big) \cap T^3 \wh f^{n_0} (L\wh \phi) 
& = \wh f^{n_0}\left( S_1 \wh f^{r-n_0}\big( R\wh \phi'\big) \cap T^3 L\wh \phi\right)\\
& \subset \wh f^{n_0}\left( S_1  R\wh \phi' \cap T^3 L\wh \phi\right)\\
& \subset \wh f^{n_0}\left(  R\wh \phi' \cap T^3 L\wh \phi\right)\\
& \subset \wh f^{n_0}\left( L\wh\alpha_0 \cap R\wh\alpha_0\right) = \emptyset
\end{align*}
by Lemma~\ref{LemLeavesDisjointTraj}. 

Therefore, Equation~\eqref{EqFourPossibInter} implies that 
\[\wh f^{r+n_0} (R\wh \phi) \cap T^3 S_1 L\wh \phi'\neq\emptyset.\]
This proves the lemma.
\end{proof}

Let $\wh\beta$ be the path given by Lemma~\ref{LemExistPathGamma}.

\begin{lemma}\label{ExistTransInterAccCase}
Either there exists $k\in\Z$ such that the paths $I^{[0,n_0]}_{\wh \F}(\wh y_0)$ and $T^k R_1^{-1} I^{[0,n_0]}_{\wh \F}(\wh y_0)$ intersect $\wh \F$-transversally, 
or there exists $k\in\Z$ such that the paths $\wh\beta$ and $T^k R_1^{-1} \wh\beta$ intersect $\wh \F$-transversally. 
\end{lemma}

\begin{figure}
\begin{center}

\tikzset{every picture/.style={line width=0.75pt}} 

\begin{tikzpicture}[x=0.75pt,y=0.75pt,yscale=-1,xscale=1]

\draw [color={rgb, 255:red, 65; green, 117; blue, 5 }  ,draw opacity=1 ]   (190,110) .. controls (230,80) and (250,140) .. (290,110) ;
\draw [shift={(243.15,111.57)}, rotate = 206.47] [fill={rgb, 255:red, 65; green, 117; blue, 5 }  ,fill opacity=1 ][line width=0.08]  [draw opacity=0] (8.04,-3.86) -- (0,0) -- (8.04,3.86) -- (5.34,0) -- cycle    ;
\draw [color={rgb, 255:red, 65; green, 117; blue, 5 }  ,draw opacity=1 ]   (390,110) .. controls (430,80) and (450,140) .. (490,110) ;
\draw [shift={(443.15,111.57)}, rotate = 206.47] [fill={rgb, 255:red, 65; green, 117; blue, 5 }  ,fill opacity=1 ][line width=0.08]  [draw opacity=0] (8.04,-3.86) -- (0,0) -- (8.04,3.86) -- (5.34,0) -- cycle    ;
\draw [color={rgb, 255:red, 65; green, 117; blue, 5 }  ,draw opacity=1 ]   (290,110) .. controls (330,80) and (350,140) .. (390,110) ;
\draw [shift={(343.15,111.57)}, rotate = 206.47] [fill={rgb, 255:red, 65; green, 117; blue, 5 }  ,fill opacity=1 ][line width=0.08]  [draw opacity=0] (8.04,-3.86) -- (0,0) -- (8.04,3.86) -- (5.34,0) -- cycle    ;
\draw [color={rgb, 255:red, 245; green, 166; blue, 35 }  ,draw opacity=1 ]   (200,180) .. controls (234.21,144.21) and (235.17,153.71) .. (249.46,182) ;
\draw [shift={(223.89,158.83)}, rotate = 340.16] [fill={rgb, 255:red, 245; green, 166; blue, 35 }  ,fill opacity=1 ][line width=0.08]  [draw opacity=0] (7.14,-3.43) -- (0,0) -- (7.14,3.43) -- (4.74,0) -- cycle    ;
\draw [color={rgb, 255:red, 245; green, 166; blue, 35 }  ,draw opacity=1 ]   (300,180) .. controls (334.21,144.21) and (335.17,153.71) .. (349.46,182) ;
\draw [shift={(323.89,158.83)}, rotate = 340.16] [fill={rgb, 255:red, 245; green, 166; blue, 35 }  ,fill opacity=1 ][line width=0.08]  [draw opacity=0] (7.14,-3.43) -- (0,0) -- (7.14,3.43) -- (4.74,0) -- cycle    ;
\draw [color={rgb, 255:red, 245; green, 166; blue, 35 }  ,draw opacity=1 ]   (400,180) .. controls (434.21,144.21) and (435.17,153.71) .. (449.46,182) ;
\draw [shift={(423.89,158.83)}, rotate = 340.16] [fill={rgb, 255:red, 245; green, 166; blue, 35 }  ,fill opacity=1 ][line width=0.08]  [draw opacity=0] (7.14,-3.43) -- (0,0) -- (7.14,3.43) -- (4.74,0) -- cycle    ;
\draw [color={rgb, 255:red, 227; green, 80; blue, 227 }  ,draw opacity=1 ]   (211.54,101.5) .. controls (211.38,130.83) and (240.13,140.33) .. (233.54,157.33) ;
\draw [color={rgb, 255:red, 227; green, 80; blue, 227 }  ,draw opacity=1 ]   (312.04,101.44) .. controls (311.88,130.77) and (340.63,140.27) .. (334.04,157.27) ;
\draw [color={rgb, 255:red, 227; green, 80; blue, 227 }  ,draw opacity=1 ]   (412.52,101.69) .. controls (412.35,131.02) and (441.1,140.52) .. (434.52,157.52) ;
\draw  [draw opacity=0][fill={rgb, 255:red, 144; green, 19; blue, 254 }  ,fill opacity=0.1 ] (211.54,101.5) .. controls (232.38,101.54) and (248.78,118.77) .. (265.73,119.08) .. controls (283.09,119.4) and (296.98,100.83) .. (312.04,101.44) .. controls (313.23,132.33) and (339.48,140.08) .. (334.04,157.27) .. controls (321.73,153.83) and (309.48,171.83) .. (300,180) .. controls (280.48,180.08) and (268.23,180.33) .. (248.48,181.08) .. controls (246.23,174.58) and (240.23,161.33) .. (233.54,157.33) .. controls (239.48,139.08) and (211.98,131.58) .. (211.54,101.5) -- cycle ;

\draw (274.7,146.06) node  [color={rgb, 255:red, 97; green, 7; blue, 201 }  ,opacity=1 ,xscale=1.2,yscale=1.2]  {$\wh{A}_{0}$};
\draw (327.58,128.2) node [anchor=west] [inner sep=0.75pt]  [color={rgb, 255:red, 165; green, 13; blue, 196 }  ,opacity=1 ,xscale=1.2,yscale=1.2]  {$\wh{\sigma }$};
\draw (212.48,122.95) node [anchor=east] [inner sep=0.75pt]  [color={rgb, 255:red, 165; green, 13; blue, 196 }  ,opacity=1 ,xscale=1.2,yscale=1.2]  {$T^{-1}\wh{\sigma }$};
\draw (429.83,139) node [anchor=south west] [inner sep=0.75pt]  [color={rgb, 255:red, 165; green, 13; blue, 196 }  ,opacity=1 ,xscale=1.2,yscale=1.2]  {$T\wh{\sigma }$};
\draw (377.12,118.65) node [anchor=north] [inner sep=0.75pt]  [color={rgb, 255:red, 65; green, 117; blue, 5 }  ,opacity=1 ,xscale=1.2,yscale=1.2]  {$\wh{\alpha }_{0}$};
\draw (347.08,170.63) node [anchor=south west] [inner sep=0.75pt]  [color={rgb, 255:red, 213; green, 144; blue, 31 }  ,opacity=1 ,xscale=1.2,yscale=1.2]  {$\wh{\phi }$};
\draw (444.83,168.87) node [anchor=south west] [inner sep=0.75pt]  [color={rgb, 255:red, 213; green, 144; blue, 31 }  ,opacity=1 ,xscale=1.2,yscale=1.2]  {$T\wh{\phi }$};
\draw (216.65,162.87) node [anchor=south east] [inner sep=0.75pt]  [color={rgb, 255:red, 213; green, 144; blue, 31 }  ,opacity=1 ,xscale=1.2,yscale=1.2]  {$T^{-1}\wh{\phi }$};

\end{tikzpicture}

\caption{Configuration of the proof of Lemma~\ref{ExistTransInterAccCase}.}\label{FigExistTransInterAccCase}
\end{center}
\end{figure}
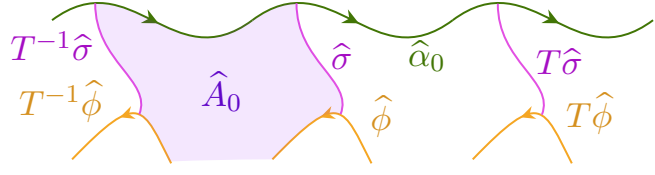

\begin{proof}
Recall that $\wh\phi = R_{-1}\wh\phi_0$ and $\wh\phi' = R_{1}\wh\phi_0$.  
Let $\wh\sigma$ be a transverse path included in $R(\wh\alpha_0)\cap R(\wh\phi)$ linking $\wh\phi$ to $\wh\alpha_0$ that is disjoint from its translates by iterates of $T$ (see Figure~\ref{FigExistTransInterAccCase}). Denote $\wh A_0$ the connected component of 
\[\left(R(\wh\alpha_0) \cap \bigcap_{\ell\in \Z} R(T^\ell\wh\phi)\right)
\setminus \bigcup_{\ell\in\Z} T^\ell\wh\sigma\]
whose closure intersects both $T^{-1}\wh\phi$ and $\wh\phi$.

As $R_1^{-1}\wh\phi_0\subset R(\wh B)$, there exists $k\in\Z$ such that:
\begin{itemize}
\item either $R_1^{-1}\wh\phi_0 \cap T^k\wh\sigma\neq\emptyset$;
\item or $R_1^{-1}\wh\phi_0 \subset L(T^k\wh\phi)$;
\item or $R_1^{-1}\wh\phi_0$ is included in $T^k \wh A_0$.
\end{itemize} 
In the last case, this implies that $R_1^{-1}\wh\phi_0$ is between $T^k\wh\phi$  and $T^{k+1}\wh\phi$ relative to $\wh\phi_0$ (see Definition~\ref{DefInterTrans}).

In the first case, let us show that such a $k$ is unique. Suppose by contradiction that there exists also $k'\in\Z$, $k'\neq k$, such that $R_1^{-1}\wh\phi_0 \cap T^{k'}\wh\sigma\neq\emptyset$. 
This implies that either $R_1^{-1}\wh\phi_0 \cap T^{k+1}\wh\sigma\neq\emptyset$, or $R_1^{-1}\wh\phi_0 \cap T^{k-1}\wh\sigma\neq\emptyset$. Suppose we are in the first case, the second being identical.
Then $T^k\wh\sigma$ meets both $R_1^{-1}\wh\phi_0$ and $T^{-1}R_1^{-1}\wh\phi_0$. Hence (because $\wh\sigma$ is transverse), this means that either $L(R_1^{-1}\wh\phi_0)\subset L(T^{-1}R_1^{-1}\wh\phi_0)$, or $R(R_1^{-1}\wh\phi_0)\subset R(T^{-1}R_1^{-1}\wh\phi_0)$. If the first inclusion held, as $T^{-1}R_1^{-1}\wh\phi_0\subset L(\wt\alpha_0)$ and $R_1^{-1}I^{[0,n_0]}_{\wh\F}(\wh y_0)$ crosses both $\wh\alpha_0$ and $R_1^{-1}\wh\phi_0$, this would imply that $R_1^{-1} I^{[0,n_0]}_{\wt\F}(\wt y_0)$ crosses $T^{-1}R_1^{-1}\wh\phi_0$, contradicting Lemma~\ref{LemLeavesDisjointTraj}. The second inclusion is also impossible for the same reasons. 

This allows us to formulate an equivalent of the above trichotomy: there exists $k\in\Z$ such that (Figure~\ref{FigIdeaPropBndedDevRatCase1} depicts the third case of this trichotomy):
\begin{itemize}
\item either $L(T^k\wh\phi) \subset L(R_1^{-1}\wh\phi_0)$, and, for any $i\neq k$, we have $L(T^i\wh\phi) \cap L(R_1^{-1}\wh\phi_0) = \emptyset$;
\item or $L(R_1^{-1}\wh\phi_0) \subset L(T^k\wh\phi)$ (recall that there exists a transverse path linking $R_1^{-1}\wh\phi_0$ to $\wh\alpha_0$: the path $R_1^{-1}I^{[0,n_0]}_{\wh\F}(\wh y_0)$);
\item or $R_1^{-1}\wh\phi_0$ is between $T^k\wh\phi$  and $T^{k+1}\wh\phi$ relative to $\wh\phi_0$.
\end{itemize}
Similarly, there exists $k'\in\Z$ such that:
\begin{itemize}
\item either $R(T^{k'}\wh\phi') \subset R(R_1^{-1}R_2\wh\phi_0)$, and, for any $i\neq k'$, we have $R(T^i\wh\phi') \cap R(R_1^{-1}R_2\wh\phi_0) = \emptyset$;
\item or $R(R_1^{-1}R_2\wh\phi_0) \subset R(T^{k'}\wh\phi')$;
\item or $R_1^{-1}R_2\wh\phi_0$ is between $T^{k'}\wh\phi'$  and $T^{k'+1}\wh\phi'$ relative to $\wh\phi_0$.
\end{itemize} 

Suppose $|k-k'|\ge 2$, for example $k'\ge k+2$ (see Figure~\ref{FigIdeaPropBndedDevRatCase1} page~\pageref{FigIdeaPropBndedDevRatCase1}, left). Then $R_1^{-1}\wh\phi_0$ is below $T^{k+1}\wh\phi$ relative to $\wh\phi_0$, but $R_1^{-1}R_2\wh\phi_0$ is above $T^{k'-1}\wh\phi'$ which itself is either equal to $T^{k+1}\wh\phi'$, or above $T^{k+1}\wh\phi'$ relative to $\wh\phi_0$. This implies that $I^{[0,n_0]}_{\wh \F}(\wh y_0)$ and $T^{k+1} R_1^{-1} I^{[0,n_0]}_{\wh \F}(\wh y_0)$ intersect $\wh \F$-transversally. 
Similarly, if $k'\le k-2$, then $I^{[0,n_0]}_{\wh \F}(\wh y_0)$ and $T^{k-1} R_1^{-1} I^{[0,n_0]}_{\wh \F}(\wh y_0)$ intersect $\wh \F$-transversally. 

Hence, we can suppose that $|k-k'|\le 1$. Then $TR_1^{-1}\wh\phi_0$ is above $T^{k}\wh\phi$ relative to $\wh\phi_0$, but $TR_1^{-1}R_2\wh\phi_0$ is below $T^{k'+2}\wh\phi'$ which itself is on the left of, or equal to  $T^{k+3}\wh\phi'$ relative to $\wh\phi_0$. 

This implies that $\wh\beta$ and $T^{k-1}R_1^{-1} \wh\beta$ intersect $\wh \F$-transversally.
\end{proof}


\begin{proof}[Proof of Proposition~\ref{PropBndedDevRatCase1}]
Note that as the bands $\wt B$ and $R_1\wt B$ have the same orientation (because the transverse trajectory of $\wt y$ crosses both of them, and because of Proposition~\ref{Prop3.3GLCP}.3), the axis of the deck transformation $R_1$ has to cross the one of $T$. This implies that the axis of $T^k R_1^{-1}$ has to cross the one of $T$.

By Lemma~\ref{ExistTransInterAccCase}, there exist $k\in\Z$ and an $\F$-transverse trajectory that intersects transversally the image of itself by the deck transformation $T^k R_1^{-1}$. By Theorem~\ref{thmMlct2}, this implies that there is a periodic orbit whose tracking geodesic is $T^k R_1^{-1} = R_0 T^k R_1^{-1}$-invariant. As already noted in the end of the proof of Proposition~\ref{PropBndedDevRatCase1a}, it forces this geodesic to cross both $R_0\wt\gamma_{\wt z}$ and $R_1\wt\gamma_{\wt z}$; this proves the proposition. 
\end{proof}

\subsection{Final proof of Theorem~\ref{ThBndedDevRat} and Corollary~\ref{CoroBndedDevRat}}

\begin{proof}[Proof of Theorem~\ref{ThBndedDevRat}]
By Lemma~\ref{LemUseResidFinite} one can suppose that the lifts of $\gamma$ crossed by the orbit segment $\wt y,\dots,\wt f^{n_0}(\wt y)$ have their $D+ m_1 d(\wt f, \Id_{\wt S})$-neighbourhood that are pairwise disjoint and have the same orientation

Suppose that there exist 40 different copies of $\wt B$, denoted by $(R_i \wt B)_{1\le i\le 40}$, such that the following is true. 
First, we suppose that the sets $R_i V_{D+ m_1 d(\wt f, \Id_{\wt S})}(\wt\gamma)$ are pairwise disjoint and well ordered: for $i,j,k$ pairwise different, one of $R_i\wt\gamma$, $R_j\wt\gamma$ and $R_k\wt\gamma$ separates the two other ones. 
Second, we suppose that for all $i$, either $\wt y_0\in L\big(R_i V_{D+ m_1 d(\wt f, \Id_{\wt S})}(\wt\gamma)\big)$ and $\wt f^{n_0}(\wt y_0)\in R\big(R_i V_{D+ m_1 d(\wt f, \Id_{\wt S})}(\wt\gamma)\big)$, or $\wt y_0\in R\big(R_i V_{D+ m_1 d(\wt f, \Id_{\wt S})}(\wt\gamma)\big)$ and $\wt f^{n_0}(\wt y_0)\in L\big(R_i V_{D+ m_1 d(\wt f, \Id_{\wt S})}(\wt\gamma)\big)$.

Without loss of generality, we can suppose that for all $i$, we have $\wt y_0\in L\big(R_i V_{D+ m_1 d(\wt f, \Id_{\wt S})}(\wt\gamma)\big)$ and $\wt f^{n_0}(\wt y_0)\in R\big(R_i V_{D+ m_1 d(\wt f, \Id_{\wt S})}(\wt\gamma)\big)$. We also suppose that the $R_i\wt\gamma$ are ordered: for any $i<j$ we have $R_i\wt\gamma\subset L(R_j\wt\gamma)$. 

Let us group the $R_i\wt B$ in 4 groups of 10: $1\le i\le 10$, $11\le i\le 20$, $21\le i \le 30$ and $31\le i\le 40$. Let us study the first group.

\textbf{First case}: For some $1\le i \le 5$, the trajectory $I^{[0,n_0]}_{\wt\F}(\wt y_0)$ intersects $L(R_i\wt B)$. Then (because the sets $R_i\wt B$ are well ordered: thanks to Proposition~\ref{LemRealizPeriod}, the path $\wt\alpha_0$ is simple) for any $6\le j\le 10$ the trajectory $I^{[0,n_0]}_{\wt\F}(\wt y_0)$ intersects $L(R_j\wt B)$.

\textit{First subcase}: Either for some $6\le j\le 10$, the trajectory $I^{[0,n_0]}_{\wt\F}(\wt y_0)$ intersects $R(R_j\wt B)$. In this case, the trajectory $I^{[0,n_0]}_{\wt\F}(\wt y_0)$ crosses the band $R_j\wt B$.

\textit{Second subcase}: Or for any $6\le j\le 10$, the trajectory $I^{[0,n_0]}_{\wt\F}(\wt y_0)$ stays in $L(R_j\wt B) \cup R_j\wt B$. In this case, it is possible to apply 
Proposition~\ref{PropBndedDevRatCase2}, which proves the theorem.

\textbf{Second case}: For any $1\le i \le 5$, the trajectory $I^{[0,n_0]}_{\wt\F}(\wt y_0)$ does not intersect $L(R_i\wt B)$. In this case, it is possible to apply 
Proposition~\ref{PropBndedDevRatCase2}, which proves the theorem.
\bigskip

We are reduced to the case where for each of the 4 groups of $R_i\wt B$, there exists some $R_j$ such that the trajectory $I^{[0,n_0]}_{\wt\F}(\wt y_0)$ crosses the band $R_j\wt B$: more precisely the trajectory $I^{[0,n_0]}_{\wt\F}(\wt y_0)$ crosses the bands $(R_{j_i}\wt B)_{1\le i\le 4}$ for some $j_1<j_2<j_3<j_4$. 
If for any $1\le i\le 4$ the trajectory $R_{i}\wt\alpha_0$ does not accumulate in $I^{[0,n_0]}_{\wt\F}(\wt y_0)$, then one can apply 
Proposition~\ref{PropBndedDevRatCase1a}, which proves the theorem.
If for some $1\le i_0\le 4$ the trajectory $R_{i_0}\wt\alpha_0$ accumulates in $I^{[0,n_0]}_{\wt\F}(\wt y_0)$, then by Lemma~\ref{LemAccumul}, for any $1\le i\le 4$ the trajectory $R_{i}\wt\alpha_0$ accumulates in $I^{[0,n_0]}_{\wt\F}(\wt y_0)$. This allows us to apply Proposition~\ref{PropBndedDevRatCase1}, which proves the theorem.
\end{proof}

\begin{coro}\label{CoroBndedDevRat2}
Let $S$ be a compact boundaryless hyperbolic surface and $f\in \Homeo_0(S)$. Let $\gamma$ be a closed geodesic that is:
\begin{itemize}
\item either a tracking geodesic for some ergodic $f$-invariant probability measure that does not belong to a chaotic class;
\item or the boundary component of the surface associated to a chaotic class.
\end{itemize}
Let $\wt f$ be the canonical lift of $f$ to the universal cover $\wt S$ of $S$. Then there exists $N>0$ such that an orbit of $\wt f$ cannot cross more than $N$ different lifts of $\gamma$.
\end{coro}

\begin{proof}[Proof of Corollary~\ref{CoroBndedDevRat2}]
The first point of the corollary is a direct consequence of Theorem~\ref{ThBndedDevRat}: consider the constant $N$ given by Theorem~\ref{ThBndedDevRat}, associated to $\gamma$ the closed tracking geodesic of some ergodic $f$-invariant probability measure $\mu$, which does not belong to a chaotic class.
Suppose that there exists an orbit of $\wt f$ crossing at least $N$ different lifts of $\gamma$. Then by Theorem~\ref{ThBndedDevRat} there exists a periodic point whose tracking geodesic crosses $\gamma$. This is a contradiction with the fact that $\mu$ does not belong to a chaotic class.
\bigskip

Let us prove the second point of the corollary: $\gamma$ is the boundary component of the surface $S_i$ assocoated to a chaotic class. Let $\Lambda_i = \bigcup_{\mu\in\cl_i} \Lambda_{\mu}$ be the set of geodesics associated to the class $\cl_i$. Recall that $\wt S_i$ is the convex hull of a connected component $\wt\Lambda_i^0$ of the lift $\wt\Lambda_i$ to $\wt S$, and that $S_i$ is the projection of $\wt S_i$ to $S$.

We first prove the existence of a finite number of tracking geodesics $\wt\gamma_1,\dots,\wt\gamma_k$ of periodic points belonging to the class, cutting the surface $S_i$ into topological discs and possibly topological open annuli having a boundary component of $S_i$ as a boundary component. 
Let us consider the union $\Lambda_i^\mathrm{p}$ of all tracking geodesics of periodic points belonging to the class. Let $\alpha\subset \inte(S_i)$ be a closed geodesic and let us prove it meets $\Lambda_i^\mathrm{p}$. There is a lift $\wt\alpha$ of $\alpha$ to $\wt S$ separating $\wt S_i$. 
As $\wt S_i$ is the convex hull of the connected set $\wt\Lambda_i^0$, there exists a geodesic of $\wt\Lambda_i^0$ crossing $\wt\alpha$. By \cite[Theorem 5.8]{alepablo}, the set of tracking geodesics of periodic points is dense in $\Lambda_i$, hence there is also a tracking geodesic of a periodic orbit belonging to $\cl_i$ and crossing $\wt\alpha$. In other words, $\alpha$ intersects $\Lambda_i^\mathrm{p}$. 
This means that the complement of $\Lambda_i^\mathrm{p}$ in $S_i$ is made of essential sets plus possibly some topological open annuli that contain a boundary component of $S_i$ in their boundary. 
We can then build by hand the desired finite collection of tracking geodesics of periodic points step by step, reducing at each step the genus of the connected components of its complement, until -- after a finite number of steps -- reaching the fact that the complement of it is made of topological discs and possibly topological open annuli having a boundary component of $S_i$ as a boundary component.

Let $N$ be the maximum of all the constant given by Theorem~\ref{ThBndedDevRat} applied to all the closed tracking geodesics $\wt\gamma_1,\dots,\wt\gamma_k$.
Suppose that there exists an orbit $\wt y_0,\dots,\wt f^{n_0}(\wt y_0)$ crossing at least $2Nk+1$ different lifts of $\gamma$ (a boundary component of $S_i$). Let $\wt\gamma$ be a lift of $\gamma$ to $\wt S$ such that $\wt\gamma\subset \partial \wt S_i$.
Note that as $\gamma$ is the boundary component of $S_i$, either a left neighbourhood of $\wt\gamma$ or a right neighbourhood of $\wt\gamma$ is included in $S_i$. This implies that the orbit $\wt y_0,\dots,\wt f^{n_0}(\wt y_0)$ crosses at least $Nk$ copies of $\wt S_i$: there exists $T_1,\dots, T_{Nk}\in\G$ pairwise different such that for any $1\le j\le Nk$, the points $\wt y_0$ and $\wt f^{n_0}(\wt y_0)$ belong to different connected components of $(T_j \wt S_i)^\complement$. By construction of the geodesics $\gamma_\ell$, this means that for any $j$, there exists $1\le \ell_j\le k$ such that the points $\wt y_0$ and $\wt f^{n_0}(\wt y_0)$ belong to different connected components of $(T_j \wt\gamma_{\ell_j})^\complement$. By the pigeonhole principle, this implies that there exists $1\le \ell\le k$ and $j_1< \dots< j_N$ such that for any $1\le m\le N$, the points $\wt y_0$ and $\wt f^{n_0}(\wt y_0)$ belong to different connected components of $(T_{j_m} \wt\gamma_{\ell})^\complement$. This contradicts Theorem~\ref{ThBndedDevRat} and finishes the proof of the corollary.
\end{proof}

\small
\bibliographystyle{alpha}
\bibliography{Biblio}

\end{document}